\documentclass{article}

\usepackage{graphicx} 

\usepackage{hyperref}
\hypersetup{colorlinks = true,
	    linkcolor = black,
	    urlcolor = blue,
        citecolor = blue}
\usepackage{authblk}
\usepackage{fullpage}  
\usepackage{float}
\usepackage{pb-diagram}  		
\usepackage{esvect}
\usepackage[utf8]{inputenc} 
\usepackage[T1]{fontenc}    
\usepackage{hyperref}       
\usepackage{url}            
\usepackage{booktabs}       
\usepackage{amsfonts}       
\usepackage{nicefrac}       
\usepackage{microtype}      
\usepackage{fullpage}  
\usepackage{float}
\usepackage{pb-diagram}
\usepackage{makecell}		

\usepackage{url}

\usepackage{algorithm} 
\usepackage{algorithmic} 

\usepackage{multirow} 
\usepackage{hhline}  
\usepackage{amsmath}
\usepackage{color,xcolor}

\usepackage{graphicx}
\usepackage{subfigure}
\usepackage{amscd}
\usepackage{amssymb,mathrsfs}
\input{epsf.sty}
\usepackage{amsthm,amscd}
\usepackage{color}
\usepackage{latexsym}
\usepackage{epic}
\usepackage{appendix}
\usepackage{enumerate}
\usepackage{longtable}
\usepackage{lscape}
\usepackage{extarrows}
\usepackage{epstopdf}
\usepackage{listings}
\newlength\myindent

\newcommand{\Rmnum}[1]{\expandafter\@slowromancap\romannumeral #1@}
\newcommand{\inner}[3][]{{\langle #2,#3 \rangle_{#1}}}

\allowdisplaybreaks

\newtheorem{definition}{Definition}[section]

\newtheorem{theorem}{Theorem}[section]
\newtheorem{lemma}{Lemma}[section]

\newtheorem{assumption}{Assumption}[section]

\numberwithin{equation}{section}

\DeclareMathOperator{\grad}{\mathrm{grad}}

\DeclareMathOperator*{\argmin}{arg\,min}

\providecommand{\keywords}[1]
{
\small 
\textbf{Keywords:}
}

\title{Federated Learning on Riemannian Manifolds  with Differential Privacy}

\author[1]{Zhenwei Huang}
\author[1]{Wen Huang \thanks{Corresponding author: wen.huang@xmu.edu.cn}}
\author[2]{Pratik Jawanpuria}
\author[2]{Bamdev Mishra}
\affil[1]{ Xiamen University, Xiamen, China.\vspace{.15cm}}
\affil[2]{ Microsoft, Hyderabad, India.}

\date{}

\usepackage{amsmath}

\begin{document}

\maketitle

\begin{abstract}

In recent years, federated learning (FL) has emerged as a prominent paradigm in distributed machine learning. 
Despite the partial safeguarding of agents' information within FL systems, a malicious adversary can potentially infer sensitive information through various means. 
In this paper, we propose a generic private FL framework defined on Riemannian manifolds (PriRFed) based on the differential privacy (DP) technique. 
We analyze the privacy guarantee while establishing the convergence properties. To the best of our knowledge, this is the first federated learning framework on Riemannian manifold with a privacy guarantee and convergence results. Numerical simulations are performed on synthetic and real-world datasets to showcase the efficacy of the proposed PriRFed approach. 


\end{abstract}

\section{Introduction}  \label{sec:intr}

This paper considers designing a distributed optimization algorithm to solve the federated learning problems in the form of 
\begin{align} \label{introprob}
	\argmin_{x\in \mathcal{M}}f(x):=\sum_{i=1}^Np_i f_{i}(x):=\sum_{i=1}^N p_i \left( \sum_{j=1}^{N_i}\frac{1}{N_i}f_{i,j}(x)\right),\; p_i=\frac{N_i}{\sum_{i=1}^NN_i}
\end{align}
while guaranteeing the privacy of data, where the parameter $x$ (sometimes referred to as ``model'') is in a finite-dimensional Riemannian manifold $\mathcal{M}$, $N$ is the number of agents used to solve this problem, $N_i$ is the number of data items stored in the agent $i$, and $f_{i, j}(x)$ is an empirical risk function for the $j$-th datum in agent $i$. Such problems appear in various machine learning problems, see e.g., \cite{AZ05,CL17,KJM19} and references therein.


Conventional machine learning methods typically necessitate the uploading of all data to central servers for the purpose of learning a model of interest. However, the sensitivity of agents' data entails inherent risks and responsibilities associated with storing such data on central servers.
In light of these considerations, a privacy-preserving distributed learning approach known as Federated Learning (FL)~\cite{KMRR16, KMY17, MMRyA23} has been proposed and has garnered growing attention since its inception.
In the FL system, sensitive data is stored individually by agents. The server's sole responsibility is to aggregate locally trained parameters uploaded by active agents and subsequently broadcast the aggregated parameter to all agents. After receiving the aggregated global parameter, active agents autonomously update it and upload their revised versions to the server. A notable strength of FL lies in its ability to conduct agent-side local training without the need for sensitive data exchange between the server and agents. Additionally, FL can efficiently utilize the computational capabilities of individual agents.

Nevertheless, an adversary can still infer sensitive information from certain agents by manipulating the input datasets and probing the output of algorithm~\cite{SS15,WSZ18}. This underscores the pressing need for a FL framework that offers enhanced privacy protection for the agents.
A number of works have explored this challenging problem from different perspectives, e.g., $k$-anonymity~\cite{SS98}, secure multiparty computing~\cite{Can00}, homomorphic encryption~\cite{Gen09}, and differential privacy~\cite{DMNS06,KLNRS11,DR14}. 
Differential Privacy (DP) offers a rigorous framework for addressing data privacy by precisely quantifying the deviation in the output distribution of a mechanism when a small number of input datasets are modified. It has become a \textit{de facto} standard for safeguarding agents' privacy in machine learning, as evidenced by works such as ~\cite{CMS11,ACG16,WYX17,HMJG22,UHJM23}. Consequently, our work centers on the utilization of DP-based techniques to enhance the privacy of FL.

\subsection{Related work}

The initial approach within the FL framework is FedAvg~\cite{MMRA16} (updated in 2023~\cite{MMRyA23}), utilizing stochastic gradient descent (SGD) for agents, and it has demonstrated empirical success in practical applications. Subsequently, McMahan et al.~\cite{MRTZ17} integrated two FL algorithms, namely FedSGD and FedAvg, with DP, denoted as DP-FedSGD and DP-FedAvg, respectively. 
DP-FedSGD is a direct extension of non-federated DP-SGD~\cite{ACG16} into the federated learning framework. On the other hand, DP-FedAvg exploits DP-SGD for local training. 
A key notion in the definition of DP is the definition of adjacent datasets which often depends on the problems. 
Both methods above focus on user-level privacy, where the concept of adjacent datasets is defined as follows: two datasets $D$ and $D'$ are adjacent if $D'$ can be derived by entirely changing the data of a single agent from $D$. Many subsequent works act on user-level, e.g.,~\cite{GKN17,LZLCL22}.
Recent studies~\cite{WLD20,WLDMSZP22} have also introduced privacy-preserving FL frameworks based on record-level privacy. In this setting, two datasets $D$ and $D'$ are adjacent, if $D'$ can be derived by changing a record from $D$. In this paper, we focus on the DP based on record-level privacy.

\paragraph{In Euclidean space.} The most common way to ensure DP of an algorithm is to add noise. At each outer iteration, DP-FedSGD and DP-FedAvg~\cite{MRTZ17} add noise into the global parameter before broadcasting it to agents. Therefore, the server has the exact iterates without noise.
However, if the server is not trusted, this framework will make agents' sensitive data easily leaked by the server. 
To this end,
in the study~\cite{WLD20}, Wei et al. proposed an algorithm referred to as NbAFL. This algorithm involves the incorporation of noise in two ways: i) agents inject noise into local parameters before uploading them to the server, and ii) the server injects noise into the global parameter before broadcasting it to agents.
The algorithm proposed by~\cite{WLDMSZP22} performs one-step gradient descent for all selected agents and then these agents inject noise into the revised parameters before sending them to the server.
Furthermore, in the research conducted by~\cite{ZCLS21}, the authors also followed this idea and proposed PriFedSync, where the local training is a procedure preserving privacy based on $f$-DP~\cite{DRS20}, a different variant of DP.


The majority of FL algorithms operate in Euclidean space, e.g.,~\cite{KMRR16,LHYWZ19,LSZ20,KKM20,MJPH21,MMRyA23}. In recent years, the parameters of many practical problems are located in Riemannian manifolds, which are non-flat spaces without linearity in most cases. This prevents traditional FL frameworks from directly addressing these problems and few studies have been carried out on this issue.
Recently, Li and Ma in~\cite{LM22} introduced an FL algorithm on Riemannian manifolds, named RFedSVRG. This algorithm integrates the stochastic variance reduced gradient (SVRG) technique within the framework of Riemannian FL and creatively introduces an aggregation suitable for manifolds. 
In essence, RFedSVRG directly expands the RSVRG algorithm~\cite{ZRS17} into FL settings.

\paragraph{On manifolds.} On the other hand,
as differential privacy is extended to any measurable space~\cite{WZ08,AKRS19}, which includes a Riemannian manifold equipped with the Borel $\sigma$-algebra, many machine learning algorithms working on Riemannian manifolds are integrated with DP. Reimherr et al.~\cite{RBS21} is the first group that brought $\epsilon$-DP into Riemannian settings and proposed the Riemannian Laplace mechanism based on the distribution stemmed from~\cite{HISBR16}. Afterwards, a $K$-norm gradient mechanism for $\epsilon$-DP was proposed by Soto et al.~\cite{SBRS22} on Riemannian manifolds. These two mechanisms mentioned above were focused on privatizing Fr\'{e}chet mean and the latter demonstrated that better utility can be achieved compared to the former. Utpala et al.\cite{UVM23} developed DP and expanded its Gaussian mechanism to only one manifold, the symmetric positive definite matrices with the $\mathrm{Log}$-$\mathrm{Euclidean}$ metric, based on the distribution from~\cite{Sch16}. 
Han et al.~\cite{HMJG22} introduced a tangent Gaussian distribution on tangent spaces of a Riemannian manifold, which is a generalization of standard Gaussian distribution defined on Euclidean spaces. Based on tangent Gaussian distribution, the Gaussian mechanism commonly used in Euclidean spaces can be naturally generalized to tangent spaces. In the same work, the Riemannian SGD ensuring DP (DP-RSGD) was proposed. Subsequently, Utpala et al.~\cite{UHJM23} integrated SVRG with DP (DP-RSVRG) using the Riemannian Gaussian mechanism~\cite{HMJG22} on Riemannian manifolds. Recently, Jiang et al.~\cite{JCLDKJ23} expanded Gaussian Differential privacy~\cite{DRS20} to Riemannian manifolds based on the distribution from~\cite{PSF19}. In these works, the mechanisms in~\cite{RBS21,SBRS22,UVM23,JCLDKJ23} work on the manifold itself and are suitable for output perturbation. On the other hand, the ones in~\cite{HMJG22,UHJM23} work in tangent spaces and are suitable for gradient perturbation.

To the best of our knowledge, the problem of privatizing FL framework on Riemannian manifolds is unexplored. 
It is worth noting that in~\cite{LM22}, the authors aimed to propose an FL algorithm on Riemannian manifolds, however, DP is not considered. Whereas, to enhance privacy protection rigorously, in this work, we make an effort to privatize an FL algorithm on Riemannian manifolds, and propose a novel FL framework integrated with DP defined on Riemannian manifolds (PriRFed), which exploits the Riemannian Gaussian mechanism~\cite{HMJG22}.


\subsection{Our work and main contributions}

{In this paper, we propose a generic FL framework on Riemannian manifolds while preserving privacy based on DP technique, named PriRFed; refer to Algorithm~\ref{alg:PriRFed}. 
The motivation for developing this algorithm is that in FL, sensitive information from certain agents can still be inferred by a malicious adversary by manipulating the input datasets and probing the output of the algorithm, 
and that there exist many real-world applications, such as principal eigenvector computation over sphere manifolds and computing averaging points in symmetric positive matrices manifolds, where the parameters are located in the non-flat spaces: Riemannian manifolds.
In such cases, classical (Euclidean) FL methods are difficult to be competent effectively. For the former, we require that the active agents privately locally train the local parameters by a DP-based algorithm before sending them to the server. Since the uploaded local parameters obey DP by themselves, this avoids the necessity of the trusted server assumption.
For the latter, we adopt Riemannian optimization techniques and generalize FL setting to Riemannian manifolds. 

The main contributions in this paper are summarized as follows:
}{}

\begin{itemize}
	\item 
	We propose a generic Riemannian FL framework named PriRFed, designed to safeguard data privacy from two perspectives: locality and globality. From the viewpoint of locality, agents execute private local training to attain local DP, effectively eliminating the need for a trusted server assumption. Furthermore, it has been theoretically demonstrated that the final parameter learned by this FL framework ensures DP. This is from the perspective of globality.
	\item 
	We analyze the convergence properties of the proposed PriRFed framework, considering the privately local training procedures specified as DP-RSGD~\cite{HMJG22} and DP-RSVRG~\cite{UHJM23}. To the best of our knowledge, this marks the first attempt to address the FL setting on Riemannian manifolds while ensuring DP and simultaneously analyzing the convergence properties.
	\item 
	We conduct extensive numerical experiments on synthetic and real-world datasets to evaluate the performance of the proposed PriRFed in solving problems such as principal eigenvector computation, Fr\'{e}chet mean computation, and hyperbolic structured prediction. The reported results demonstrate the consistency between our theoretical findings and the experimental outcomes.
\end{itemize}


\subsection{Outline}
The remainder of this paper is outlined as follows: Section~\ref{sec:prel} introduces preliminaries related to Riemannian geometry, federated learning, and differential privacy.  Section~\ref{sec:PriRFed} discusses a novel Riemannian federated learning framework designed to obey differential privacy. The privacy guarantee of the proposed FL framework on Riemannian manifolds is detailed in Section~\ref{sec:privacy}. 
In Section~\ref{sec:conv}, we analyze the convergence of PriRFed with specific local training procedures~\cite{HMJG22,UHJM23}.
Section~\ref{sec:Num} presents the results of our numerical experiments and Section~\ref{sec:concl} concludes this work.

\section{Preliminaries} \label{sec:prel}
Throughout this paper, we use $\mathbb{R}$ and $\mathbb{R}^{m\times n}$ to represent the real numbers and the set of real matrices of size $m\times n$, respectively. $\|A\|_F$ denotes the $F$-norm of a matrix $A\in \mathbb{R}^{m\times n}$ and $[n]$ refers to the set $\{1,2,\dots,n\}$. For a set $\mathcal{S}$ having finite entries, $|\mathcal{S}|$ is its cardinality. We use $D_1\oplus D_2$ to denote the symmetric difference of two sets $D_1$ and $D_2$.
The logarithm with base $e$ and the exponential function with base $e$ are denoted by $\log(\cdot)$ and $\exp(\cdot)$, respectively. Additionally, we use $\mathbb{E}[\cdot]$ to represent the expectation. 

\paragraph{Riemannian geometry.} The Riemannian geometry can be found in the standard textbook, e.g.,~\cite{Boo75,Lee97}, and the notation follows~\cite{AMS08}. Denote by $\mathcal{M}$ a manifold. For any $x\in\mathcal{M}$, $\mathrm{T}_x\mathcal{M}$ is a tangent space of $\mathcal{M}$ at $x$ and the elements in $\mathrm{T}_x\mathcal{M}$ are called tangent vectors of $\mathcal{M}$ at $x$. In this paper we only consider $d$ dimensional manifold, implying that for any $x\in\mathcal{M}$, the tangent space $\mathrm{T}_x\mathcal{M}$ is $d$ dimensional linear space. The union of all tangent spaces, denoted by $\mathrm{T}\mathcal{M}$, is called the tangent bundle of $\mathcal{M}$. A Riemannian manifold $(\mathcal{M},\left<\cdot,\cdot\right>)$ is a smooth manifold when endowed with a Riemannian metric $\left<\cdot,\cdot\right>$, which is defined on the Whitney sum of two tangent bundles, namely $\mathrm{T}\mathcal{M}\oplus \mathrm{T}\mathcal{M}=\{(v,u):x\in\mathcal{M},v,u\in\mathrm{T}_x\mathcal{M}\}$. For each $x\in\mathcal{M}$ when restricted on $\mathrm{T}_x\mathcal{M}$, the metric $\left<\cdot,\cdot\right>$ induces an inner product $\left<\cdot,\cdot\right>_x=\left<\cdot,\cdot\right>\big|_{\mathrm{T}_x\mathcal{M}}$; and the norm of $v\in\mathrm{T}_x\mathcal{M}$ induced by $\left<\cdot,\cdot\right>$ is defined as $\|v\|_x=\sqrt{\left<v,v\right>_x}$. 
Given an orthonormal basis $\mathscr{B}_x=(\mathfrak{b}_{1}^x,\mathfrak{b}_2^x,\dots,\mathfrak{b}_d^x)$ for $\mathrm{T}_x\mathcal{M}$, and a tangent vector $u\in \mathrm{T}_x \mathcal{M}$, there exists unique coefficient vector $\vec{u}\in \mathbb{ R}^d$ such that $u=\mathscr{B}_x\vec{u}=\sum_{i=1}^d\vec{u}_i \mathfrak{b}_i^x$. Under this basis $\mathscr{B}_x$, the inner product can be identified by $\left<u,v\right>=\vec{u}^T\vec{v}$, where $\vec{u},\vec{v}\in\mathbb{R}^d$ are the coefficients of tangent vectors $u,v$ in the orthonormal coordinate system. For a smoothly real-valued function $f:\mathcal{M}\rightarrow\mathbb{R}$, the tangent vector $\mathrm{grad}f(x)$ denotes the Riemannian gradient of $f$ at $x$, which is defined by the unique tangent vector satisfying $\left<\mathrm{grad}f(x),v\right>_x=\mathrm{D}f(x)[v]$ for any $v\in\mathrm{T}_x\mathcal{M}$.

Let $\gamma:\mathcal{I}\rightarrow\mathcal{M}$ denote a smooth curve and $\dot{\gamma}(t)\in\mathrm{T}_{\gamma(t)}\mathcal{M}$ denote its velocity at $t$, where $\mathcal{I}\subset \mathbb{R}$ is an open interval containing $[0,1]$. The distance between $x=\gamma(0)$ and $y=\gamma(1)$ is defined by $\mathrm{dist}(x,y)=\inf_{\{\gamma:\gamma(0)=x,\gamma(1)=y\}}\int_{0}^{1}\|\dot{\gamma}(t)\|_{\gamma(t)}\mathrm{d}t$. A smooth curve $\gamma:[0,1]\rightarrow\mathcal{M}$ is called the geodesic if it locally minimizes the distance between $x=\gamma(0)$ and $y=\gamma(1)$. For any $v\in\mathrm{T}_x\mathcal{M}$, the exponential map is defined as $\mathrm{Exp}_{x}(v)=\gamma(1)$ with $\gamma(0)=x$ and $\dot{\gamma}(0)=v$. For two points $x,y\in\mathcal{M}$, if there exists a unique geodesic connecting them, then the exponential map has a smooth inverse and the Riemannian distance is given by $\mathrm{dist}(x,y)=\|\mathrm{Exp}_{x}^{-1}(y)\|_x=\|\mathrm{Exp}_y^{-1}(x)\|_y$. Given two points $x,y\in\mathcal{M}$, 
the parallel translation\footnote{Strictly speaking, the notion of parallel translation depends on the choice of a curve (not necessarily geodesic) connecting $x$ to $y$ and the corresponding parallel vector field. In this paper, we particularly specify the curve as minimizing geodesic, which is no problem for a complete Riemannian manifold. One may refer to~\cite{AMS08,Bou23} for more details.} from $x$ to $y$ transports a tangent vector on $\mathrm{T}_x\mathcal{M}$ into $\mathrm{T}_x \mathcal{M}$, denoted by $\Gamma_x^{y}:\mathrm{T}_x\mathcal{M} \rightarrow \mathrm{T}_y\mathcal{M}$, along the unique minimizing geodesic connecting $x$ to $y$ while preserving the inner product along a geodesic in the sense that for any $u,v\in\mathrm{T}_x\mathcal{M}$, it holds that $\left<\Gamma_x^yu,\Gamma_x^yv\right>_y=\left<u,v\right>_x$. Moreover, parallel translation is a linear map.
A neighborhood $\mathcal{W}\subset\mathcal{M}$ is totally normal if, for any of its two points, the exponential map is invertible. We assume throughout this paper that all the iterates are in a totally normal neighborhood.

\paragraph{Function classes.} Here we review some function classes satisfying desired conditions for analysis.

\begin{definition}[Geodesic Lipschitz continuity] \label{def:geodLip}
A function $f:\mathcal{M}\rightarrow\mathbb{R}$ is called geodesically $L_f$-Lipschitz continuous if for any $x,y\in\mathcal{M}$, there exists $L_f\ge0$ such that
\[
	|f(x)-f(y)| \le L_f\mathrm{dist}(x,y).
\]
\end{definition}
Under the assumption of continuous gradient, a function $f:\mathcal{M}\rightarrow \mathbb{R}$ is geodesically $L_f$-Lipschitz continuous if and only if~\cite{Bou23}
\[
	\|\mathrm{grad}f(x)\|_x \le L_f, \hbox{ for all }x\in\mathcal{M}.
\]

\begin{definition}[Geodesic smoothness] \label{def:geodsmooth}
	A differentiable function $f:\mathcal{M}\rightarrow\mathbb{R}$ is call geodesically $L_g$-smooth if its gradient is $L_g$-Lipschitz continuous, that is,
	\[
		\|\mathrm{grad}f(x) - \Gamma_y^x(\mathrm{grad}f(y))\|_x \le L_g \mathrm{dist}(x,y).
	\]
\end{definition}
It can be shown that if $f$ is geodesically $L_g$-smooth, then we have
	\[
		f(y) \le f(x) + \left<\mathrm{grad}f(x), \mathrm{Exp}_{x}^{-1}(y)\right>_y + \frac{L_g}{2}\mathrm{dist}^2(x,y)
	\]
	for all $x,y\in\mathcal{M}$~\cite{ZS16}.

\begin{definition}[Geodesic convexity~\cite{ZS16}] \label{def:geodesic convex}
	A function $\mathcal{M} \rightarrow \mathbb{R}$ is called geodesically convex if for any $x,y\in \mathcal{M}$, a geodesic $\gamma$ satisfying $\gamma(0)=x$ and $\gamma(1)=y$, and $t\in[0,1]$, it holds that 
	\[
		f(\gamma(t))\le (1-t)f(x) + tf(y).
	\]
\end{definition}
It has been shown that for any $x,y\in \mathcal{M}$, $f: \mathcal{M} \rightarrow \mathbb{R}$ is geodesically convex if and only if
\[
	f(y) \ge f(x) + \left<{g}_x, \mathrm{Exp}_x^{-1}(y)\right>_x,
\]
where ${g}_x$ is a Riemannian subgradient of $f$ at $x$ or the Riemannian gradient if $f$ is differentiable~\cite{ZS16}.

\paragraph{Federated learning.}
This paper considers an FL system where one server and $N$ agents collaboratively fit a model of interest working on a Riemannian manifold. Let $D=\{z_1,z_2,\dots,z_n\}$ be a dataset, where data point $z_i$ is sampled from an input data space $\mathcal{Z}$, i.e., $D\in \mathcal{Z}^n$. In the FL system, each agent $i$ has a local dataset $D_i=\{z_{i,1},z_{i,2},\dots,z_{i,N_i}\}\in \mathcal{Z}^{N_i}$ of size $N_i=|D_i|$. Formally, the FL system aims to learn a parameter $x\in \mathcal{M}$ by collaboratively optimizing the following problem: 
\begin{equation} \label{prob}
	\begin{aligned}
		\min_{x\in \mathcal{M}}f(x)&:=\sum_{i=1}^Np_if(x,D_i)  
		:=\sum_{i=1}^Np_i\left(\frac{1}{N_i}\sum_{j=1}^{N_i}f(x,z_{i,j})\right),
	\end{aligned}
\end{equation}
where $\mathcal{M}$ is a $d$-dimensional Riemannian manifold under consideration, $p_i=N_i/(\sum_{i=1}^NN_i)$, and $\frac{1}{N_i}\sum_{j=1}^{N_i}f(x,z_{i,j})$ is the local objective on the agent $i$ holding the local dataset $D_i$. For convenience, in what follows, we use $f_i(x)$ and $f_{i,j}(x)$ to denote $f(x,D_i)$ and $f(x, z_{i,j})$ respectively. 
The classic FedAvg~\cite{MMRyA23} (i.e., $\mathcal{M}=\mathbb{R}^n$) mainly follows the steps below:

\begin{itemize}
	\item Initialization: the server randomly samples $s$ agents, denoted by $\mathcal{S}_t$, and broadcasts an initial guess $x^{(0)}$ to these agents.
	\item Local training: at the $t$-th iteration, each sampled agent $i \in \mathcal{S}_t$ locally performs stochastic gradient descent $K$ times and sends the updated $x_i^{(t+1)}$ to the server. 
	\item Server aggregation: the server aggregates the received $x_i^{(t)}$ by weighted average, 
\begin{align} \label{prel:weight-ave}
x^{(t+1)}=\sum_{i\in \mathcal{S}_t}\frac{N_i}{\sum_{i\in \mathcal{S}_t}N_i}x_i^{(t+1)},
\end{align}
	and then broadcasts the global parameter $x^{(t+1)}$ to all agents.
\end{itemize}
Nevertheless, it is not suitable for Problem~\eqref{prob} because of the difficulty of aggregating points in a manifold, a non-linear space in most cases. 
An efficient aggregation on Riemannian manifolds will be discussed in Section~\ref{sec:PriRFed}.

\paragraph{Differential privacy.} 
Differential privacy has been extended to encompass any measurable space~\cite{WZ08,AKRS19}, including the incorporation of a Riemannian manifold equipped with the Borel $\sigma$-algebra. Consequently, defining differential privacy over a Riemannian manifold is a natural extension. Two datasets, $D$ and $D'$, belonging to the space $\mathcal{Z}^n$, are adjacent, denoted by $D\sim D'$, if they differ by at most one point. Therefore, $D \sim D'$ if and only if $|D\oplus D'|\le 2$. With the definition of adjacency of two datasets, differential privacy is defined in Definition~\ref{def:DP}. 

\begin{definition}[Differential privacy~\cite{DR14}] \label{def:DP}
Let $\mathcal{A}$ be a manifold-valued randomized mechanism, namely, $\mathcal{A}:\mathcal{Z}^n \rightarrow \mathcal{M}$. We say that $\mathcal{A}$ is $(\epsilon,\delta)$-differentially private ($(\epsilon,\delta)$-DP for short) if for any two adjacent datasets $D,D'$ and for all measurable sets $\mathcal{S}\subseteq\mathcal{M}$, it holds that
\[
	\mathrm{P}\{\mathcal{A}(D)\in \mathcal{S}\}\le \exp(\epsilon)\mathrm{P}\{\mathcal{A}(D')\in \mathcal{S}\}+\delta.
\]
The parameter $\epsilon$ called privacy budget is the distinguishable bound of all outputs of $\mathcal{A}$ on adjacent datasets $D$ and $D'$ in the database $\mathcal{Z}^n$, and $\delta$ means the probability of failure of differential privacy.
\end{definition}


Theorem~\ref{th:post-processing} states that post-processing data under differential privacy protection still guarantees at least the same level of differential privacy.

\begin{theorem}[{Post-processing} {\cite[Proposition 2.1]{DR14}}]   \label{th:post-processing}
	Suppose that $\mathcal{A}:\mathcal{Z}^n \rightarrow \mathcal{M}$ is a randomized algorithm that is $(\epsilon,\delta)$-DP. Let $P:\mathcal{M}\rightarrow\mathcal{M}'$ be an arbitrary mapping. Then $P\circ \mathcal{A}:\mathcal{Z}^n \rightarrow\mathcal{M}'$ is $(\epsilon,\delta)$-DP.
\end{theorem}


The concept of composability stands as a fundamental property within the realm of differential privacy. This property empowers the modular design of intricate private mechanisms from simpler components. Theorems~\ref{th:composition} and~\ref{th:adv composition} formally establish that the composition of several simpler private mechanisms remains differentially private, with the added assurance that its privacy guarantee does not significantly degrade.

\begin{theorem}[Sequential composition theorem {\cite[Theorem 3.16]{DR14}}]\label{th:composition}
	Suppose that $\mathcal{A}_i:\mathcal{Z}^n \rightarrow\mathcal{M}_i$ is an $(\epsilon_i,\delta_i)$-DP algorithm for $i\in[k]$. Then if $\mathcal{A}_{[k]}:\mathcal{Z}^n \rightarrow \prod_{i=1}^k\mathcal{M}_i$ is defined to be $\mathcal{A}_{[k]}:=(\mathcal{A}_1,\mathcal{A}_2,\dots,\mathcal{A}_k)$, then $\mathcal{A}_{[k]}$ is $(\sum_{i=1}^k\epsilon_i,\sum_{i=1}^k\delta_i)$-DP.
\end{theorem}

It is noteworthy that from Theorem~\ref{th:composition}, considering a series of mechanisms $\{\mathcal{A}_t\}_{t=1}^k$ which all obey $(\epsilon,\delta)$-DP, the sequential composition $\mathcal{A}_{[k]}$, defined by Theorem~\ref{th:composition}, is $(k\epsilon,k\delta)$-DP. Therefore, the privacy budget of $\mathcal{A}_{[k]}$ increases linearly with respect to the number $k$ of compositions. This claim also holds for adaptive composition of $\{\mathcal{A}\}_{t=1}^{k}$, defined by $\mathcal{A}_{[k]}=\mathcal{A}_{k}\circ \mathcal{A}_{k-1}\circ \cdots \circ \mathcal{A}_{1}$, where $\mathcal{A}_{j}$ takes the dataset and the outputs of $\mathcal{A}_{j-1},\dots,\mathcal{A}_{1}$ as input; see, e.g.,~\cite{GLW21}~\cite[Chapter~7.2]{Vad17}. 
Whereas, if it is allowed to relax the parameter $\delta$, then the privacy budget under adaptive composition only increases sublinearly, as seen in Theorem~\ref{th:adv composition}.

\begin{theorem}[Advanced composition theorem {~\cite{DRV10,KOV17,WRRW23}}] \label{th:adv composition} 
Suppose that $\mathcal{A}_i: \mathcal{Z}^n \rightarrow \mathcal{M}$ is an $(\epsilon,\delta)$-DP algorithm for $i\in[k]$. Then the adaptive composition $\mathcal{A}_{[k]}$ of $\mathcal{A}_{1},\dots,\mathcal{A}_k$ is $(\epsilon',\delta'+k\delta)$-DP with $\epsilon'=\sqrt{2k\log(1/\delta')}\epsilon+k\epsilon(\mathrm{exp}(\epsilon)-1)$ and any $\delta'>0$.
\end{theorem}
Kairouz et al.~\cite{KOV17} have presented an optimal composition bound for differential privacy, which provides an exact choice of the best privacy parameters. Nevertheless, the resulting optimal composition bound is comparably complex, see details in~\cite{MV15}. 
Hence, as a pragmatic alternative, Theorem~\ref{th:composition} and~\ref{th:adv composition} are often preferred in practice due to their convenience. 



Intuitively, according to Definition~\ref{def:DP}, a smaller $\epsilon$ corresponds to a more private mechanism. In the Euclidean space $\mathbb{R}^d$, a commonly employed mechanism that satisfies $(\epsilon,\delta)$-DP is the Gaussian mechanism, which introduces noise following a Gaussian distribution~\cite{DR14} to its output. Notably, the Gaussian mechanism has been extended from Euclidean spaces to Riemannian manifolds~\cite{HMJG22}, allowing the addition of noise following an isotropic Gaussian distribution (refer to Definition \ref{def:tangent space Gaussian distribution}) with respect to the Riemannian metric in the tangent space of the manifold. This extension is advantageous as it enables the measurement of sensitivity with respect to the Riemannian metric. Consequently, several Riemannian optimization algorithms ensuring differential privacy have been proposed based on the tangent space Gaussian distribution, as seen in works such as~\cite{HMJG22,UHJM23}. 

\begin{definition}[Tangent space Gaussian distribution~\cite{HMJG22,UHJM23}] \label{def:tangent space Gaussian distribution}
	For any point $x$ in $\mathcal{M}$, a tangent vector $\xi \in \mathrm{T}_{x}\mathcal{M}$ obeys a tangent space Gaussian distribution at $x$ with mean $\mu\in \mathrm{T}_{x}\mathcal{M}$ and standard deviation $\sigma>0$ if its density is given by $p_x(\xi)={C^{-1}_{x,\sigma}}\exp(-\frac{\left\|\xi-\mu\right\|_x^2}{2\sigma^2})$, where $C_{x,\sigma}$ is the normalizing constant. We use $\xi\sim \mathcal{N}_x(\mu,\sigma^2)$ to denote that $\xi$ follows the tangent space Gaussian distribution with $\mu,\sigma$ at $x$.
\end{definition}

Consider the vector $\vec{\varepsilon}\in \mathbb{R}^d$ obtained by vectorizing the tangent vector $\varepsilon\in \mathrm{T}_{x} \mathcal{M}\simeq \mathbb{R}^d$ in the orthonormal coordinate system $\mathscr{B}_x=(\mathfrak{b}_{1}^x,\mathfrak{b}_2^x,\dots,\mathfrak{b}_d^x)$. The density function can then be expressed as $p_x(\varepsilon)=C^{-1}_{x,\sigma}\exp\left(-\frac{1}{2\sigma^2}\|\vec{\varepsilon}-\vec{\mu}\|_2^2\right)$. This formulation corresponds to a standard Gaussian distribution on $\mathbb{R}^d$ with a mean of $\vec{\mu}$ and a covariance of $\sigma^2I_d$. The normalizing constant is given by $C^{-1}_{x,\sigma}=\sqrt{(2\pi)^d\sigma^2}$. Consequently, if $\varepsilon\in \mathcal{N}_x(0,\sigma^2)$, it follows that $\mathbb{E}[\|\varepsilon\|_{x}^2]=\mathbb{E}[\|\vec{\varepsilon}\|_2^2] = d\sigma^2$.

By~\cite[Proposition 2]{HMJG22}, for a query $H:\mathcal{Z}^n \rightarrow \mathrm{T}_{x}\mathcal{M}$, a associated Gaussian mechanism $\mathcal{H}(D)=H(D)+\xi$ with $\xi \sim \mathcal{N}_x(0,\sigma^2)$ is $(\epsilon,\delta)$-DP, if the standard deviation $\sigma$ satisfies $\sigma^2\ge 2\log(1.25/\delta)\Delta_H^2/\epsilon^2$, where $\Delta_H:=\sup_{\{D,D'\subset \mathcal{Z}^n:D\sim D'\}}\|H(D)-H(D')\|_x$ is the global sensitivity of $H$ with respect to the Riemannian metric. Hence, a smaller $\epsilon$ signifies stronger privacy protection, leading to an increase in the variance of the added noise. From an algorithmic standpoint, incorporating a substantial amount of noise can adversely impact the convergence performance of the algorithm. That is to say, a trade-off exists between ensuring privacy and optimizing convergence performance.

\section{Private federated learning on Riemannian manifolds}
\label{sec:PriRFed}


The proposed Riemannian federated learning framework with differential privacy is stated in Algorithm~\ref{alg:PriRFed}, which partly draws on the Euclidean federated learning framework with differential privacy in ~\cite{ZCLS21}. 


Algorithm~\ref{alg:PriRFed} can be divided into two stages: local training and server aggregation. In the first place, the initial global guess $x^{(0)}$ is sent to active agents subsampled without replacement by the server. Subsequently, the two stages are executed alternately $T$ times as follows.
\begin{itemize}
	\item In the first stage, all the selected agents parallelly perform private local training procedures. 
 The resulting trained parameters $\tilde{x}_i^{(t+1)}$ are then uploaded to the server. See Lines~\ref{alg:PriRFed:step4} to~\ref{alg:PriRFed:step8} in Algorithm~\ref{alg:PriRFed}. 
	\item In the second stage, the server aggregates the parameters $\tilde{x}_i^{(t+1)}$ received from the agents to generate a global parameter $x^{(t+1)}$, and then re-subsamples $s_{t+1}$ agents and broadcasts the global parameter to the selected agents. See Lines~\ref{alg:PriRFed:step10},~\ref{alg:PriRFed:step2} and~\ref{alg:PriRFed:step3} in Algorithm~\ref{alg:PriRFed}. 	
\end{itemize}
Finally, Algorithm~\ref{alg:PriRFed} returns a global parameter $\tilde{x}$ in two ways: one accepts the final outer iterates $x^{(T)}$ (i.e., \textbf{Option 1}), and the other one is uniformly sampled from all outer iterates at random (i.e., \textbf{Option 2}).

\begin{algorithm}[ht]
  \caption{Private Riemannian Federated Learning: PriRFed}
  \begin{algorithmic}[1] \label{alg:PriRFed}
  \REQUIRE Initial global parameter $x^{(0)}$, datasets $\{{D}_1,D_2,\dots,D_N\}$, the number of aggregations $T$, the number of local iterations $K$, stepsizes $\{\alpha_t\}_{t=0}^{T-1}$;
  \ENSURE $\tilde{x}$.
  \FOR{$t=0,1,\dots, T-1$}
  \STATE Sample $\mathcal{S}_t$ from $[N]$ without replacement of size $|\mathcal{S}_t|=s_t$; \label{alg:PriRFed:step2}
    \STATE \label{alg:PriRFed:step3} Broadcast the global parameter $x^{(t)}$ to selected agents.
  \STATE  \textbf{Local training process: }   \label{alg:PriRFed:step4}
  \WHILE {$i \in \mathcal{S}_t$ \textbf{in parallel}} \label{alg:PriRFed:step5}
  \STATE \label{alg:PriRFed:step6} $\tilde{x}_i^{(t+1)} \gets \mathrm{PrivateLocalTraining}(x^{(t)}, {D}_i, K, \sigma_i,\alpha_t)$;

  \STATE \label{alg:PriRFed:step7} Send $\tilde{x}^{(t+1)}_i$ to the server;
  \ENDWHILE \label{alg:PriRFed:step8}
  \STATE \label{alg:PriRFed:step9} \textbf{Sever aggregating process: }
  \STATE \label{alg:PriRFed:step10} Aggregate the received local parameters to produce a global parameter $x^{(t+1)}$ by~\eqref{Prelim:tangentMean};
  \ENDFOR \label{alg:PriRFed:step11}
  \STATE \textbf{Option 1:} $\tilde{x} = x^{(T)}$; \label{alg:PriRFed:step12}
  \STATE \textbf{Option 2:} $\tilde{x}$ is uniformly sampled from $\{x^{(t)}\}_{t=1}^{T}$ at random; \label{alg:PriRFed:step13}
 \end{algorithmic}
\end{algorithm}

One can use any local training procedure that satisfies $(\epsilon,\delta)$-DP in Line~\ref{alg:PriRFed:step6}. In this paper, the DP-RSGD~\cite{HMJG22} and DP-RSVRG~\cite{UHJM23} are used for convergence analysis, as shown in Section~\ref{sec:conv}.


Due to the nonlinearity of Riemannian manifolds in most cases, the weighted average~\eqref{prel:weight-ave} is not competent for the aggregation in the server.
Fr\'{e}chet mean, as a generalization of weighted average
\begin{align} \label{PriRFed:Freche mean}	
	x^{(t+1)} \gets \argmin_{x\in \mathcal{M}} \frac{1}{2} \sum_{i\in \mathcal{S}_t}\frac{N_i}{\sum_{i \in \mathcal{S}_t}N_i}\mathrm{dist}^2(x,x_i^{(t+1)}),
\end{align}
is a natural way to aggregate the local parameters. However, a Fr\'{e}chet mean of points on a Riemannian manifold generally does not have a closed-form solution. Therefore its computation requires extra work, i.e., solving optimization problems. One remedy is only to solve Problem~\eqref{PriRFed:Freche mean} approximately. One approach in~\cite{LM22} is to use one step of Riemannian steepest gradient method. Specifically, the Riemannian gradient of the  objective in~\eqref{PriRFed:Freche mean} is
\[
	\mathrm{grad}\left(\frac{1}{2}\sum_{i\in \mathcal{S}_t}\frac{N_i}{\sum_{i\in \mathcal{S}_t}}\mathrm{dist}(x,x_i^{(t+1)})\right) = -\sum_{i\in \mathcal{S}_t}\frac{N_i}{\sum_{i\in \mathcal{S}_t}}\mathrm{Exp}_x^{-1}(x_i^{(t+1)}).
\]
One step of the Riemannian steepest gradient descent with initial guess $x^{(t)}$ yields
\begin{align} \label{Prelim:tangentMean}
	x^+ \gets \mathrm{Exp}_{x^{(t)}}\left( \sum_{i\in \mathcal{S}_t}\frac{N_i}{\sum_{i \in \mathcal{S}_{t}}N_i}\mathrm{Exp}_{x^{(t)}}^{-1}(x_i^{(t+1)}) \right).
\end{align}
which is used to approximate the solution of Problem~\eqref{PriRFed:Freche mean}. The approximation $x^+$ to the solution of~\eqref{PriRFed:Freche mean} is called tangent mean.


\section{Privacy guarantee analysis} \label{sec:privacy}

In Algorithm~\ref{alg:PriRFed}, the server randomly selects some agent(s), at each round of outer iteration, to locally update the parameter(s). This processing is called subsampling and is formally defined in Definition~\ref{def:subsample}. 
\begin{definition}[Subsample]\label{def:subsample}
	Given a dataset $D$ of $n$ points, namely $D=\{D_1,D_2,\dots,D_n\}$ (where $D_i$ may be a dataset), an operator $\mathcal{S}_\rho^{\mathrm{wor}}$, called subsample, randomly chooses a sample from the uniform distribution over all subsets of $D$ of size $m$ without replacement. The ratio $\rho:= m/n$ is called the sampling rate of the operator $\mathcal{S}_\rho^{\mathrm{wor}}$. Formally, letting $\{i_1,i_2,\dots,i_m\}$ be the sampled indices, $\mathcal{S}_{\rho}^{\mathrm{wor}}$ is defined by 
\begin{align*}
  \mathcal{S}_\rho^{\mathrm{wor}}:& \mathcal{Z}^{N_1} \times \mathcal{Z}^{N_2} \times \ldots \times \mathcal{Z}^{N_n} \rightarrow \mathcal{Z}^{N_{i_1}} \times \mathcal{Z}^{N_{i_2}} \times \ldots \times \mathcal{Z}^{N_{i_m}}:
  (D_1,D_2,\dots,D_n) \mapsto (D_{i_1},D_{i_2},\dots,D_{i_m}).
\end{align*} 
\end{definition}

Compositing a mechanism obeying DP with a subsampling operator without replacement, the resulting mechanism also satisfies DP and the privacy of the mechanism is amplified to some extent, as stated in Lemma~\ref{lem:Subsampling lemma}.
In machine learning, we often work on mini-batches of a dataset. From the viewpoint of subsampling, this can amplify the privacy of the machine learning algorithm based on DP. Another setting is federated learning framework, where the server randomly samples some agents to perform local updates at each outer iteration.
These were discussed in the papers~\cite{KLNRS11,BNS13, BBG18}. A proof of Lemma~\ref{lem:Subsampling lemma} can be directly found in the note~\cite{Ull17}.

\begin{lemma}[Subsampling lemma] \label{lem:Subsampling lemma}
Let $\mathcal{S}_{\rho}^{\mathrm{wor}}:\mathcal{Z}^{N_1} \times \mathcal{Z}^{N_2} \times \ldots \times \mathcal{Z}^{N_n} \rightarrow \mathcal{Z}^{N_{i_1}}\times \mathcal{Z}^{N_{i_2}}\times \dots \times \mathcal{Z}^{N_{i_m}}$ be a subsample operator (defined by Definition~\ref{def:subsample}) with sampling rate $\rho=m/n$ and sampled indices $\{i_1,i_2,\dots,i_m\}$. 
 Suppose that $\mathcal{A}:\mathcal{Z}^{N_{i_1}}\times \mathcal{Z}^{N_{i_2}}\times \dots \times \mathcal{Z}^{N_{i_m}} \rightarrow \mathcal{M}$ is a mechanism obeying $(\epsilon,\delta)$-DP. Then $\mathcal{A}':=\mathcal{A}\circ \mathcal{S}_{\rho}^{\mathrm{wor}}:\mathcal{Z}^{N_1} \times \mathcal{Z}^{N_2} \times \ldots \times \mathcal{Z}^{N_n} \rightarrow \mathcal{M}$ provides $(\epsilon',\delta')$-DP, where $\epsilon'=\log(1+\rho (\mathrm{exp}(\epsilon)-1))$ and $\delta'\le\rho \delta$.
\end{lemma}
A stable transformation is a map that for any two datasets, the difference of the two mapped results is not too different from the difference between the original datasets. The formal definition is given by Definition~\ref{def:Stable Transformation}. 
\begin{definition}[$c$-stable transformation~{\cite[Definition~2]{McS09}}] \label{def:Stable Transformation}
Let $\bar{\mathcal{Z}}:=\mathcal{Z}^{N_1} \times \mathcal{Z}^{N_2} \times \ldots \times \mathcal{Z}^{N_n}$.	A transformation $T: \bar{\mathcal{Z}} \rightarrow \bar{\mathcal{Z}}$ is said $c$-stable if for any two data sets $D_1,D_2\in \bar{\mathcal{Z}}$, the following holds
	\[
		|T(D_1)\oplus T(D_2)|\le c|D_1\oplus D_2|.
	\]
 where $c$ is a constant.
\end{definition}
Theorem~\ref{thm:Pre-processing} (pre-processing) shows that the composition of a mechanism obeying DP and a $c$-stable transformation also obeys DP and the privacy budget of the composition does not increase significantly.
\begin{theorem}[Pre-processing]\label{thm:Pre-processing}
Let $\bar{\mathcal{Z}}:=\mathcal{Z}^{N_1} \times \mathcal{Z}^{N_2} \times \ldots \times \mathcal{Z}^{N_n}$.	Suppose that $\mathcal{A}: \bar{\mathcal{Z}} \rightarrow \mathcal{M}$ be $(\epsilon,\delta)$-DP, and $T: \bar{\mathcal{Z}} \rightarrow \bar{\mathcal{Z}}$ be an arbitrary $c$-stable transformation. Then the composition $\mathcal{A}\circ T$ provides $(c\epsilon,	c \mathrm{exp}((c-1)\epsilon)\delta)$-DP.
\end{theorem}

\begin{proof}
	Using the definition of $c$-stability, for any two adjacent data sets $D$ and $D'$, we have $|T(D)\oplus T(D')|\le c|D\oplus D'|\le2c$, which is indicative that $T(D)$ and $T(D')$ differ in at most $c$ data points. Therefore, by the group privacy~\cite[pp.~20]{DR14}, we have, for any measurable set $\mathcal{S}\in \mathcal{M}$, 
	\begin{align*}
		\mathrm{P}[\mathcal{A}(T(D))\in \mathcal{S}] \le \mathrm{exp}(c\epsilon)\mathrm{P}[\mathcal{A}(T(D'))\in \mathcal{S}] + c \mathrm{exp}((c-1)\epsilon)\delta, 
	\end{align*}
	which is exactly what we point out.
\end{proof}

In particular, if $T$ is $1$-stable, then the composition is still $(\epsilon,\delta)$-DP by Theorem~\ref{thm:Pre-processing}. 
Now, we are ready to prove the privacy guarantee of Algorithm~\ref{alg:PriRFed} provided that each local training procedure obeys differential privacy, as stated in Theorem~\ref{th:DP}. 
\begin{theorem}[Privacy guarantee of Algorhtm~\ref{alg:PriRFed}]
\label{th:DP}
	If the privately local training satisfies $(\epsilon, \delta)$-DP for agent $i\in[N]$, then Algorithm~\ref{alg:PriRFed} obeys $(\epsilon',\delta')$-DP with 
	$
		\epsilon' = \min(T\tilde{\epsilon},\sqrt{2T\ln({1}/{\hat{\delta}})}\tilde{\epsilon} + T\tilde{\epsilon}(\mathrm{exp}(\tilde{\epsilon})-1))
	$ and $\delta'=\hat{\delta}+T\tilde{\delta}$ 
	for any $\hat{\delta}\in(0,1)$, where  $		\tilde{\epsilon} = \log(1+\rho(\mathrm{exp}(s\epsilon)-1))$, $\tilde{\delta} = \rho s\delta$, $\mathcal{S}_t\subseteq[N]$ with $s_t=s\le N$, and $\rho=s/N$ is the sampling rate.

\end{theorem}

\begin{proof}
At the $t$-th round of outer iteration, let $\mathcal{I}_s=\{i_k\}_{k=1}^s\subset[N]$ be any index of size $s$ without replacement and 
let $\mathrm{PLT}_i^{(t)}$ denote the privately local training of agent $i$, and let $\tilde{x}_i^{(t+1)}$ denote its output. 
Consider a transformation $\mathcal{H}_i:\prod_{i\in \mathcal{I}_s}\mathcal{Z}^{N_i} \rightarrow \prod_{i\in \mathcal{I}_s}\mathcal{Z}^{N_i}:(D_{i_1},D_{i_2},\dots,D_{i_k}) \mapsto (0,\dots,D_{i},\dots,0)$, which is evidently $1$-stable since for any two $\tilde{D},\tilde{D}'\in \prod_{i\in \mathcal{I}_s}\mathcal{Z}^{N_i}$, $|\mathcal{H}_i(\tilde{D})\oplus \mathcal{H}_i(\tilde{D}')|=|\tilde{D}_i \oplus \tilde{D}'_i|\le | \tilde{D} \oplus \tilde{D}' |$. Due to that $\mathrm{PLT}_i^{(t)}$ obeys $(\epsilon,\delta)$-DP, it follows that $\mathcal{A}_i^{(t)}:=\mathrm{PLT}_i^{(t)}\circ \mathcal{H}_i:\prod_{i\in \mathcal{I}_s}\mathcal{Z}^{N_i} \rightarrow \mathcal{M}$ is also $(\epsilon,\delta)$-DP from Theorem~\ref{thm:Pre-processing}. 
Subsequently, consider the sequential composition of $\{\mathcal{A}_{i_k}^{(t)}\}_{k=1}^s$, defined by 
\begin{align*}
	\tilde{\mathcal{A}}^{(t)}:\prod_{i\in\mathcal{I}_s}\mathcal{Z}^N_i \rightarrow \mathcal{M}^s : \tilde{D} \mapsto (\mathcal{A}_{i_1}^{(t)}(\tilde{D}), \mathcal{A}_{i_2}^{(t)}(\tilde{D}), \mathcal{A}_{i_s}^{(t)}(\tilde{D})).
\end{align*}
According to the Sequential Composition Theorem~\ref{th:composition}, $\tilde{\mathcal{A}}^{(t)}$ provides $(s\epsilon,s\delta)$-DP. 

Now, we consider the subsampling $\mathcal{S}_{\rho,t}^{\mathrm{wor}}$ corresponding to Line~\ref{alg:PriRFed:step2} in Algorithm~\ref{alg:PriRFed}, with $\mathcal{S}_t=\{i_1^{(t)},i_2^{(t)},\dots,i_s^{(t)}\}$ and $\rho=s/N$ being the sampled indices and the sampling rate, respectively. We define $\tilde{\mathcal{A}}_{\mathcal{S}_t}^{(t)}:\prod_{i\in[N]}\mathcal{Z}^{N_i} \rightarrow \mathcal{M}$ as a mechanism such that $\tilde{\mathcal{A}}_{\mathcal{S}_t}^{(t)}(D_1,D_2,\dots,D_N)=\tilde{\mathcal{A}}^{(t)}(D_{i_1},D_{i_2},\dots,D_{i_s})$, i.e., $\tilde{\mathcal{A}}^{(t)}_{\mathcal{S}_t}:=\tilde{\mathcal{A}}^{(t)}\circ \mathcal{S}_{\rho,t}^{\mathrm{wor}}$. 
From Subsampling Lemma~\ref{lem:Subsampling lemma}, the composited mechanism $\tilde{\mathcal{A}}^{(t)}_{\mathcal{S}_t}$ obeys $(\tilde{\epsilon},\tilde{\delta})$-DP, where $\tilde{\epsilon}=\log(1+\rho(\mathrm{exp}(s\epsilon)-1))$, and $\tilde{\delta}=\rho	s\delta$. 

Therefore, for one round of outer iteration, the mechanism $\mathcal{R}_t$ is the composition of $\mathcal{A}_{\mathcal{S}_t}^{(t)}$ and the tangent mean $\mathcal{T}_x:\mathcal{M}^s\rightarrow \mathcal{M}:(x_{1_1},x_{i_2},\dots,x_{i_s})\mapsto \mathcal{T}_x(x_{i_1},x_{i_2},\dots,x_{i_s})=\mathrm{Exp}_{x}\left(\sum_{i\in \mathcal{S}_t}\frac{N_i}{\sum_{i\in \mathcal{S}_t}N_i}\mathrm{Exp}_{x}^{-1}(x_i)\right)$ for a given $x\in \mathcal{M}$, i.e., $\mathcal{R}_t=\mathcal{T}_{x^{(t)}}\circ {\mathcal{A}}_{\mathcal{S}_t}^{(t)}$. Then by the Post-processing Theorem~\ref{th:post-processing}, $\mathcal{R}_t$ is also $(\tilde{\epsilon},\tilde{\delta})$-DP. 

By now, we have that ${\mathcal{R}}_1$, ${\mathcal{R}}_2$, $\dots$, ${\mathcal{R}}_T$ are randomized mechanisms such that ${\mathcal{R}}_t$ is $(\tilde{\epsilon},\tilde{\delta})$-DP. After $T$ iterations, the whole mechanism $\mathcal{R}$ (i.e., Algorithm~\ref{alg:PriRFed}) is the composed sequence of these mechanisms, i.e., $\mathcal{R}=\mathcal{R}_{T}\circ \mathcal{R}_{T-1}\circ \dots \circ \mathcal{R}_1$. 
On the one hand, by the adaptive composability of Sequential Composition Theorem~\ref{th:composition}, $\mathcal{R}$ obeys $(\epsilon',\delta')$-DP with $\epsilon'=T\tilde{\epsilon}$ and $\delta'=T\tilde{\delta}$. 
On the other hand, by the adaptive composability of Advanced Composition Theorem~\ref{th:adv composition},  $\mathcal{R}$ satisfies $(\epsilon', \delta')$-DP, where $\delta'=\hat{\delta} + T\tilde{\delta}$ for any $\hat{\delta}>0$, and 
$\epsilon' = \sqrt{2T\ln({1}/{\hat{\delta}})}\tilde{\epsilon} + T\tilde{\epsilon}(\mathrm{exp}(\tilde{\epsilon})-1). $

\end{proof}

Table~\ref{table} reports the privacy parameter pairs $(\epsilon',\delta')$ in Theorem~\ref{th:DP}
using multiple parameters, i.e., the number of agents $N$, the number of sampled agents $s$, the sampling rate $\rho$, the number of iterations $T$. 
\begin{table}[ht]
\centering
\caption{The privacy parameters $(\epsilon',\delta')$ of Algorithm~\ref{alg:PriRFed} against the number of outer iterations with $(\epsilon,\delta)=(0.15,10^{-4})$ and $\hat{\delta}=10^{-3}$. The number $a.bc_{k}$ means $a.bc\times 10^{k}$.} 
\resizebox{16.5cm}{!}{
\begin{tabular}{lllllllllllllll}
\toprule
\multirow{2}{*}{$N$} & \multirow{2}{*}{$s$} & \multirow{2}{*}{$\rho$} & \multicolumn{2}{c}{$T=50$} & \multicolumn{2}{c}{$T=100$} & \multicolumn{2}{c}{$T=200$} & \multicolumn{2}{c}{$T=300$} & \multicolumn{2}{c}{$T=400$} & \multicolumn{2}{c}{$T=500$} \\ \cline{4-15}
                     &                      &                         & $\epsilon'$  & $\delta'$   & $\epsilon'$  & $\delta'$    & $\epsilon'$  & $\delta'$    & $\epsilon'$  & $\delta'$    & $\epsilon'$  & $\delta'$    & $\epsilon'$  & $\delta'$    \\ \hline
$100$                & $  1$                & $1.00_{-2}$             & $4.26_{-2}$  & $1.05_{-3}$ & $6.04_{-2}$  & $1.10_{-3}$  & $8.55_{-2}$  & $1.20_{-3}$  & $1.05_{-1}$  & $1.30_{-3}$  & $1.21_{-1}$  & $1.40_{-3}$  & $1.36_{-1}$  & $1.50_{-3}$  \\
$100$                & $  5$                & $5.00_{-2}$             & $     1.58$  & $2.25_{-3}$ & $     2.32$  & $3.50_{-3}$  & $     3.46$  & $6.00_{-3}$  & $     4.41$  & $8.50_{-3}$  & $     5.25$  & $1.10_{-2}$  & $     6.03$  & $1.35_{-2}$  \\
$200$                & $  1$                & $5.00_{-3}$             & $2.13_{-2}$  & $1.03_{-3}$ & $3.01_{-2}$  & $1.05_{-3}$  & $4.26_{-2}$  & $1.10_{-3}$  & $5.23_{-2}$  & $1.15_{-3}$  & $6.04_{-2}$  & $1.20_{-3}$  & $6.76_{-2}$  & $1.25_{-3}$  \\
$500$                & $  5$                & $1.00_{-2}$             & $2.98_{-1}$  & $1.25_{-3}$ & $4.25_{-1}$  & $1.50_{-3}$  & $6.09_{-1}$  & $2.00_{-3}$  & $7.52_{-1}$  & $2.50_{-3}$  & $8.75_{-1}$  & $3.00_{-3}$  & $9.85_{-1}$  & $3.50_{-3}$  \\
$300$                & $  5$                & $1.67_{-2}$             & $5.02_{-1}$  & $1.42_{-3}$ & $7.20_{-1}$  & $1.83_{-3}$  & $     1.04$  & $2.67_{-3}$  & $     1.29$  & $3.50_{-3}$  & $     1.51$  & $4.33_{-3}$  & $     1.70$  & $5.17_{-3}$  \\
$300$                & $ 10$                & $3.33_{-2}$             & $     3.52$  & $2.67_{-3}$ & $     5.36$  & $4.33_{-3}$  & $     8.32$  & $7.67_{-3}$  & $ 1.09_{1}$  & $1.10_{-2}$  & $ 1.33_{1}$  & $1.43_{-2}$  & $ 1.55_{1}$  & $1.77_{-2}$  \\
$400$                & $  5$                & $1.25_{-2}$             & $3.74_{-1}$  & $1.31_{-3}$ & $5.35_{-1}$  & $1.63_{-3}$  & $7.68_{-1}$  & $2.25_{-3}$  & $9.51_{-1}$  & $2.88_{-3}$  & $     1.11$  & $3.50_{-3}$  & $     1.25$  & $4.13_{-3}$  \\
$400$                & $ 10$                & $2.50_{-2}$             & $     2.56$  & $2.25_{-3}$ & $     3.83$  & $3.50_{-3}$  & $     5.84$  & $6.00_{-3}$  & $     7.55$  & $8.50_{-3}$  & $     9.11$  & $1.10_{-2}$  & $ 1.06_{1}$  & $1.35_{-2}$ \\
\bottomrule
\end{tabular}}
\label{table}
\end{table}

From Table~\ref{table}, we observe that for a suitable size of sampling, e.g., $s=5$, the privacy budget from Theorem~\ref{th:DP} for Algorithm~\ref{alg:PriRFed} is controlled in a reasonable range. In addition, if the sampling rate is sufficiently small, then the number of outer iterations, $T$, can be large while keeping $\epsilon'$ and $\delta'$ reasonably small. 
The numerical results in Section~\ref{sec:Num} show that using a small sampling rate (e.g., sampling only one agent), the resulting trained model has satisfactory accuracy with a reasonable size of the number of outer iterations $T$.

\section{Convergence analysis} \label{sec:conv}

In this section, we investigate the convergence properties for Algorithm~\ref{alg:PriRFed} combined with two privately local training procedures. One is the DP-RSGD~\cite{HMJG22} (the DP-SGD on Riemannian manifolds) discussed in Section~\ref{sec:conv:DP-RSGD} and the other one is the DP-RSVRG~\cite{UHJM23} (the DP-SVRG on Riemannian manifolds) discussed in Section~\ref{sec:conv:DP-RSVRG}. The proofs can be found in Appendix~\ref{appendix1} and~\ref{appendix2}.


The convergence analysis relies on Assumptions~\ref{ass:geodes-Lip}, \ref{ass:smooth} and~\ref{ass:regularization}, which have been used in the federate learning and the differential privacy setting~\cite{ZRS17,LM22,HMJG22,UHJM23}.
\begin{assumption} \label{ass:geodes-Lip}
	For all $i\in [N]$, $j\in[N_i]$, we suppose that the function $f_{ij}$ is geodesic $L_{f}$-Lipschitz continuous. Thus, the function $f_i$ is geodesic $L_f$-Lipschitz continuous for $i\in[N]$ and $f$ also is $L_f$-Lipschitz continuous.
\end{assumption}

\begin{assumption} \label{ass:smooth}
	For all $i\in [N]$, $j\in[N_i]$, we suppose that the function $f_{ij}$ is geodesically ${L}_{g}$-smooth, which implies that $f_i$ is geodesically ${L}_{g}$-smooth and that $f$ is geodesically ${L}_g$-smooth.
\end{assumption}

\begin{assumption} \label{ass:regularization}
	The manifold under consideration is complete and there exists a compact subset $\mathcal{W}\subset \mathcal{M}$ (diameter bounded by $M$) so that all the iterates of Algorithm~\ref{alg:PriRFed} and the optimal points are located in $\mathcal{W}$. The sectional curvature in $\mathcal{W}$ is bounded in $[\kappa_{\min},\kappa_{\max}]$. Moreover, we denote the following key geometrical constant that captures the impact of the manifold:
	\begin{equation} \label{eq:02}
	\begin{aligned}
\zeta = \begin{cases}
 		\frac{\sqrt{|\kappa_{\min}|}M}{\tanh(\sqrt{|\kappa_{\min}|}M)} &\text{if }\kappa_{\min} <0,  \\
 		1 & \text{if }\kappa_{\min}\ge0 .
\end{cases}
\end{aligned}
\end{equation}
\end{assumption}

\subsection{DP-RSGD for privately local training} \label{sec:conv:DP-RSGD}


The DP-RSGD algorithm inspired by \cite{HMJG22} for the local training of agent $i$ in Algorithm~\ref{alg:PriRFed} is stated in Algorithm~\ref{alg:DP-RSGD}.

\begin{algorithm}[ht]
  \caption{Private Local Training for Agent $i$: DP-RSGD}
  \begin{algorithmic}[1] \label{alg:DP-RSGD}
  \REQUIRE global parameter $x^{(t)}$, local dataset ${D}_i$, batch size $b_i$, noise scale $\sigma_i$, the number of training $K$, stepsize $\alpha$;
  \ENSURE $\tilde{x}_i^{(t+1)}$.
  \STATE Set $x_{i,0}^{(t)}\gets x^{(t)}$;
  \FOR{$k = 0,1,\dots,K-1$}
  \STATE Select $\mathcal{B}_k\subseteq [N_i]$ of size $b_i$, where the samples are randomly selected uniformly without replacement;
  \STATE Set $\eta_k^{(t)} \gets \left( \frac{1}{b_i}\sum_{p\in \mathcal{B}_k}\mathrm{clip}_{\tau}(\mathrm{grad}f_{i,p}(x_{i,k}^{(t)})) \right)+\varepsilon_{i,k}^{(t)}$, where $\mathrm{clip}_{\tau}(\cdot): \mathrm{T}_x \mathcal{M} \rightarrow \mathrm{T}_{x} \mathcal{M}: v \mapsto \min\{\tau/\|v\|_x,1\}v$ and $\varepsilon_{i,k}^{(t)}  \sim \mathcal{N}_{x_{i,k}^{(t)}}(0,\sigma_i^2)$;
  \STATE Update $x_{i,k+1}^{(t)} \gets \mathrm{Exp}_{x_k^{(t)}}(-\alpha \eta_k^{(t)})$;
  \ENDFOR
  \STATE \textbf{Option 1:} $\tilde{x}_i^{(t+1)} \gets x_{i,K}^{(t)}$ if \textbf{Option 1} is chosen in Algorithm~\ref{alg:PriRFed};
  \STATE \textbf{Option 2:} $\tilde{x}_i^{(t+1)}$ as uniformly selected at random from $\{x_{i,k}^{(t)}\}_{k=0}^{K}$ if \textbf{Option 2} is chosen in Algorithm~\ref{alg:PriRFed};
 \end{algorithmic}
\end{algorithm}

Algorithm~\ref{alg:DP-RSGD} starts by receiving the global parameter $x^{(t)}$ from the server.
At the $k$-th iteration, $b_i$ samples are randomly and uniformly selected from $D_i$ without replacement. One step of noisy gradient descent with fixed step size $\alpha$ is then performed, where the noise is generated by the so-called Gaussian mechanism on tangent space, see~\cite{HMJG22}. After executing $K$ iterations, the final output is returned in two ways, corresponding to the two options in Algorithm~\ref{alg:PriRFed}.


The clipping operation $\mathrm{clip}_{\tau}(\cdot)$ ensures that the norm of the clipped vector is not greater than $\tau$. It has been shown in~\cite[Theorem~1]{HMJG22} that setting $\sigma_i^2=o_i\frac{K\log(1/\delta)\tau^2}{N_i^2\epsilon^2}$ ensures that Algorithm~\ref{alg:DP-RSGD} is $(\epsilon,\delta)$-DP for agent $i$, where $o_i$ is a constant.
Note that under the assumption that the function $f_{i, j}$ is $L_f$-Lipschitz continuous, the norm of the gradient $\mathrm{grad}f_{i,p}(x_{i,k}^{(t)})$ is always smaller than $L_f$. However, the constant $L_f$ is unknown in practice. Using the clipping operation yields a known upper bound, which makes the parameter $\sigma_i$ more tractable. This technique has been used before~\cite{ACG16,UHJM23}. 

Now we are ready to give the first convergence of Algorithm~\ref{alg:PriRFed} combined with Algorithm~\ref{alg:DP-RSGD}, as stated in Theorem~\ref{th:conv:for DP-RSGD1}, which is inspired by~\cite{LM22}. Theorem~\ref{th:conv:for DP-RSGD1} depends on the setting, where all agents in an FL system are active and they perform one-step gradient descent with a fixed step size.

\begin{theorem}[Nonconvex, $K=1$] \label{th:conv:for DP-RSGD1}
	Suppose that Problem~\eqref{prob} satisfies Assumption~\ref{ass:geodes-Lip} and~\ref{ass:smooth}. 
	We run Algorithm~\ref{alg:PriRFed} combined with Algorithm~\ref{alg:DP-RSGD} with \textbf{Option 1}. If we set $s_t=N$, $K=1$, $\alpha_t=\alpha\le 1/L_g$, and $b_i=N_i$ (i.e., all agents are selected and they perform one-step gradient descent with fixed stepsize), then for the iterates $\{x^{(t)}\}$ of Algorithm~\ref{alg:PriRFed}, it holds that 	
	\begin{align} \label{th:conv:for DP-RSGD1:0}
		\min_{0 \leq t \leq T} \mathbb{E}[\|\mathrm{grad}f(x^{(t)})\|^2] \le \frac{2}{T\alpha}(f(x^{(0)})-f(x^*)) + \frac{dL_g\alpha}{ (\sum_{i=1}^NN_i)^2}\sum_{i=1}^NN_i^2\sigma_i^2,
	\end{align}
	where the expectation is taken over $\varepsilon_{i,0}^{(t)}$, and $x^*=\argmin_{x\in\mathcal{M}}f(x)$.
\end{theorem}

Theorem~\ref{th:conv:for DP-RSGD1} approximately estimates the upper bound of the norm of gradients. If one does not consider the privacy, which implies $\sigma_i=0$, and let the step size $\alpha = \chi / L_g$ for a constant $\chi \in (0, 1)$, then the second term vanishes in the right-hand side of inequality~\eqref{th:conv:for DP-RSGD1:0}. This is consistent with the existing results, e.g.,~\cite[Theorem~7]{LM22}. In contrast with Theorem~\ref{th:conv:for DP-RSGD1}, if one uses decaying step sizes, 
then Algorithm~\ref{alg:PriRFed} combined with Algorithm~\ref{alg:DP-RSGD} has global convergence, as stated in Theorem~\ref{th:conv:for DP-RSGD1_}.

\begin{theorem}[Nonconvex, $K=1$] \label{th:conv:for DP-RSGD1_}
	Under the same conditions as Theorem~\ref{th:conv:for DP-RSGD1} except that the step sizes $\{\alpha_t\}_{t=0}^{T-1}$ satisfies  
\begin{align} \label{th:conv:for DP-RSGD1_:1}
	\lim_{T \rightarrow \infty}\sum_{t=0}^{T-1}\alpha_t=\infty, \; 
	\lim_{T \rightarrow \infty}\sum_{t=0}^{T-1}\alpha_t^2 < \infty,\hbox{ and }  L_g\alpha_t \le 2\delta,
\end{align}
for a constant $\delta\in(0,1)$, the iterates $\{x^{(t)}\}_{t=0}^{T-1}$ generated by Algorithm~\ref{alg:PriRFed} with Algorithm~\ref{alg:DP-RSGD} satisfies 
$\liminf_{T \rightarrow \infty} \mathbb{E}\left[ \|\mathrm{grad}f(x^{(t)})\|^2 \right] = 0$.
\end{theorem}
Note that the conditoins~\eqref{th:conv:for DP-RSGD1_:1} are standard and commonly used in Riemannian/Euclidean stochastic gradient methods and Euclidean federated learning.
Under the same assumptions as Theorem~\ref{th:conv:for DP-RSGD1} along with the geodesic convexity of $f_{i,j}$, the convergence bound of the difference between the cost value at the final iterate and the minimum cost value can be characterized, as stated in Theorem~\ref{th:conv:for DP-RSGD3}.

\begin{theorem}[Convex, $K=1$]
\label{th:conv:for DP-RSGD3}
	Suppose that Problem~\eqref{prob} satisfies Assumption~\ref{ass:geodes-Lip},~\ref{ass:smooth} and~\ref{ass:regularization}, and that the local functions $f_{i}$'s are geodesically convex in $	\mathcal{W}$. 
    Let $\{x^{(t)}\}_{t=0}^{T}$ be a sequence generated by Algorithm~\ref{alg:PriRFed} with Algorithm~\ref{alg:DP-RSGD} and \textbf{Option 1}. Set the parameters to be $s_t=N$, $K=1$, $\alpha_t=\alpha\le 1/(2L_g)$, and $b_i=N_i$. then 	
	it holds that
\begin{equation} \label{eq:04}
		\mathbb{E}[f(x^{(T)})-f(x^*)] \le \frac{\zeta\mathrm{dist}^2(x^{(0)},x^*)}{2\alpha(\zeta + T - 1)} + \frac{T(6\zeta + T-1)d}{16{L}_g(\zeta + T - 1)(\sum_{i=1}^NN_i)^2}\sum_{i=1}^NN_i^2\sigma_i^2,
\end{equation}
where the expectation is taken over $\varepsilon_{i,0}^{(t)}$, $x^*=\argmin_{x\in\mathcal{M}}f(x)$, and $\zeta$ is defined in~\eqref{eq:02}.
\end{theorem}
By Theorem~\ref{th:conv:for DP-RSGD3}, there exists an optimal choice of the number of training iterations $T$ such that the right hand side of~\eqref{eq:04} reaches its minimum.
Note that in non-DP settings (i.e., $\sigma_i=0$), this result is consistent with the one in~\cite[Theorem~9]{LM22}.


The sampling rate $\rho$ is $1$ in Theorems~\ref{th:conv:for DP-RSGD1}, \ref{th:conv:for DP-RSGD1_} and~\ref{th:conv:for DP-RSGD3}. With additional assumptions stated in Assumption~\ref{ass:4}, the convergence result can be established with $\rho < 1$. Assumption~\ref{ass:4} is a Riemannian generalization of the Euclidean versions in, e.g.,~\cite{WLD20,WLDMSZP22}. 

\begin{assumption} \label{ass:4}
	Problem~\eqref{prob} satisfies the following conditions:
	\begin{itemize}
		\item $f_i$ is geodesic convex in $\mathcal{M}$;
		\item $f_i$ satisfies the Riemannian Polyak-Lojasiewicz (RPL) condition with the positive constant $\mu$, meaning that $f_i(x)-f(x^*)\le \frac{1}{2\mu}\|\mathrm{grad}f(x)\|^2$, where $x^*$ is the optimal point;
		\item $\alpha_t=\alpha \le \frac{1}{{L}_g}$, where $\alpha_i$ is the step size;
		\item For any $i\in[N]$ and $x \in \mathcal{M}$, $\|\mathrm{grad}f_i(x)-\mathrm{grad}f(x)\|^2\le v_i$ and $\mathbb{E}[v_i]=v$.
	\end{itemize}
\end{assumption}
In Assumption~\ref{ass:4}, $v_i$, called divergence metric in~\cite{WLD20}, characterizes the divergence between the gradient of the local function and the global function, which is essential for analyzing the convergence of SGD~\cite{WLD20}. The divergence also manifests how the samples are distributed in different agents in some sense.

\begin{theorem}[RPL, $K=1$] \label{th:conv:for DP-RSGD4}
	Suppose that Problem~\eqref{prob} satisfies Assumption~\ref{ass:geodes-Lip},~\ref{ass:smooth} and~\ref{ass:4}. 
 Let $\{x_t\}_{t=0}^{T}$ be a sequence generated by  Algorithm~\ref{alg:PriRFed} with Algorithm~\ref{alg:DP-RSGD} and \textbf{Option 1}. Set $N_i=N_j \; (\forall \; i,j\in[N])$, $s_t=S<N$, $K=1$, and $b_i=N_i$. Then it holds that	
\begin{equation} \label{th:conv:for DP-RSGD4:0}
	\mathbb{E}[f(x^{(T)})] - f(x^*) \le P^T (f(x^{(0)})-f(x^*)) + (1-P^T)\left(\kappa_0 \frac{\sum_{i\in \mathcal{S}}\sigma_i^2}{S^2} + \kappa_1 \frac{N-S}{S(N-1)}\right),
\end{equation}
where $P=(1-2\mu \alpha + \mu \alpha^2{L}_g)$, $\kappa_0=\frac{{L}_g\alpha d}{2\mu(2-{L}_g\alpha)}$, and $\kappa_1 =v\kappa_0 = \frac{{L}_g\alpha dv}{2\mu(2-{L}_g\alpha)}$.
\end{theorem}
There exists an optimal number of iterations $T$ such that the upper bound in~\eqref{th:conv:for DP-RSGD4:0} is minimized by following the same approach in~\cite[Theorem 3]{WLDMSZP22}.

The convergence results above depend on setting $K=1$ and $b_i=N_i$, that is, performing one-step gradient descent in local training procedures. 
When $K>1$, multiple inner iterations are used in the local training procedure, which makes the convergence analysis more challenging.
Here, we consider the scenario where only one agent is selected to perform private local training at each outer iteration. The convergence result gives an upper bound for the norm of the gradients, as stated in Theorem~\ref{th:conv:for DP-RSGD2}, which is inspired by~\cite{ZRS17}. The difference to the existing result lies in that i) the FL setting and the differential privacy are not taken into consideration in~\cite{ZRS17}; and ii) SVRG is used in~\cite{ZRS17} instead of SGD.

\begin{theorem}[Nonconvex, $K>1$] \label{th:conv:for DP-RSGD2}
	Suppose that Problem~\eqref{prob} satisfies Assumption~\ref{ass:geodes-Lip},~\ref{ass:smooth} and~\ref{ass:regularization}. 
	Consider Algorithm~\ref{alg:PriRFed} with Algorithm~\ref{alg:DP-RSGD} and \textbf{Option 2}.
	Set $s_t=1$, $K>1$, $\alpha_t=\alpha < (\beta^{\frac{1}{K-1}}-1)/\beta$ with $\beta>1$ being a free constant. Then the output of the algorithm $\tilde{x}$ satisfies
	\[
	\mathbb{E}[\|\mathrm{grad}f(\tilde{x})\|^2]\le \frac{f(x^{(0)})-f(x^*)}{KT\delta_{\min}} + q_0q,
	\]
	where $\delta_{\min}=\alpha(1-\frac{(1+\alpha\beta)^{K-1}}{\beta})$,
	$q_0=\frac{\alpha\beta({L}_g+2(1+\alpha\beta)^{K-1}\zeta)}{2(\beta-(1+\alpha\beta)^{K-1})}$, $q={L}_f^2+d\sigma^2$, and $\sigma=\max_{i=1,2,\dots,N}\sigma_i$.
\end{theorem}

\subsection{DP-RSVRG for privately local training} \label{sec:conv:DP-RSVRG}

Stochastic Variance Reduced Gradient (SVRG) computes the full gradient every few steps to reduce the variance and therefore is generally faster than the SGD.
Zhang et al.~\cite{ZRS17} generalizes SVRG to Riemannian manifolds and Uptake et al.~\cite{UHJM23} combine RSVRG and DP to propose an RSVRG preserving privacy. Consequently, we select DP-RSVRG as seen in Algorithm~\ref{alg:DP-RSVRG} derived from~\cite{UHJM23} as the privately local training and then analyze corresponding convergent results. 

\begin{algorithm}[ht]
  \caption{Private Local Training for Agent $i$: DP-RSVRG \cite{UHJM23}}
  \begin{algorithmic}[1] \label{alg:DP-RSVRG}
  \REQUIRE global parameter $x^{(t)}$, local dataset ${D}_i$, update frequency $m$, noise scale $\sigma_i$, number of training $K$, stepsize $\alpha_t$;
  \ENSURE $\tilde{x}_i^{(t+1)}$.
  \STATE \label{alg:DP-RSVRG:step1} Set $\tilde{x}^{(t)}_{i,0}\gets x^{(t)}$;
  \FOR{$k = 0,1,\dots,K-1$}
  \STATE \label{alg:DP-RSVRG:step3} $x_{k+1,0}^{(t)} \gets \tilde{x}_{i,k}^{(t)}$;
  \STATE \label{alg:DP-RSVRG:step4} $g_{k+1}^{(t)} \gets \frac{1}{N_i}\sum_{j=1}^{N_i}\mathrm{clip}_{\tau_0}(\mathrm{grad}f_{i,j}(\tilde{x}^{(t)}_{i,k}))$;
  \FOR{$j=0,1,\dots,m-1$}
  \STATE \label{alg:DP-RSVRG:step6} Randomly pick $l_j\in[N_i]$
  \STATE \label{alg:DP-RSVRG:step7} $\epsilon_{k+1,j}^{(t)}\sim \mathcal{N}_{x_{k+1,j}^{(t)}}(0,\sigma_i)$;
  \STATE \label{alg:DP-RSVRG:step8} $\eta_{k+1,j}^{(t)} \gets \mathrm{clip}_{\tau_1}(\mathrm{grad}f_{i,l_j}(x_{k+1,j}^{(t)})) - \Gamma_{\tilde{x}^{(t)}_{i,k}}^{x_{k+1,j}^{(t)}}(\mathrm{clip}_{\tau_1}(\mathrm{grad}f_{i,l_j}(\tilde{x}^{(t)}_{i,k}))-g_{k+1}^{(t)}) + \varepsilon_{k+1,j}^{(t)}$;
  \STATE \label{alg:DP-RSVRG:step9} $x_{k+1,j+1}^{(t)} \gets \mathrm{Exp}_{x_{k+1,j}^{(t)}}(-\alpha_t \eta_{k+1,j}^{(t)})$;
  \ENDFOR
  \STATE \label{alg:DP-RSVRG:step11} $\tilde{x}^{(t)}_{i,k+1}\gets x_{k+1,m}^{(t)}$;
  \ENDFOR
  \STATE \label{alg:DP-RSVRG:step13} \textbf{Option 1:} $\tilde{x}_i^{(t+1)} \gets x_{K,m}^{(t)}$;
  \STATE \label{alg:DP-RSVRG:step14} \textbf{Option 2:} $\tilde{x}_i^{(t+1)}$ is uniformly randomly selected from $\{\{x_{k+1,j}^{(t)}\}_{j=0}^{m-1}\}_{k=0}^{K-1}$;
 \end{algorithmic}
\end{algorithm}

Algorithm~\ref{alg:DP-RSVRG} starts by receiving the global parameter $x^{(t)}$ from the server and then enters into its outer iteration. In Line~\ref{alg:DP-RSVRG:step4}, the full gradient of $f_i$ at iterate $\tilde{x}_{i,k}^{(t)}$  is computed. This is used to form the variance reduced stochastic gradient step. In the inner iteration of Algorithm~\ref{alg:DP-RSVRG}, select randomly and uniformly a sample from the local $N_i$ samples to perform variance reduced stochastic gradient update while injecting noise into the corrected gradient step. After updating $m$-step variance reduced stochastic gradient with noise, Algorithm~\ref{alg:DP-RSVRG} computes a new full gradient at the next iterate. The out iteration is repeated $K$ times.
There are two choices of the final output, one of which is the final iterate, one of which is selected randomly and uniformly from all intermediate iterates.

In this case, by~\cite[Claim~5]{UHJM23}, setting $\sigma_i^2=\tilde{o}_i\frac{mK\log(1/\delta)\tau^2}{N_i^2\epsilon^2}$ with some constant $\tilde{o}_i$ and $\tau=\max\{\tau_0,\tau_1\}$ guarantees that Algorithm~\ref{alg:DP-RSVRG} is $(\epsilon,\delta)$-DP.

It is worth noting that Algorithm~\ref{alg:DP-RSVRG} degenerates into Algorithm~\ref{alg:DP-RSGD} when $m=1$. Therefore, Theorem~\ref{th:conv:for DP-RSGD2} also holds in this parameter setting. 
Theorem~\ref{th:conv:for DP-RSVRG} shows the convergence result for Algorithm~\ref{alg:PriRFed} with Algorithm~\ref{alg:DP-RSVRG} under the same scenario as Theorem~\ref{th:conv:for DP-RSGD2} but $m>1$, where one agent is chosen to perform privately local training using DP-RSVRG. 

\begin{theorem}[Nonconvex, $K>1,m>1$] \label{th:conv:for DP-RSVRG}
	Suppose that Problem~\eqref{prob} satisfies Assumption~\ref{ass:geodes-Lip} and~\ref{ass:smooth} and~\ref{ass:regularization}. Consider Algorithm~\ref{alg:PriRFed} with~\ref{alg:DP-RSVRG} and \textbf{Option 2}. Set $s_t=1$, $K>1$, $m=\lfloor 10N/(3\zeta^{1/2}) \rfloor>1$, and $\alpha_{t}=\alpha \le 1/(10{L}_gN^{2/3}\zeta^{1/2})$. Then the output of Algorithm~\ref{alg:PriRFed} satisfies
	\[
		\mathbb{E}[\|\mathrm{grad}f(\tilde{x})\|^2] \le c\frac{{L}_g\zeta^{1/2}}{N^{1/3}KT}[f(x^{(0)})-f(x^*)]+\frac{d\sigma^2}{100\zeta^{1/2}}\left(\frac{1}{N^{2/3}}+\frac{1}{N}\right),
	\]
	where $c> 10 N / (3 m)$ is constant, and $\sigma=\max_{i=1,2,\dots,N}\sigma_i$.
\end{theorem}


\section{Numerical experiments}  \label{sec:Num}


In this section, we illustrate the efficacy of the proposed Algorithm~\ref{alg:PriRFed} combined with two privately local training procedures: DP-RSGD and DP-RSVRG. Three common-used problems in statistics, principal eigenvector computation over sphere manifold~\cite{GH15,JKMPS15,ZRS17}, Fr\'{e}chet mean computation over symmetric positive definite matrix manifold~\cite{Bha07,JVV12,ZRS17} and hyperbolic structured prediction over hyperbolic manifold~\cite{MRC20}, which are still actively studied in machine learning community, are taken into consideration. For these problems, synthetic data and real-world data are used to test the algorithms. The implementation of operations on Riemannian manifolds involved is from the Manopt package~\cite{BMAS14}, a Riemannian optimization toolbox in Matlab, and sampling noise from tangent spaces follows from the same approach in~\cite{UHJM23}. The parameters $\epsilon, \delta, s_t, K, b_i, N_i, N, m, \tau$ of the tested algorithms are specified in later sections. Note that the noise scale $\sigma_i$ is computed by $\sigma_i^2=o_i\frac{K\log(1/\delta)\tau^2}{N_i^2\epsilon^2}$ (\cite[Theorem~1]{HMJG22}) for DP-RSGD and $\sigma_i^2=\tilde{o}_i\frac{mK\log(1/\delta)\tau^2}{N_i^2\epsilon^2}$ (\cite[Claim~5]{UHJM23}) for DP-RSVRG, where $\tilde{o}_i$ is set to be one. In what follows, we use PriRFed-DP-RSGD and PriRFed-DP-RSVRG to denote the combination of Algorithm~\ref{alg:PriRFed} and Algorithm~\ref{alg:DP-RSGD} and the combination of Algorithm~\ref{alg:PriRFed} and  Algorithm~\ref{alg:DP-RSVRG}, respectively.

\subsection{Principal eigenvector computation over sphere manifold}

The first problem is to compute the leading eigenvector of a sample covariance matrix
\begin{align} \label{NumExp:1}
	\min_{x\in \mathrm{ S }^d}f(x):=-\sum_{i=1}^N\frac{p_i}{N_i}\sum_{j=1}^{N_i}x^T(z_{i,j}z_{i,j}^T)x,
\end{align}
where $\mathrm{S}^d=\{ x\in \mathbb{R}^{d+1}:\|x\|_2 =1 \}$, $p_i=N_i/\sum_{i=1}^NN_i$, and $D_i=\{z_{i,1},\dots,z_{i,N_i}\}$ with $z_{i,j}\in \mathbb{R}^{d+1}$,  $i \in[N]$, $j\in [N_i]$.
That is, $f_i(x)=\frac{1}{N_i}\sum_{j=1}^{N_i}f_{i,j}(x)$ and $f_{i,j}(x)=-x^T(z_{i,j}z_{i,j}^T)x$. The unit sphere $\mathrm{S}^d$ is a $d$-dimensional Riemannian submanifold of $\mathbb{R}^{d+1}$ equipped with the Euclidean metric, i.e., $\left<u,v\right>_x=u^Tv$ for any $u,v\in \mathrm{T}_x\mathrm{S}^d$.  The exponential map is given by $\mathrm{Exp}_x(u)=\cos(\|u\|_2)x + \sin(\|u\|_2)u/\|u\|_2$. 

Noting that $\|\mathrm{grad}f_{i,j}(x)\|_x=\|2(I_{d+1}-xx^T)(z_{i,j}z_{i,j}^T)x\|_2\le 2\|(z_{i,j}z_{i,j}^T)x\|_2\le 2z_{i,j}^Tz_{i,j}\le 2\max_{j\in[N_i]}\{z_{i,j}^Tz_{i,j}\}:=2\tilde{\vartheta}_{i}$, hence the geodesic Lipschitz continuous constant $L_f$ is bounded as $L_f\le 2\max_{i\in[N]}\{\tilde{\vartheta}_i\}:=2\tilde{\vartheta}$. On the other hand, the action of Hessian of $f_{i,j}$ at $x\in \mathrm{S}^d$ to $v\in \mathrm{T}_x\mathrm{S}^d$ is given by $\mathrm{Hess}f_{i,j}(x)[v]=2(z_{i,j}z_{i,j}^T)v-2(x^T(z_{i,j}z_{i,j}^T)v)x-2(x^T(z_{i,j}z_{i,j}^T)x)v$. Then for any $v\in \mathrm{T}_x\mathrm{S}^{d}$ and $\|v\|_x=1$,  we have
\begin{align*}
	\left<v,\mathrm{Hess}f_{i,j}(x)[v]\right>_x&= 2v^T(z_{ij}z_{ij}^T)v - 2v^T(x^T(z_{i,j}z_{i,j}^T)v)x-2v^T(x^T(z_{i,j}z_{i,j}^T)x)v \\
	&= 2v^Tz_{i,j}z_{i,j}^Tv - 2v^T(x^T(z_{i,j}z_{i,j}^T)x)v  \le 2z_{i,j}^Tz_{i,j},
\end{align*}
which implies that $\|\mathrm{Hess}f_{i,j}(x)\|_x\le2\tilde{\vartheta}_i$. Therefore, the smooth constant $L_g$ is bounded as $L_g\le 2\tilde{\vartheta}$. Hence, we set $\tau=\tau_0=\tau_1=L_g$.

\paragraph{Synthetic data.}
The sample matrix $Z_i=[z_{i,1}^T;\dots;z_{i,N_i}^T]\in \mathbb{R}^{N_i\times (d+1)}$ for agent $\mathcal{ C }_i$ is generated by following the approach in~\cite[Section~5.1]{HMJG22}. We first construct a $(d+1)\times (d+1)$ diagonal matrix $\Sigma_i=\mathrm{diag}\{1,1-1.1 \nu,\dots,1-1.4 \nu,|y_1|/(d+1),|y_2|/(d+1),\dots\}$ where $\nu$ is the eigengap and $y_1,y_2,\dots$ are drawn from the standard Gaussian distribution. Then construct $Z_i=U_i\Sigma_iV_i$ where $U_i,V_i$ of size $N_i\times (d+1)$ and $(d+1)\times (d+1)$ are random column orthonormal matrices. To explore the trade-off between privacy and accuracy of PriRFed combined with DP-RSGD and DP-RSVRG, the algorithm on multiple privacy levels ($\epsilon=0.08,0.15,0.3$ for local training procedures with the same $\delta=10^{-5}$) is tested. The results are reported in Figure~\ref{fig:NumExp:1} where the parameters $s_t, K, b_i, N_i, N$ are given in the caption and titles.
It is evident that as the privacy budget $\epsilon$ increases, indicating a reduction in the added noise, the cost function value becomes increasingly accurate, and the norm of gradients progressively diminishes. This observation aligns with the theoretical expectations (Theorem~\ref{th:conv:for DP-RSGD1},~\ref{th:conv:for DP-RSGD2}, and~\ref{th:conv:for DP-RSVRG}).

\begin{figure}[ht]
	\centering
	\subfigure{
		\includegraphics[width=0.3\linewidth]{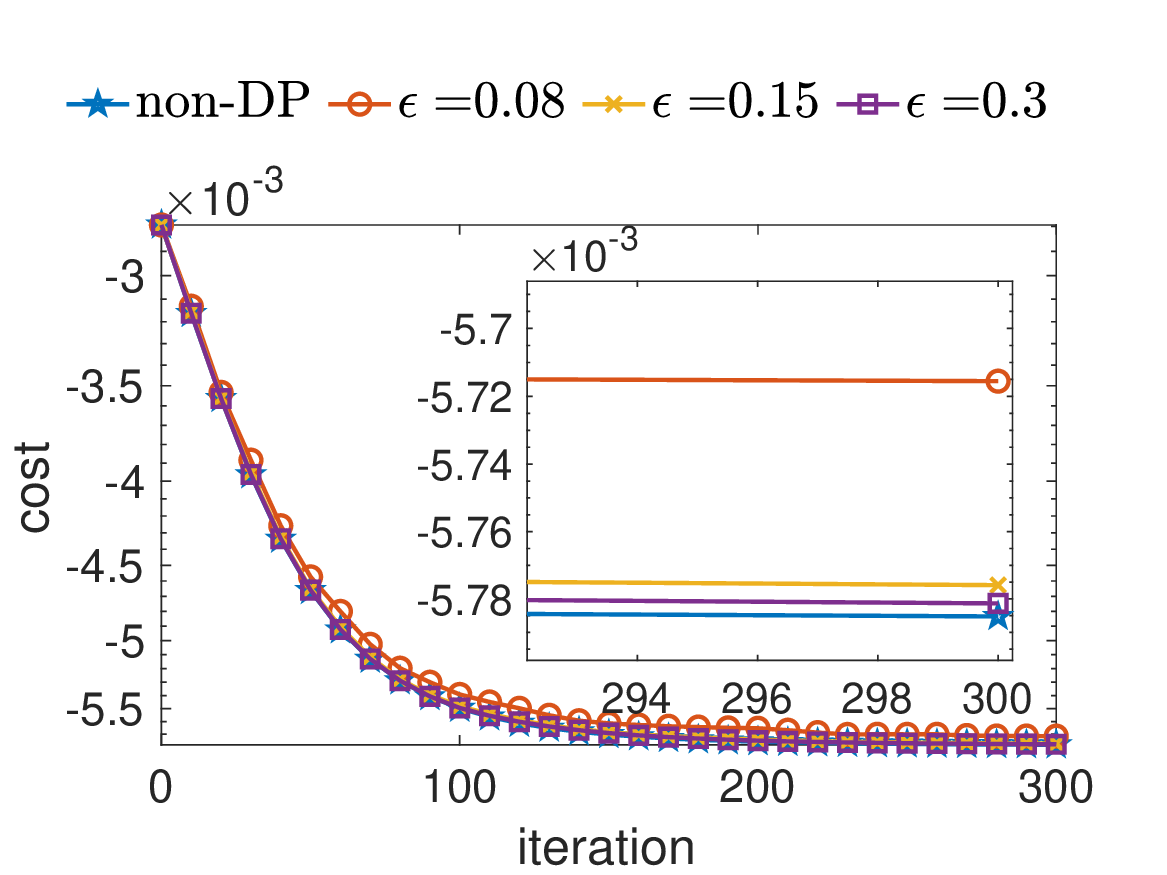}
		}
	\subfigure{
		\includegraphics[width=0.3\linewidth]{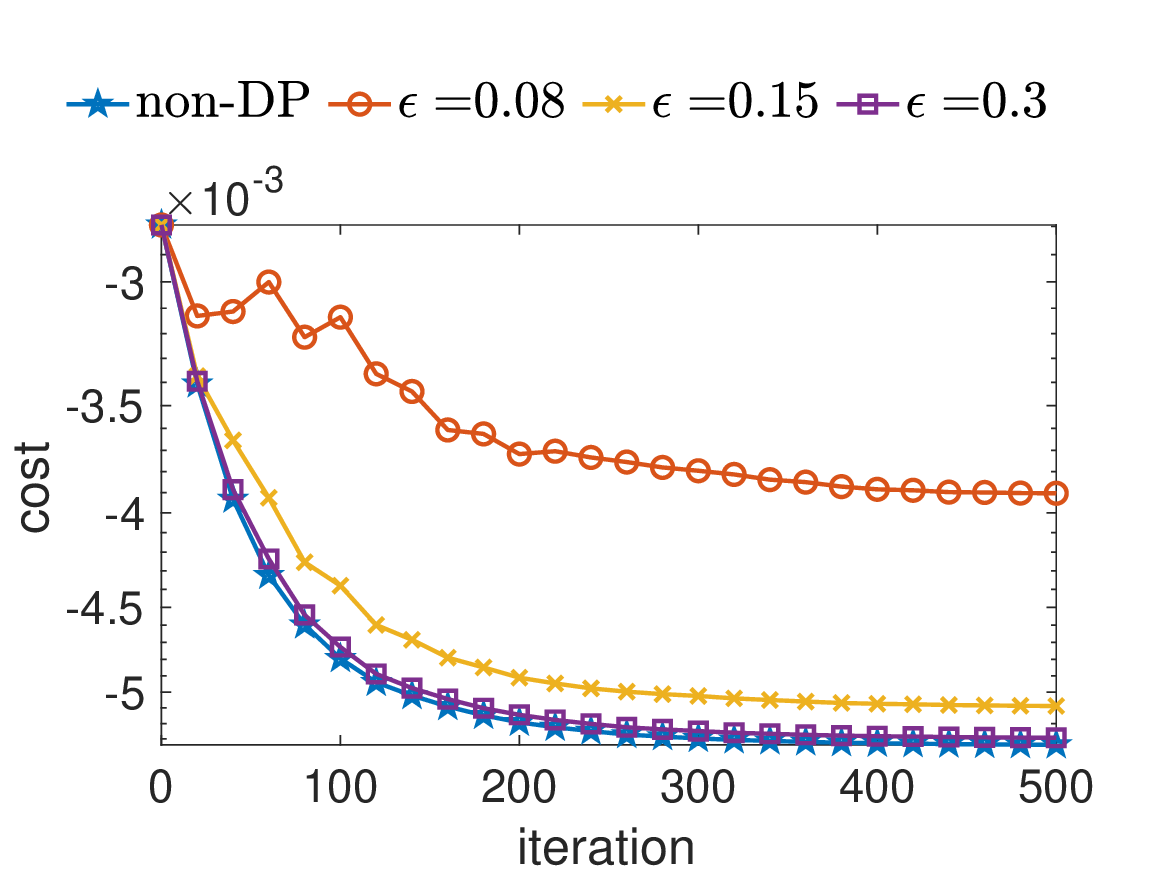}
	}
	\subfigure{
		\includegraphics[width=0.3\linewidth]{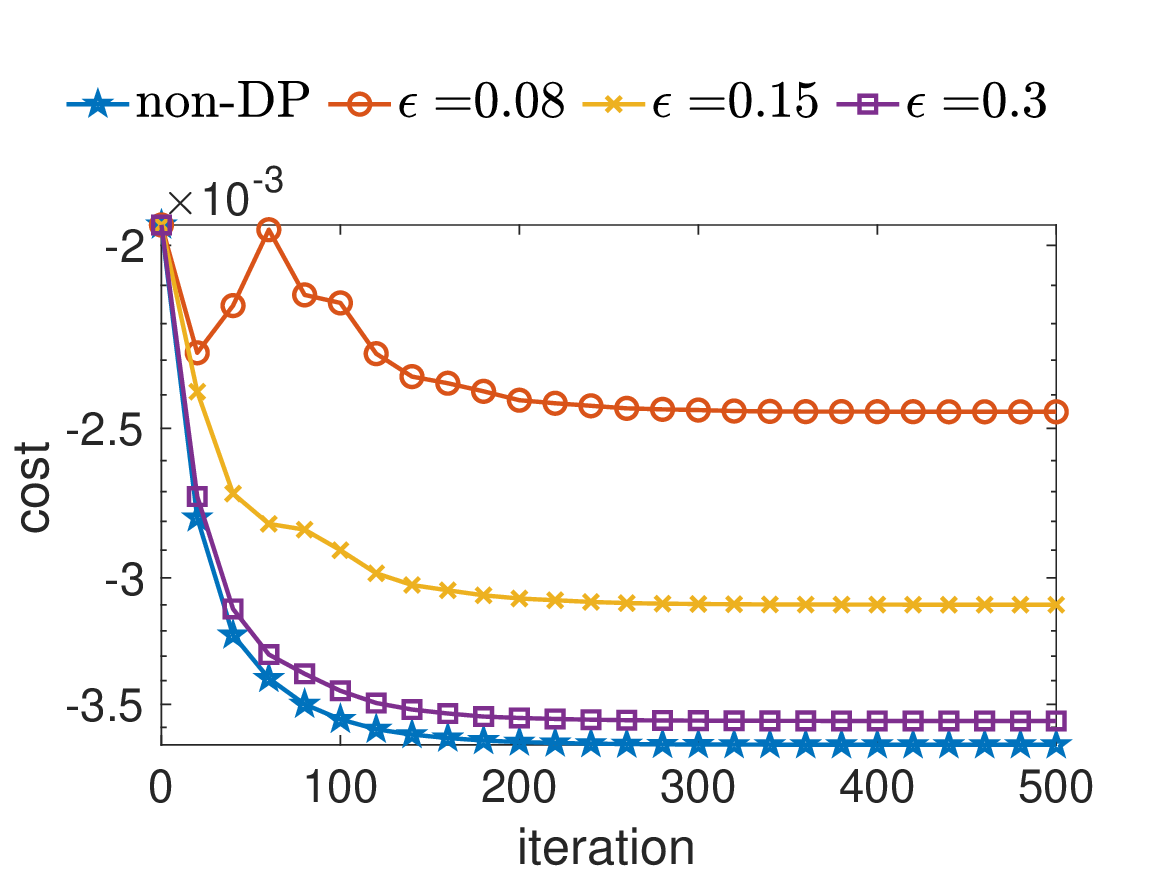}
	}

	\setcounter{subfigure}{0}
	
	\subfigure[$N=16,N_i=70$]{
		\centering
		\includegraphics[width=0.3\linewidth]{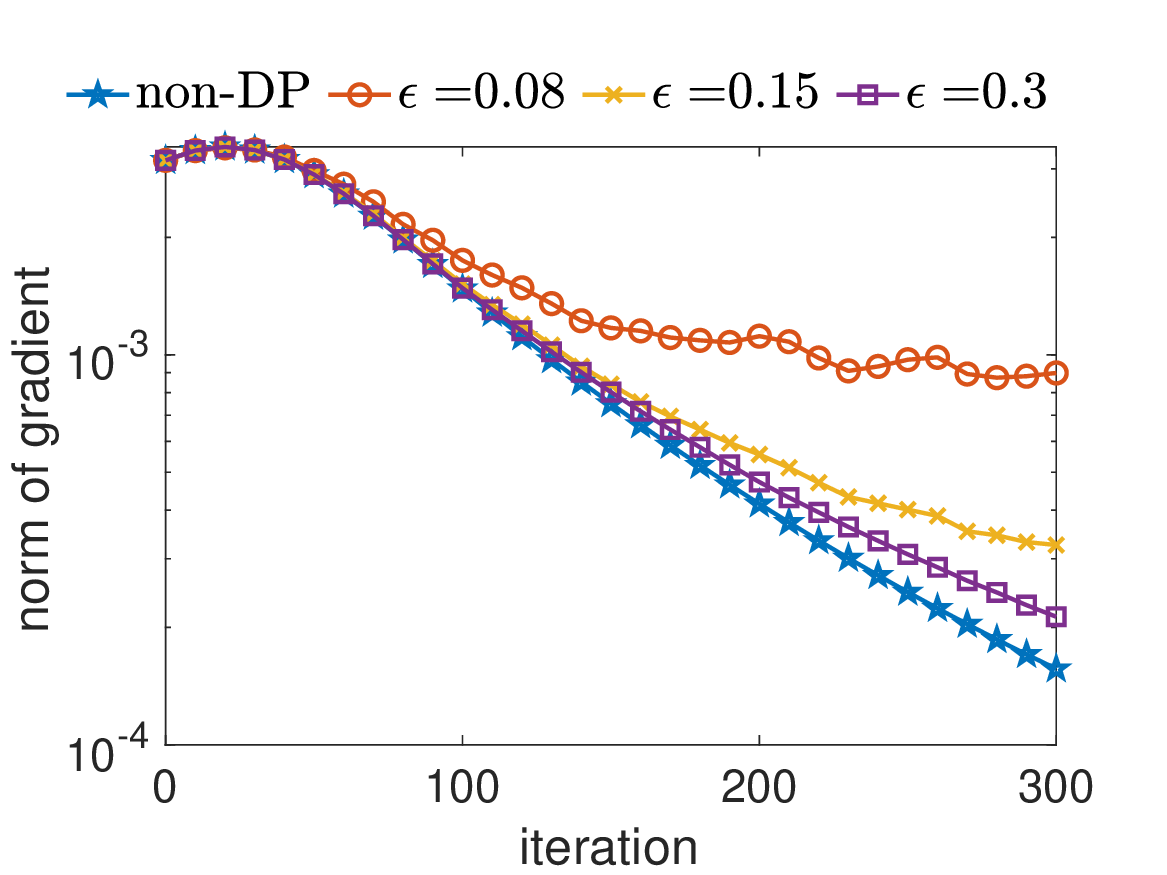}
		\label{fig:NumExp:1-1}
	}
	\subfigure[$N=16,N_i=70$]{
		\centering
		\includegraphics[width=0.3\linewidth]{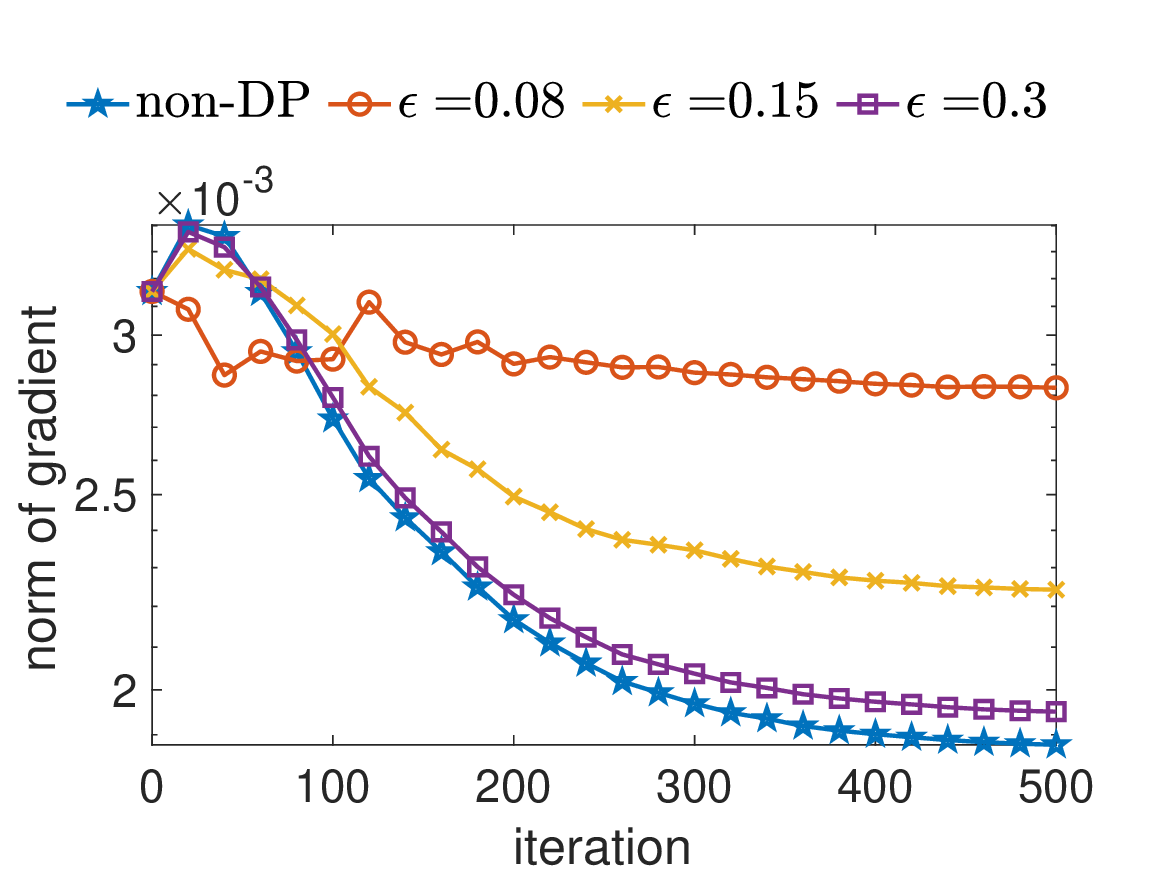}
		\label{fig:NumExp:1-2}
	}
	\subfigure[$N=16,N_i=100$]{
		\centering
		\includegraphics[width=0.3\linewidth]{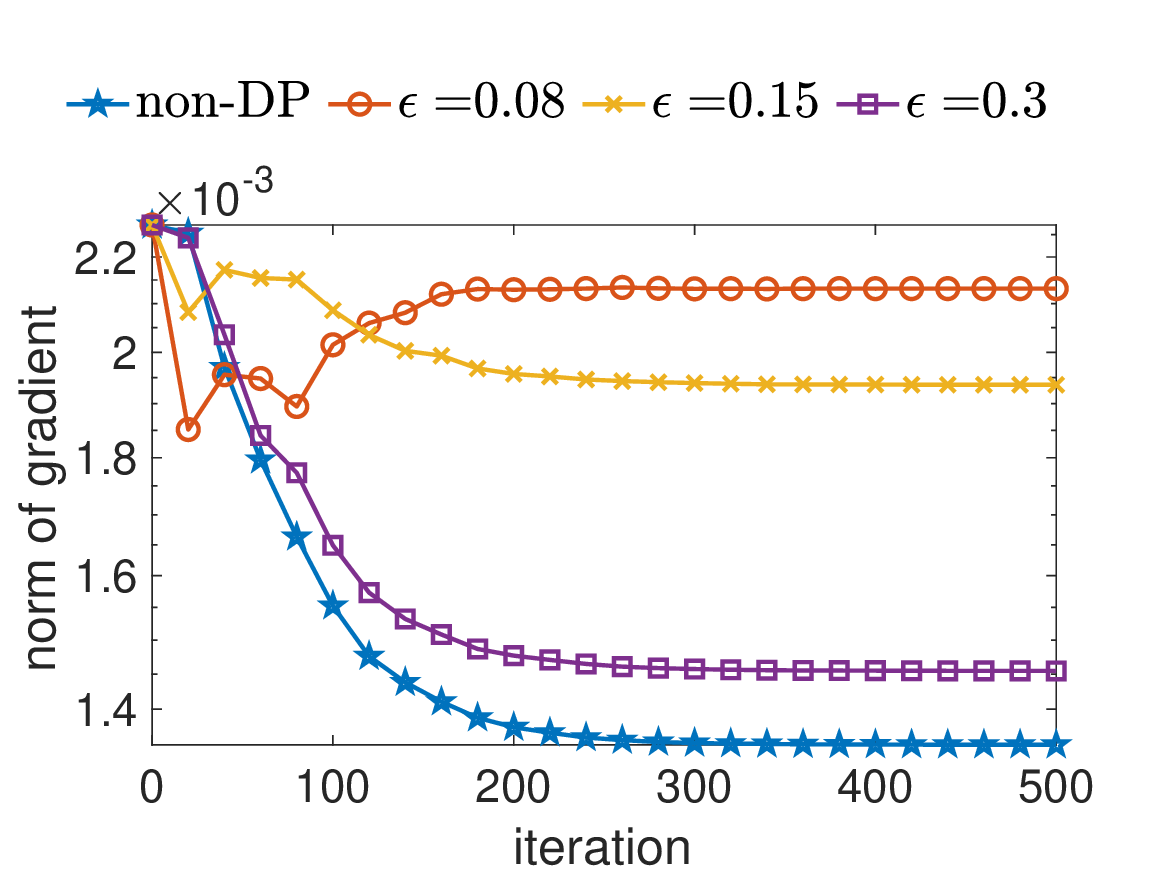}
		\label{fig:NumExp:1-3}
	}
	\caption{Averaged results over $10$ tests for Problem~\eqref{NumExp:1} using PriRFed-DP-RSGD and PriRFed-DP-RSVRG with different private levels. Here $d+1=25$, $v = 10^{-3}$. The $y$-axis of the figures in the first row is the cost value. And the one in the second row denotes $\|\mathrm{grad}f(x^{(t)})\|_{x^{(t)}}$. The $x$-axis of all figures means the iterations. For the first column, $s_t=N$, $K=1$, , $b_i=N_i$ and $\alpha_t=1/(2L_g)$. For the second column, $s_t=1$, $K=5$,  $b_i=N_i/2$, and $\alpha_t=1.0$. In the first two columns, DP-RSGD is used as the privately local training procedure, and DP-RSVRG in the third. 
	For the third column, $s_t=1$, $K=2$, $m=\lfloor 10N/3 \rfloor$, and $\alpha_t=1/(10N^{2/3}L_g)$.}
	\label{fig:NumExp:1}
\end{figure}

\paragraph{MNIST.} The MNIST~\cite{Deng12} is one of the standard datasets in the machine learning field. It contains 60,000 hand-written images of size $28\times 28$ (in the case, $d+1=28^2={784}$). 
We set $\epsilon=0.15$, $\delta=10^{-4}$, $K=3$, $m=\lfloor 10 N/3 \rfloor$ and $\alpha_t=0.1$. We test four scenarios where $(N=50, N_i=1200)$, $(N=60,N_i=1000)$, $(N=80,N_i=750)$, and $(N=100, N_i=600)$, respectively, to explore the performance of Algorithm~\ref{alg:PriRFed} with DP-RSGD and DP-RSVRG for solving Problem~\eqref{NumExp:1}. 
The cost values against the iterations is shown in Figure~\ref{fig:NumExp:4}. In all scenarios, the true value is given by the Riemannian steepest descent method from {Manopt}.

\begin{figure}[ht]
\centering
\subfigure[$N=50,N_i=1200$]{
\includegraphics[width=3.5cm]{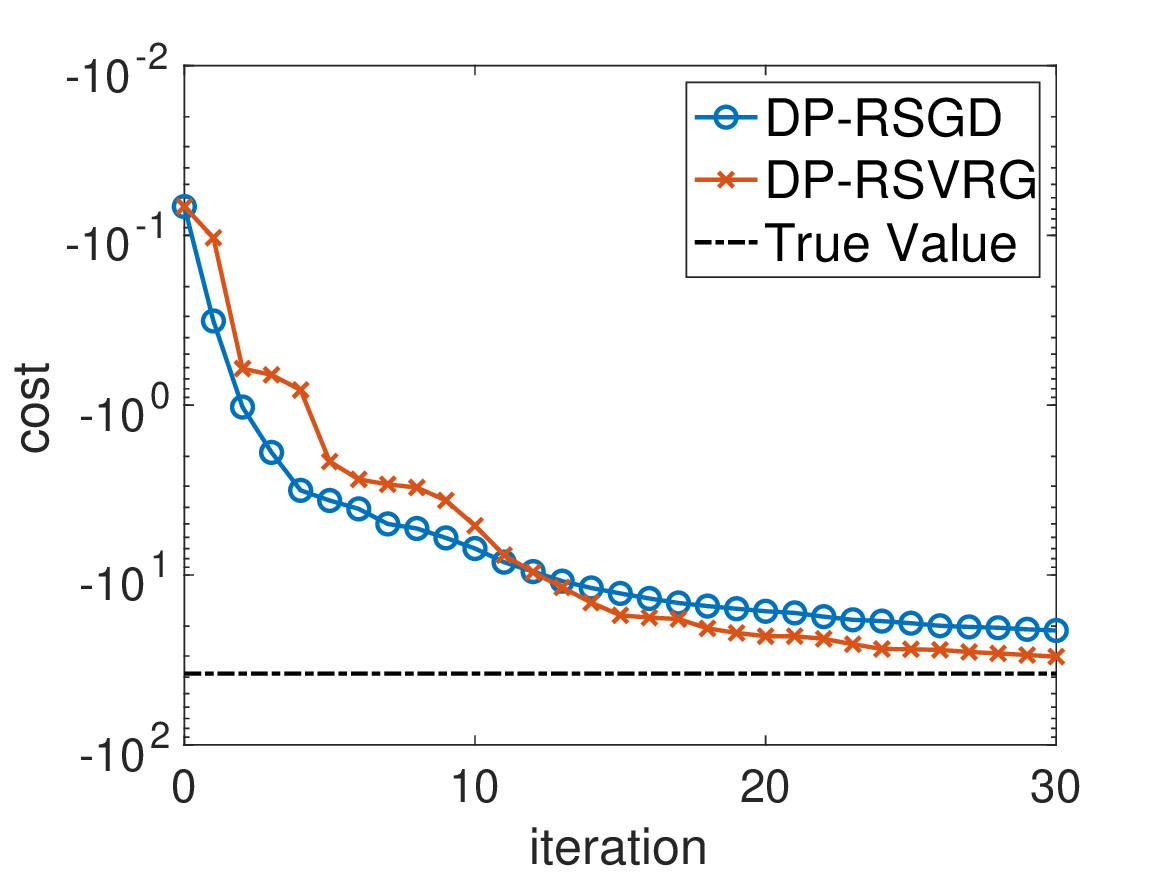}
\label{fig:NumExp:4-1}
}
\quad
\subfigure[$N=60,N_i=1000$]{
\includegraphics[width=3.5cm]{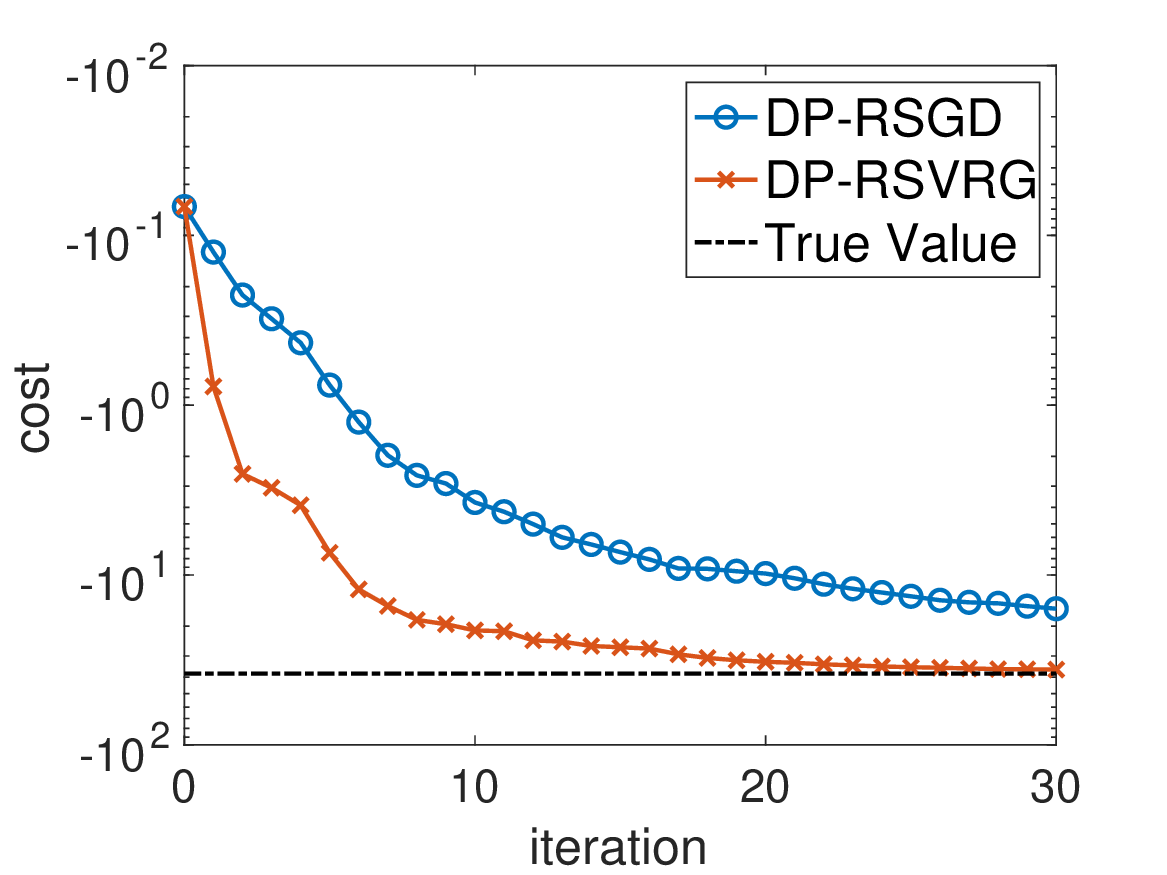}
\label{fig:NumExp:4-2}
}
\quad
\subfigure[$N=80,N_i=750$]{
\includegraphics[width=3.5cm]{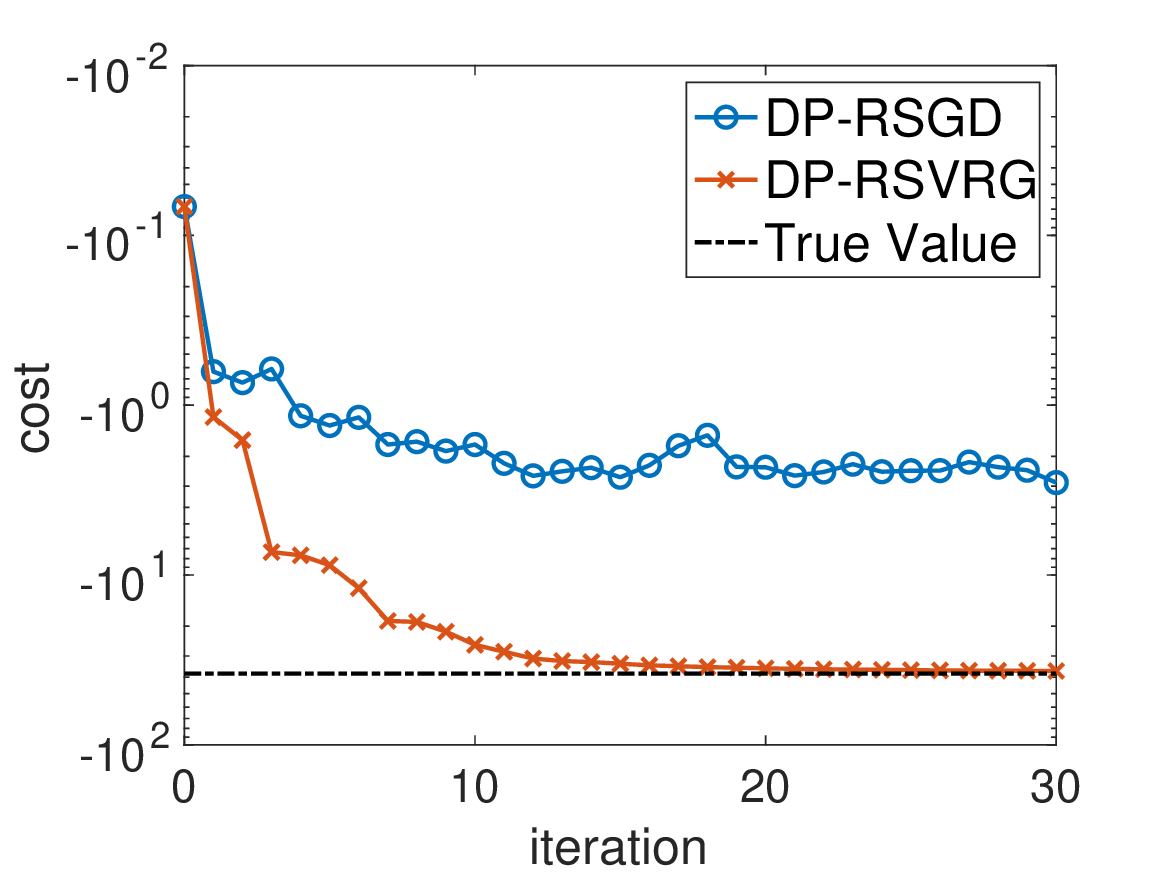}
\label{fig:NumExp:4-3}
}
\quad
\subfigure[$N=100,N_i=600$]{
\includegraphics[width=3.5cm]{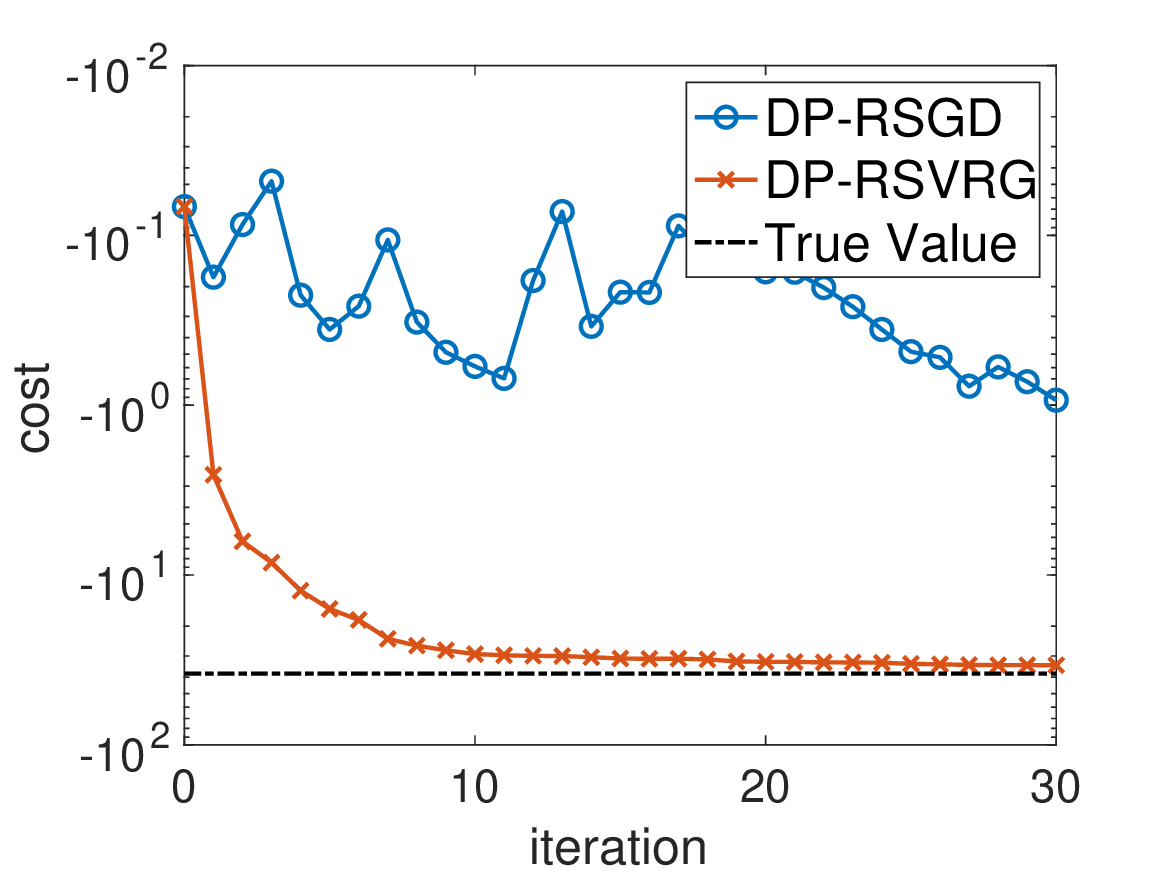}
\label{fig:NumExp:4-4}
}
\caption{Averaged results over $10$ tests for Problem~\eqref{NumExp:1} with MNIST dataset. The legends ``DP-RSGD'' and ``DP-SRSGD'' refer respectively to PriRFed-DP-RSGD and PriRFed-DP-RSVRG, and ``True Value'' is provided by Riemannian steepest descent method.}

\label{fig:NumExp:4}
\end{figure}


Given that Algorithm~\ref{alg:DP-RSVRG} employs the full gradient every $m$-step to improve stochastic gradient, while Algorithm~\ref{alg:DP-RSGD} relies only on stochastic gradient, it is reasonable to anticipate that PriRFed-DP-RSVRG converges more rapidly to a solution compared to PriRFed-DP-RSGD. This expectation is duly affirmed by Figure~\ref{fig:NumExp:4}. 
Furthermore, as per Theorem~\ref{th:conv:for DP-RSVRG}, an increase in the number of agents, denoted by $N$, corresponds to a smaller convergence bound. This outcome is further validated by Figure~\ref{fig:NumExp:4}. 


\subsection{Fr\'{e}chet mean computation over symmetric positive definite  matrix manifold}
The second problem is to compute the Fr\'{e}chet mean over the symmetric positive definite (SPD) matrix manifold of size $d\times d$, denoted by $\mathbb{S}_{++}^{d}$. The tangent space of $\mathbb{S}_{++}^d$ at $X$ is given by $\mathrm{T}_X \mathbb{S}_{++}^d=\{S\in \mathbb{R}^{d\times d}:S^T=S\}$. The Riemannian metric on $\mathbb{S}_{++}^d$ is chosen as the affine-invariant metric~\cite{Pen06, Bha07}: $\left<U,V\right>_X=\mathrm{trace}(UX^{-1}VX^{-1})$. The exponential map is given by $\mathrm{Exp}_X(U)=X^{1/2}\mathrm{expm}(X^{-1/2}UX^{-1/2})X^{1/2}$ with $\mathrm{expm(\cdot)}$ the principal matrix exponential.

Specifically, for $N$ sets of SPD matrices $\{Z_{1,1},\dots,Z_{1,N_1}\},\dots,\{Z_{N,1},\dots,Z_{N,N_N}\}\subset \mathbb{S}_{++}^d$, the Fr\'{e}chet mean of those SPD matrices is the solution of 
\begin{align} \label{NumExp:2}
	\argmin_{X\in \mathbb{S}_{++}^d}f(X)=\sum_{i=1}^N\frac{p_i}{N_i}\sum_{j=1}^{N_i}\|\mathrm{logm}(X^{-1/2}Z_{i,j}X^{-1/2})\|_F^2,
\end{align}
where $p_i=N_i/\sum_{i=1}^NN_i$, $D_i=\{Z_{i,1},\dots,Z_{i,N_i}\}$, $\mathrm{logm}(\cdot)$ is the principal matrix logarithm, $f_{i,j}(X)=\|\mathrm{logm}(X^{-1/2}Z_{i,j}X^{-1/2})\|_F^2$, and $f_i(X)=\frac{1}{N_i}\sum_{j=1}^{N_i}f_{i,j}(X)$. The Riemannian gradient of Problem~\eqref{NumExp:2} is given by $\mathrm{grad}f(X)=-2\sum_{i=1}^{N}\frac{p_i}{N_i}\sum_{j=1}^{N_i}X^{1/2}\mathrm{logm}(X^{-1/2}Z_{i,j}X^{-1/2})X^{1/2}$. 
Here we set the privacy parameters as $\epsilon=0.15$ and $\delta=10^{-4}$ for local training procedures, and the algorithm parameters as $s_t=1$, $K=3$, $m=10$, and $\tau=\tau_0=\tau_1=1.0$.

\paragraph{Synthetic data.} The synthetic samples are generated by following the steps in~\cite{RBS21}. The samples follow the Wishart distribution $W(I_d/d, d)$ with a diameter bound $D_{\mathcal{W}}$. For our simulations, we set $D_{\mathcal{W}}=1$ and $d=2$ and the number of all agents' samples is $5000$. We test two scenarios: $(N, N_i)=(20, 250), (N, N_i) = (50, 100)$ respectively. The results are reported in Figure~\ref{fig:NumExp:5} where the step size is set as $\alpha_t=0.05$. 
\begin{figure}[ht]
\centering
\subfigure[$N=20,N_i=250$]{
\includegraphics[width=3.5cm]{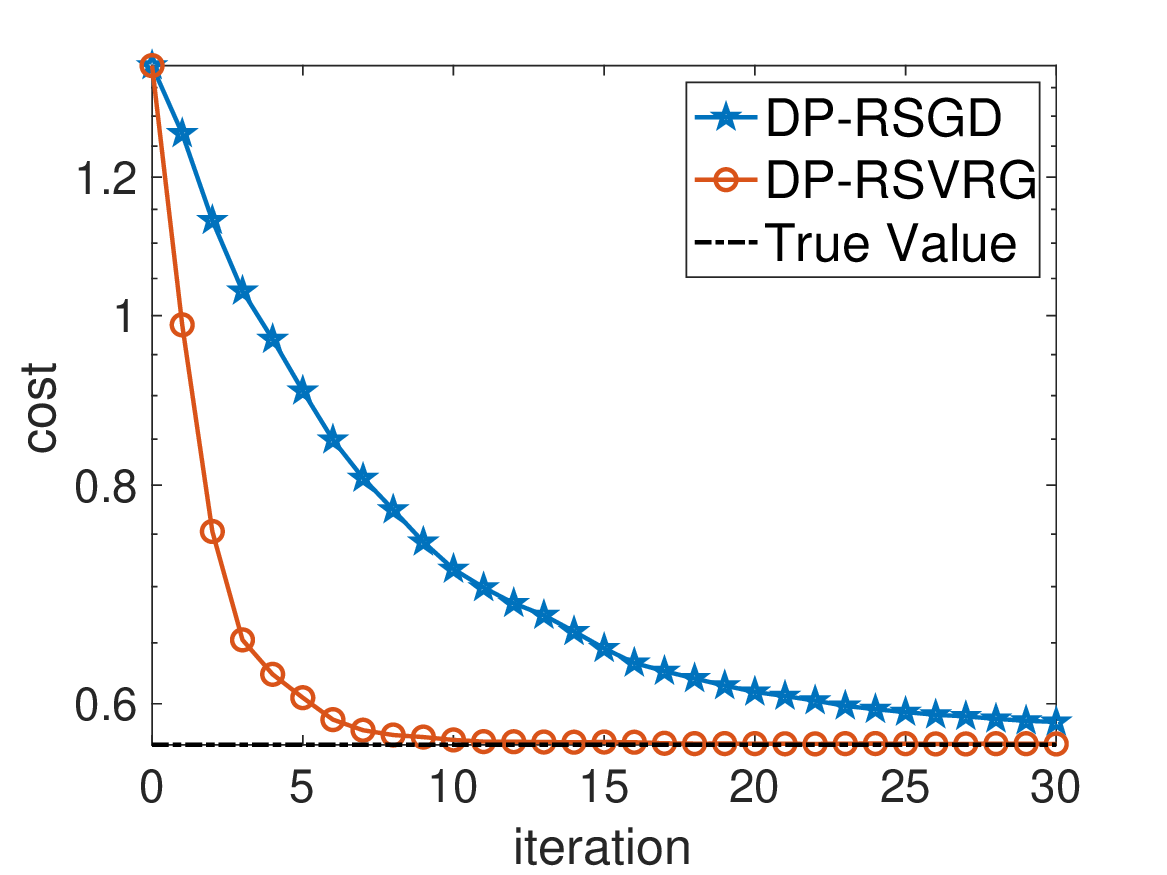}
\label{fig:NumExp:5-1}
}
\quad
\subfigure[$N=20,N_i=250$]{
\includegraphics[width=3.5cm]{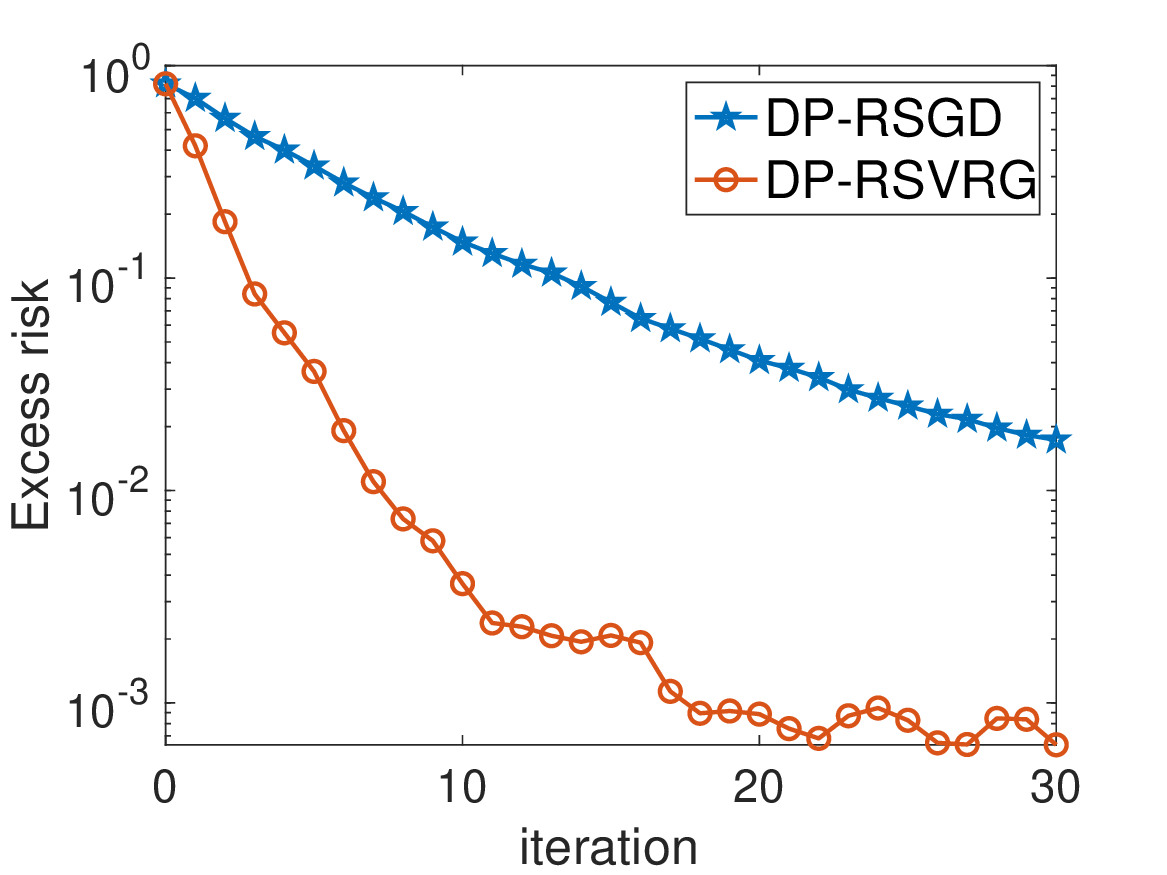}
\label{fig:NumExp:5-2}
}
\quad
\subfigure[$N=50,N_i=100$]{
\includegraphics[width=3.5cm]{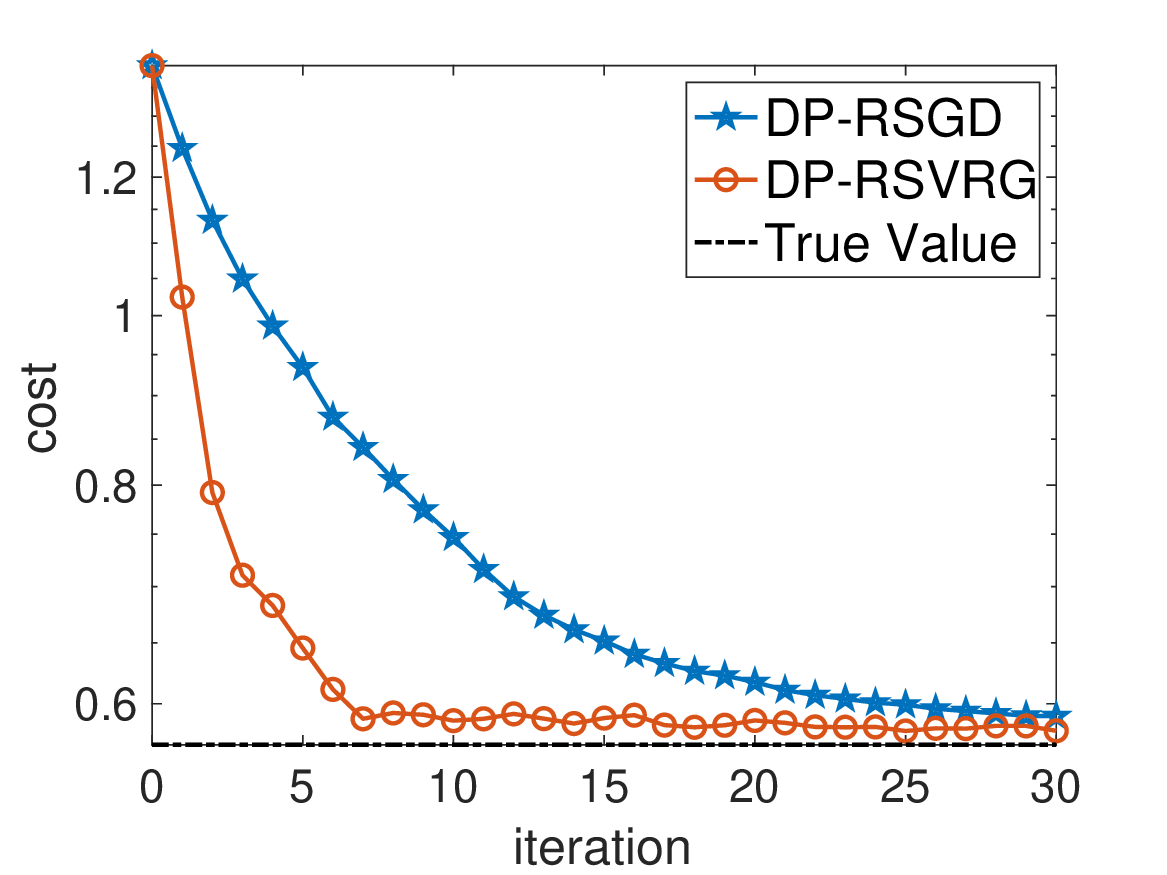}
\label{fig:NumExp:5-3}
}
\quad
\subfigure[$N=50,N_i=100$]{
\includegraphics[width=3.5cm]{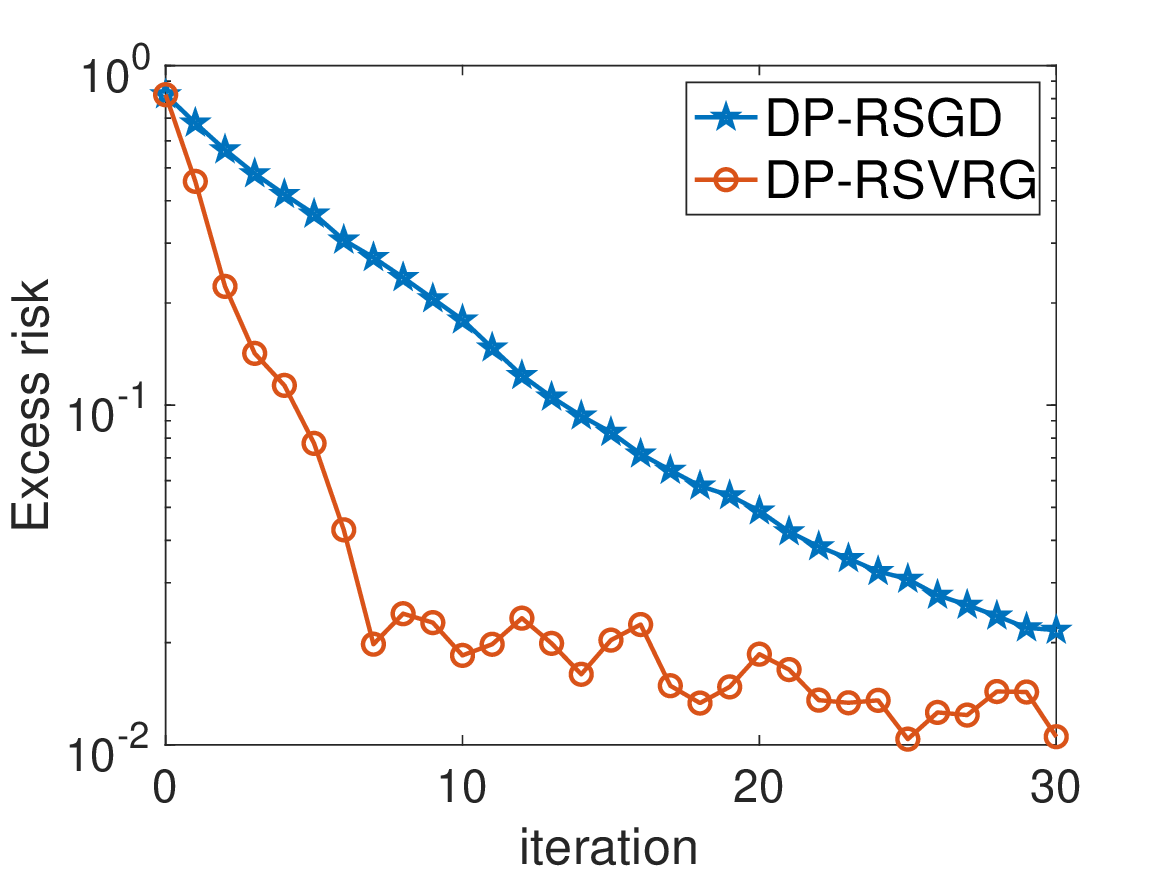}
\label{fig:NumExp:5-4}
}
\caption{Averaged results over $10$ tests for Problem~\eqref{NumExp:2} with synthetic dataset. The legends ``DP-RSGD'' and ``DP-SRSGD'' refer respectively to PriRFed-DP-RSGD and PriRFed-DP-RSVRG, and ``True Value'' is provided by Riemannian steepest gradient method. Excess risk is defined by $f(\tilde{x})-f(x^*)$, where $f(x^*)$ is given through steepest descent method.}

\label{fig:NumExp:5}
\end{figure}

\paragraph{PATHMNIST.}  The PATHMNIST dataset given in~\cite{YSW23} has $89996$ RGB images. Each image in PATHMNIST is transformed into a $9\times 9$ SPD matrix by the covariance descriptor~\cite{TPM06}. $20000$ images are randomly chosen from PATHMNIST for our tests with $(N, N_i) = (50, 400)$ and $(N, N_i) = (100, 200)$. The results are reported in Figure~\ref{fig:NumExp:6}, where the step size is set as $\alpha_t=0.3$.


\begin{figure}[ht]
\centering
\subfigure[$N=50,N_i=400$]{
\includegraphics[width=3.5cm]{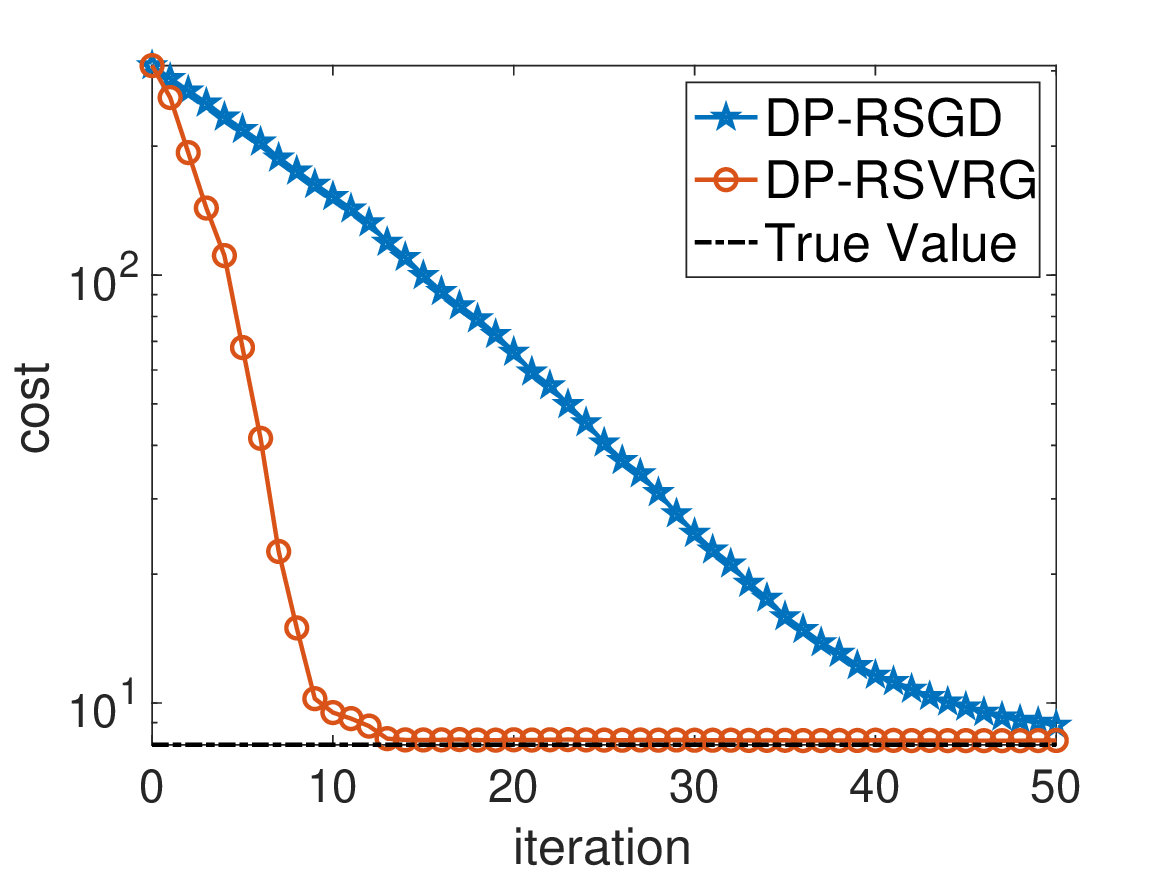}
\label{fig:NumExp:6-1}
}
\quad
\subfigure[$N=50,N_i=400$]{
\includegraphics[width=3.5cm]{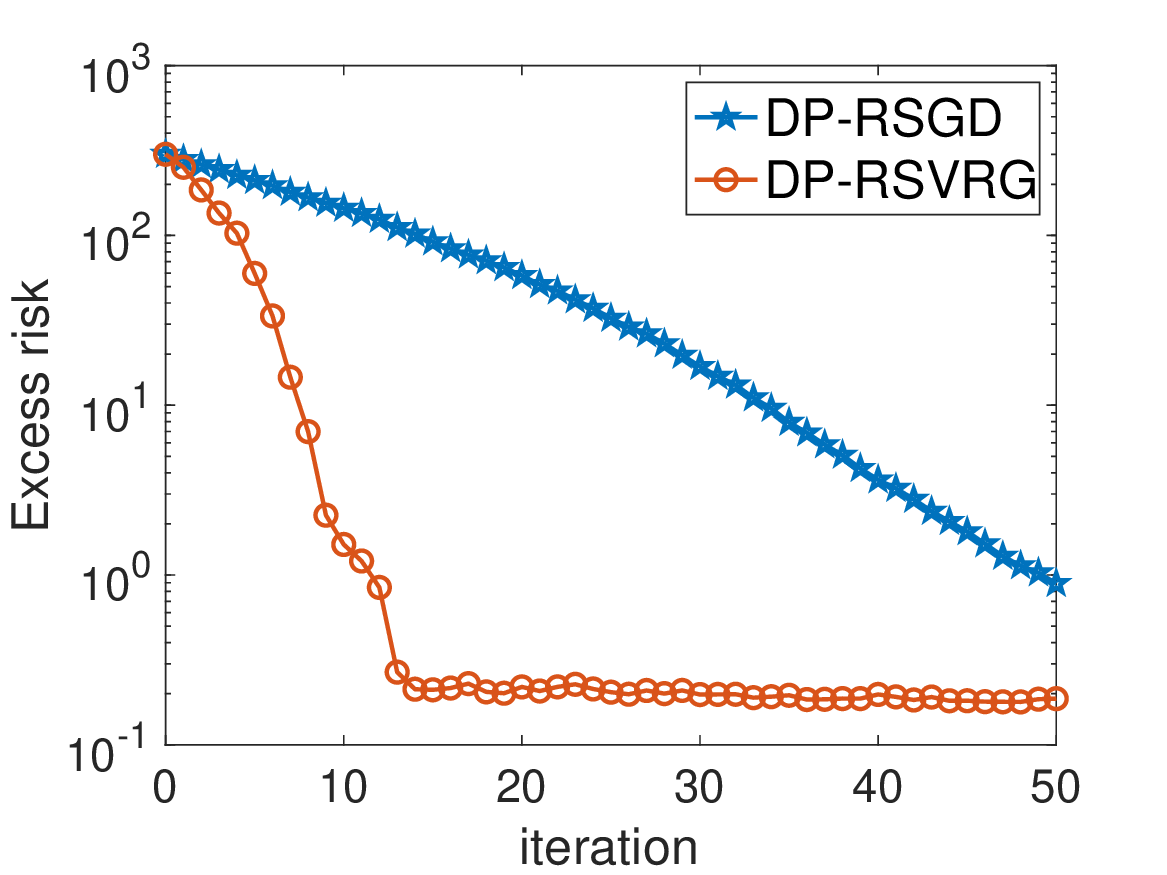}
\label{fig:NumExp:6-2}
}
\quad
\subfigure[$N=100,N_i=200$]{
\includegraphics[width=3.5cm]{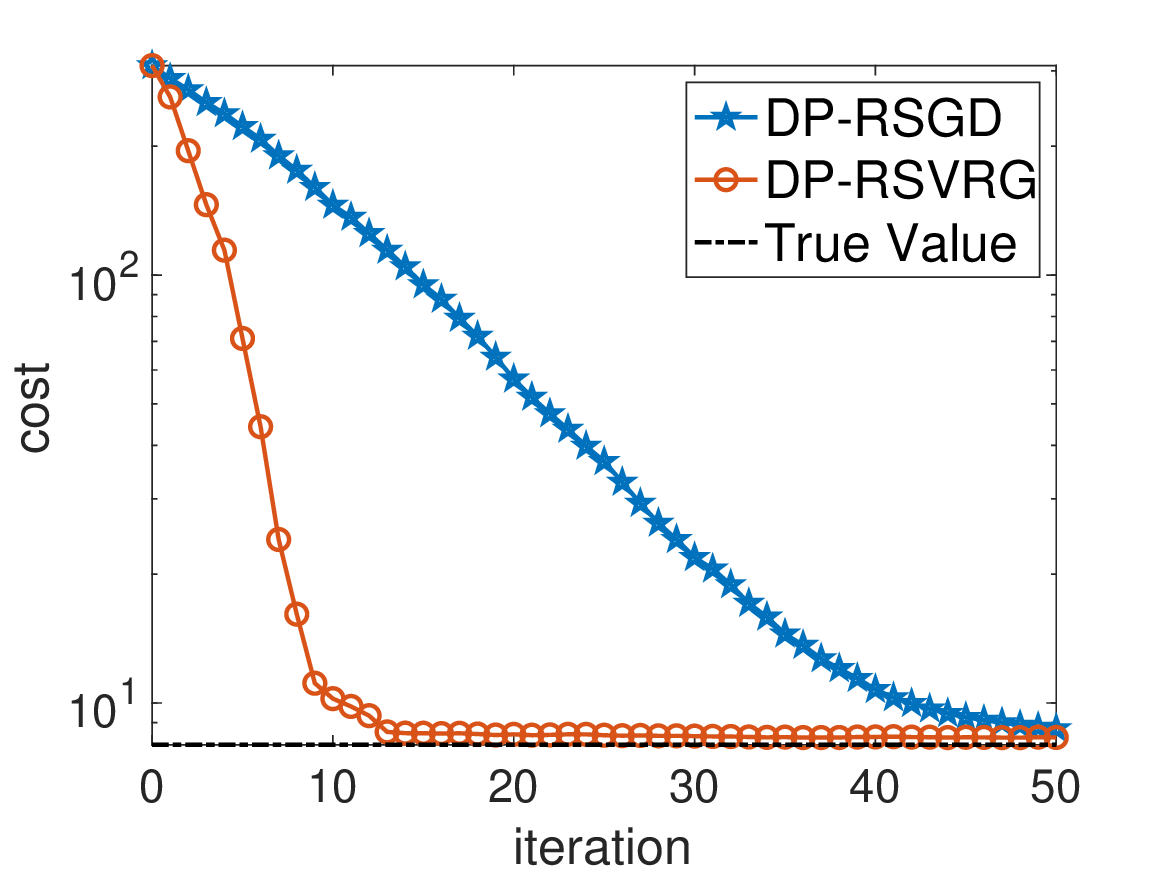}
\label{fig:NumExp:6-3}
}
\quad
\subfigure[$N=100,N_i=200$]{
\includegraphics[width=3.5cm]{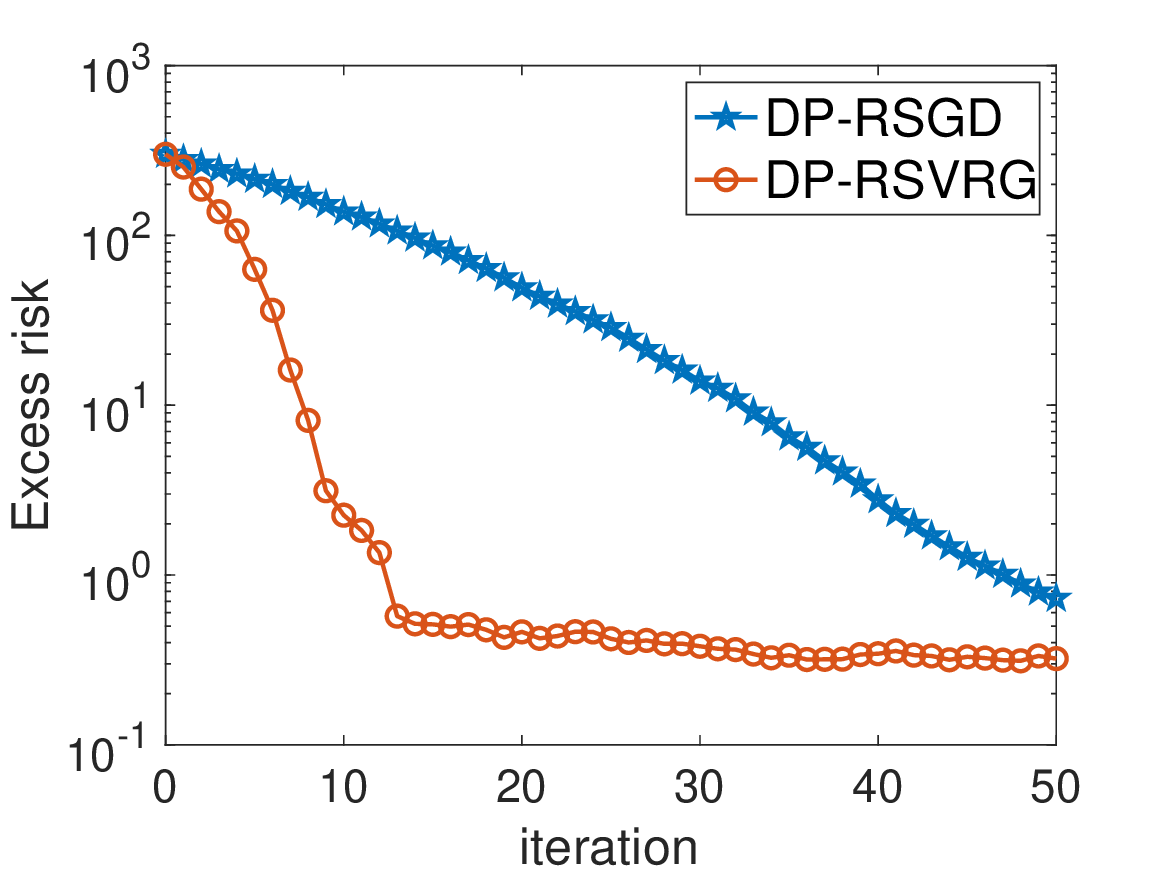}
\label{fig:NumExp:6-4}
}
\caption{Averaged results over $10$ tests for Problem~\eqref{NumExp:2} with PATHMNIST dataset. The legends ``DP-RSGD'' and ``DP-SRSGD'' refer respectively to PriRFed-DP-RSGD and PriRFed-DP-RSVRG, and ``True Value'' is provided by Riemannian steepest descent method. Excess risk is defined by $f(\tilde{x})-f(x^*)$, where $f(x^*)$ is given through steepest descent method.}

\label{fig:NumExp:6}
\end{figure}


It is shown in Figures~\ref{fig:NumExp:5} and~\ref{fig:NumExp:6} that PriRFed -DP-RSVRG consistently outperforms PriRFed-DP-RSGD in computing the Fr\'{e}chet mean of the SPD matrices. For a smaller number of local samples $N_i$ (which implies that the amount of added noise is larger), PriRFed-DP-RSGD is notably slower than PriRFed-DP-RSVRG. Such performance suggests that the convergence speed of PriRFed-DP-RSVRG is less sensitive to the added noise.



\subsection{Hyperbolic structured prediction}

Hyperbolic structured prediction (HSP) is defined on the hyperbolic manifold, which 
has great promise in representation learning of word semantics~\cite{NK17,TBG18}, graphs~\cite{CYRL19}, and images~\cite{KMUOL19}.
 The Lorentz hyperboloid model as one of equivalent represented models of hyperbolic manifold is defined by 
$
	\mathcal{H}^d:=\{x\in \mathbb{R}^{d+1}:\left<x,x\right>_{\mathcal{L}}=-1\}
$ with constant curvature $\kappa=-1$,
where $\left<x,y\right>_{\mathcal{L}}=x^Ty - 2x_1y_1$. For $x\in \mathcal{H}^d$, the tangent space at $x$ is given by 
$
	\mathrm{T}_x \mathcal{H}^d:=\{ v\in \mathbb{R}^{d+1}: \left<x,v\right>_{\mathcal{L}}=0 \}.
$
The Riemannian metric that we used here is given by $\left<u,v\right>_x= \left<u,v\right>_{\mathcal{L}}$ for any $u,v\in \mathrm{T}_x \mathcal{H}^d$. The exponential map, its inverse and the Riemannian distance function are respectively given by 
$
	\mathrm{Exp}_x(v) = \cosh(\|v\|_{\mathcal{L}})x + v\frac{\sinh(\|v\|_{\mathcal{L}})}{\|v\|_{\mathcal{L}}}
$,
$\mathrm{Exp}^{-1}_x(y)=\frac{\cosh^{-1}(-\left<x,y\right>_{\mathcal{L}})}{\sinh(\cosh^{-1}(-\left<x,y\right>_{\mathcal{L}}))}(y+ \left<x,y\right>_{\mathcal{L}}x)$, and  
$
	\mathrm{dist}(x,y) = \cosh^{-1}(-\left<x,y\right>_{\mathcal{L}}).
$

Hyperbolic structured prediction has been used to infer the taxonomy embeddings for unknown words~\cite{MRC20}. Specifically, consider a set of training pairs $D_i = \{\{(w_{i,j},y_{i,j})\}_{j=1}^{N_i}\}_{i=1}^N$, where $w_{i,j}\in \mathbb{R}^r$ are the features and $y_{i,j} \in \mathcal{H}^d$ are the hyperbolic embeddings of the class of $w_{i,j}$. Then for a test sample $w$, the goal of HSP is to predict its hyperbolic embeddings by solving the following problem: 
\begin{equation}
\begin{aligned} \label{NumExp:3}
	h(x):=\argmin_{x\in \mathcal{H}^d} f(x) = \sum_{i=1}^Np_i\sum_{j=1}^{N_i}\frac{1}{N_i}\alpha_{i,j}(w)\mathrm{dist}^2(x,y_{i,j}),
\end{aligned}
\end{equation}
where $p_i=N_i/\sum_{i=1}^NN_i$, 
$\alpha_i(w)=(\alpha_{i,1}(w),\dots,\alpha_{i,N_i}(w))^T \in \mathbb{R}^{N_i}$ can be computed as $\alpha_i(w)=(K_i+\gamma I)^{-1}K_{i,w}$, $\gamma>0$ is the regularization parameter and $K_i\in \mathbb{R}^{N_i \times N_i}$, $K_{i,w}\in \mathbb{R}^{N_i}$ are given by $(K_i)_{l,j}=k(w_{i,l},w_{i,j})$ and $(K_{i,w})_j=k(w_{i,j},w)$ for the RBF kernel function $k(w,w')=\mathrm{exp}(-\|w-w'\|_2^2/(2\bar{v})^2)$. Note that~\eqref{NumExp:3} is in the form of~\eqref{prob} with $f_i(x)=\sum_{j=1}^{N_i}\frac{1}{N_i}a_{i,j}(w)\mathrm{dist}^2(x,y_{i,j})$, and $f_{i,j}(x)=\alpha_{i,j}(w)\mathrm{dist}^2(x,y_{i,j})$. 
The gradient of $f_{i,j}$ is given by $\mathrm{grad}f_{i,j}(x)=2\alpha_{i,j}(w)\mathrm{Exp}^{-1}_{x}(y_{i,j})$.

We note that $\|\mathrm{grad}f_{i,j}(x)\|_x=\|2\alpha_{i,j}(w)\mathrm{Exp}_x^{-1}(y_{i,j})\|_x\le 2|\alpha_{i,j}(w)|\mathrm{dist}(x,y_{i,j})\le 2|\alpha_{i,j}(w)|M$, which show that the geodesic Lipschitz constant is $L_f=2\max_{i,j}|\alpha_{i,j}(w)|M$. On the other hand, from~\cite[Lemma~2]{AOBL20} it follows that $\|\mathrm{Hess}\,\mathrm{dist}^2(x,y_{i,j})[v]\|_x\le L_1\|v\|_x$ for any $v\in \mathrm{T}_x \mathcal{H}^d$ with $L_1=\frac{\sqrt{|\kappa_{\min}|}M}{\tanh(\sqrt{|\kappa_{\min}|}M)}=\frac{M}{\tanh(M)}$, implying that $\|\mathrm{Hess}f_{i,j}(x)\|_x\le |\alpha_{i,j}(w)|L_1$. Therefore, the smooth constant $L_g$ is bounded as $L_g\le \max_{i,j}|\alpha_{i,j}(w)|\frac{M}{\tanh(M)}$ and we set $\tau=\tau_0=\tau_1=L_f$.

\paragraph{WordNet.} WordNet dataset~\cite{Mil95} is used to conduct the experiment of inferring hyperbolic embeddings. Following~\cite{NK17}, the pretrained hyperbolic embeddings on $\mathcal{H}^2$ of the mammals subtree with the transitive closure containing $n = 1180$ nodes (words) and $6540$ edges (hierarchies) are used\footnote{It is referred to website https://github.com/facebookresearch/poincare-embeddings.}. The features are stemmed from Laplacian eigenmap~\cite{BN03} to dimension $r=3$ of the adjacency matrix formed by the edges. In other words, we obtained $\{(w_i,y_i)\}_{i=1}^n\subset \mathbb{R}^3\times \mathcal{H}^2$. This setting is in line with the work in~\cite{HMJG22}. In the experiments, the word ``primate'' is selected as the test sample, and the remainder is used to train. Therefore, the hyperbolic embedding of the word ``primate'' is known and is viewed as the true embedding. We conduct the experiment under the scenarios: $N=9$ (thus $N_i=131$ for agent $i$), and set the problem parameters as $\gamma=10^{-5}$, $\bar{v}=0.3$, $L_f=\max_{i,j}|\alpha_{i,j}(w)|M$, $L_g=\max_{i,j}|\alpha_{i,j}(w)|\frac{M}{\tanh(M)}$ with $M=4$ (estimated from samples),  the privacy parameters as $\epsilon=0.15$, $\delta=10^{-4}$, and the algorithm parameters as $s_t=1$, $K=3$, $m=5$, and $\alpha_t=0.3$. The Riemannian distance averaged over 10 runs between the prediction and the true embedding on $\mathcal{H}^2$ and the values of cost averaged over 10 runs against the number of iterations are plotted in Figure~\ref{fig:NumExp:7}, respectively. 

\begin{figure}[ht]
\centering
\subfigure[]{
\includegraphics[width=0.35\textwidth]{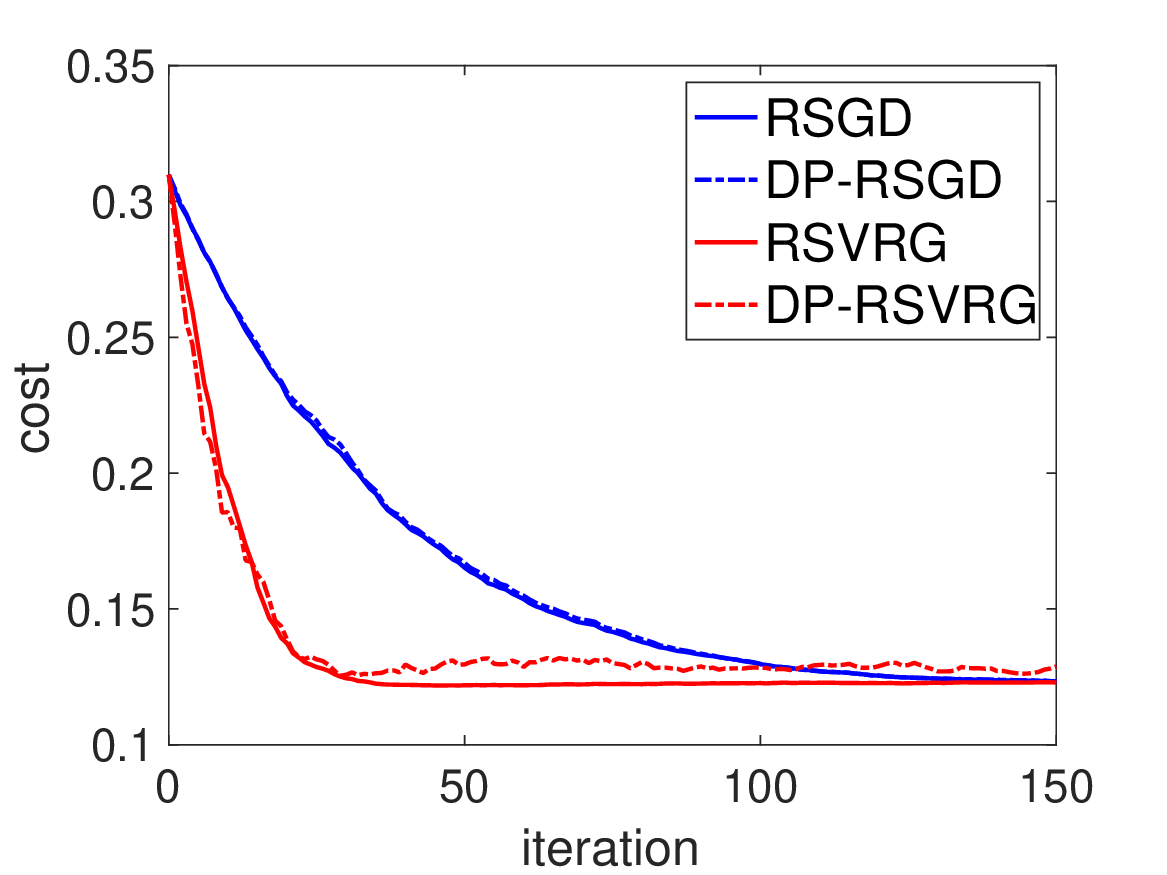}
\label{fig:NumExp:7-1}
}
\quad
\subfigure[]{
\includegraphics[width=0.35\textwidth]{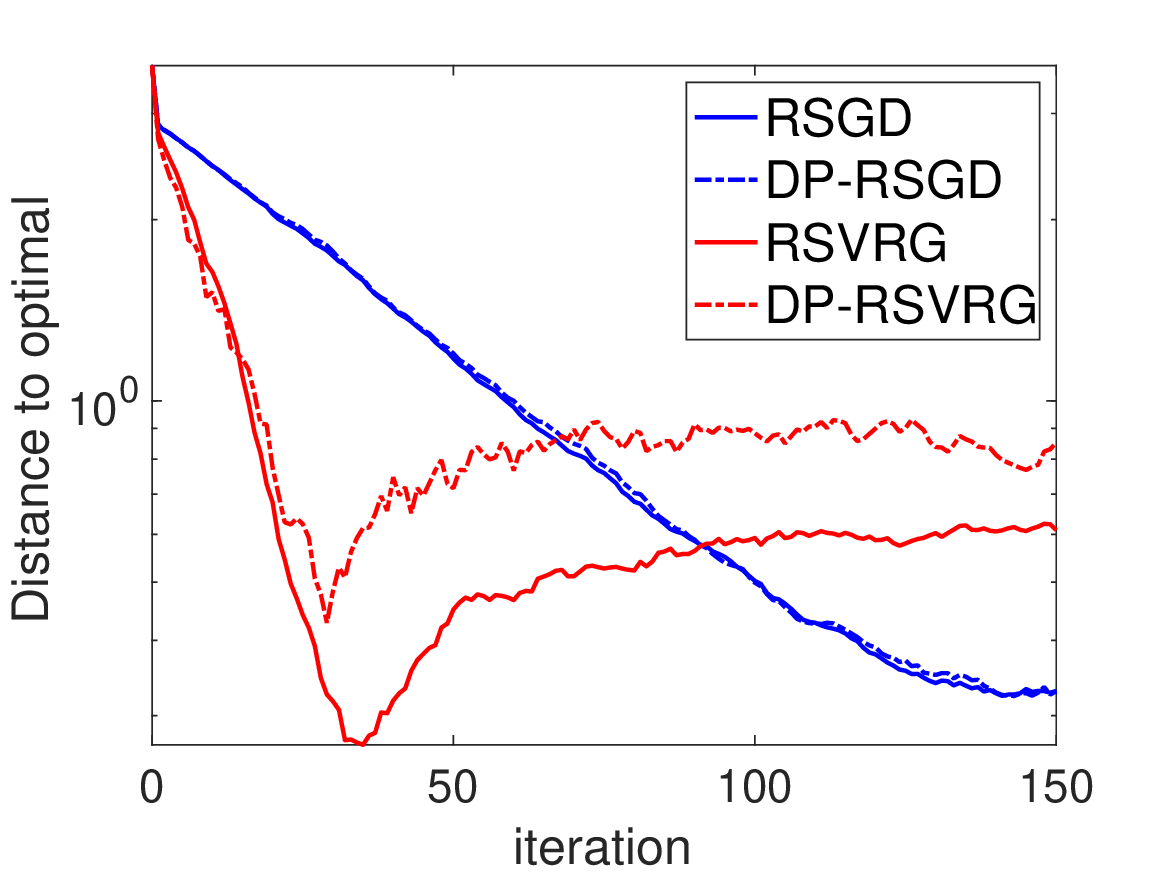}
\label{fig:NumExp:7-2}
}
\caption{Averaged results over $10$ tests for Problem~\eqref{NumExp:3} with WordNet dataset. The legends ``RSGD'', ``DP-RSGD'', ``RSVRG'' and ``DP-SRSGD'' refer respectively to PriRFed-RSGD, PriRFed-DP-RSGD, PriRFed-RSVRG and PriRFed-DP-RSVRG, respectively. (a) The values of cost against iterations. (b) The Riemannian distance against iterations.}

\label{fig:NumExp:7}
\end{figure}

We observe from Figure~\ref{fig:NumExp:7} that PriRFed-DP-RSGD and PriRFed-DP-RSVRG converges to the pretrained embedding while making the outputs private under $(\epsilon,\delta)$-differential privacy. Moreover, as expected, PriRFed-DP-RSVRG outperforms PriRFed-DP-RSGD in the sense of the number of iterations. 
It is worth noting that PriRFed with DP-RSVRG and RSVRG converge rapidly to acceptable solutions, and subsequently overfitting leads to that the iterates leave solutions.
To further illustrate the validation of the two combinations, we plot the hyperbolic embeddings on the Poincare disk in Figure~\ref{fig:NumExp:8}. We see from Figure~\ref{fig:NumExp:8-2} that the predictions are close to the optimal point (the true embedding), reserving the hierarchical structures as well as ensuring privacy.

\begin{figure}[H]
\centering
\subfigure[]{
\includegraphics[width=0.4\textwidth]{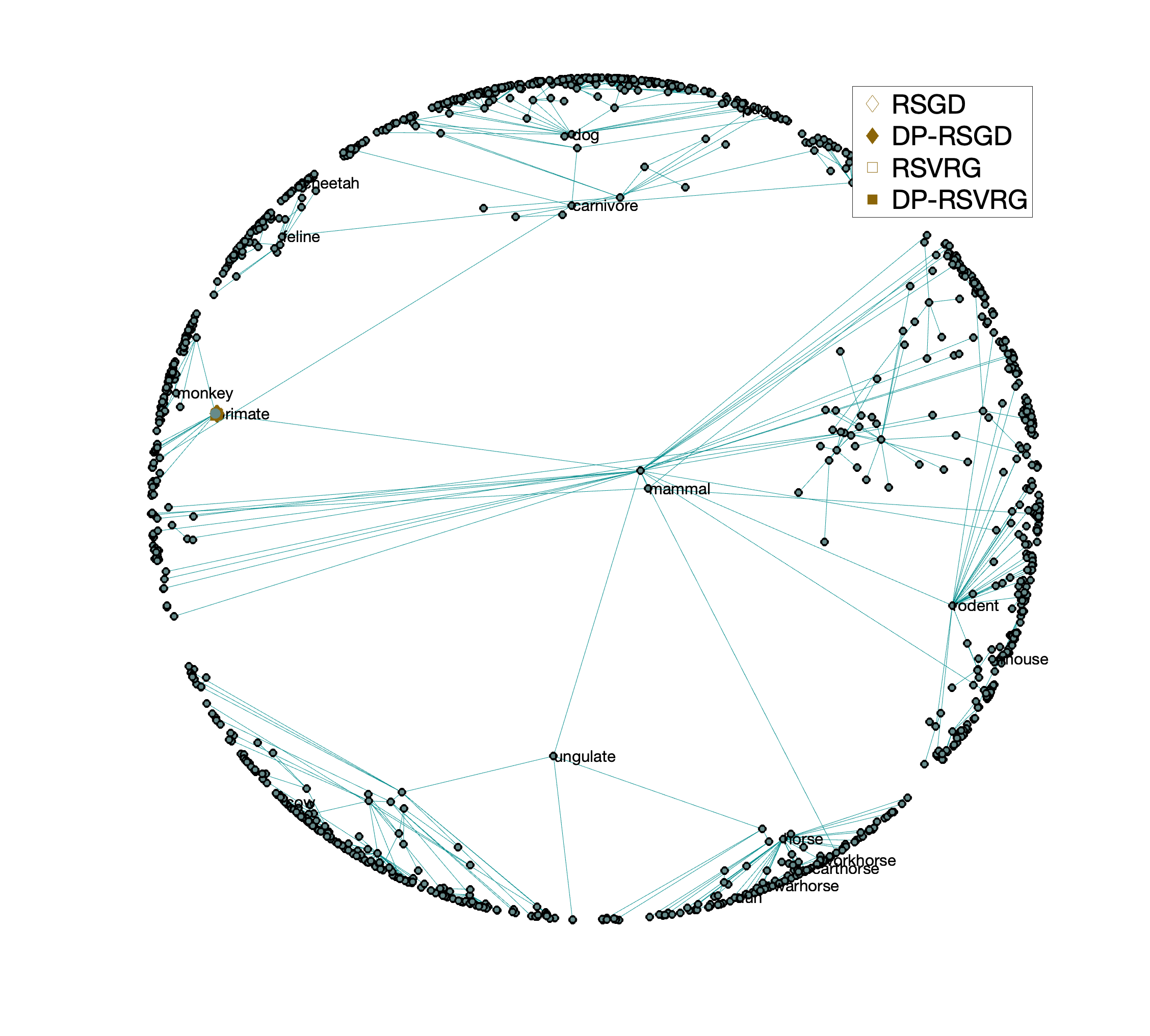}
\label{fig:NumExp:8-1}
}
\quad
\subfigure[]{
\includegraphics[width=0.4\textwidth]{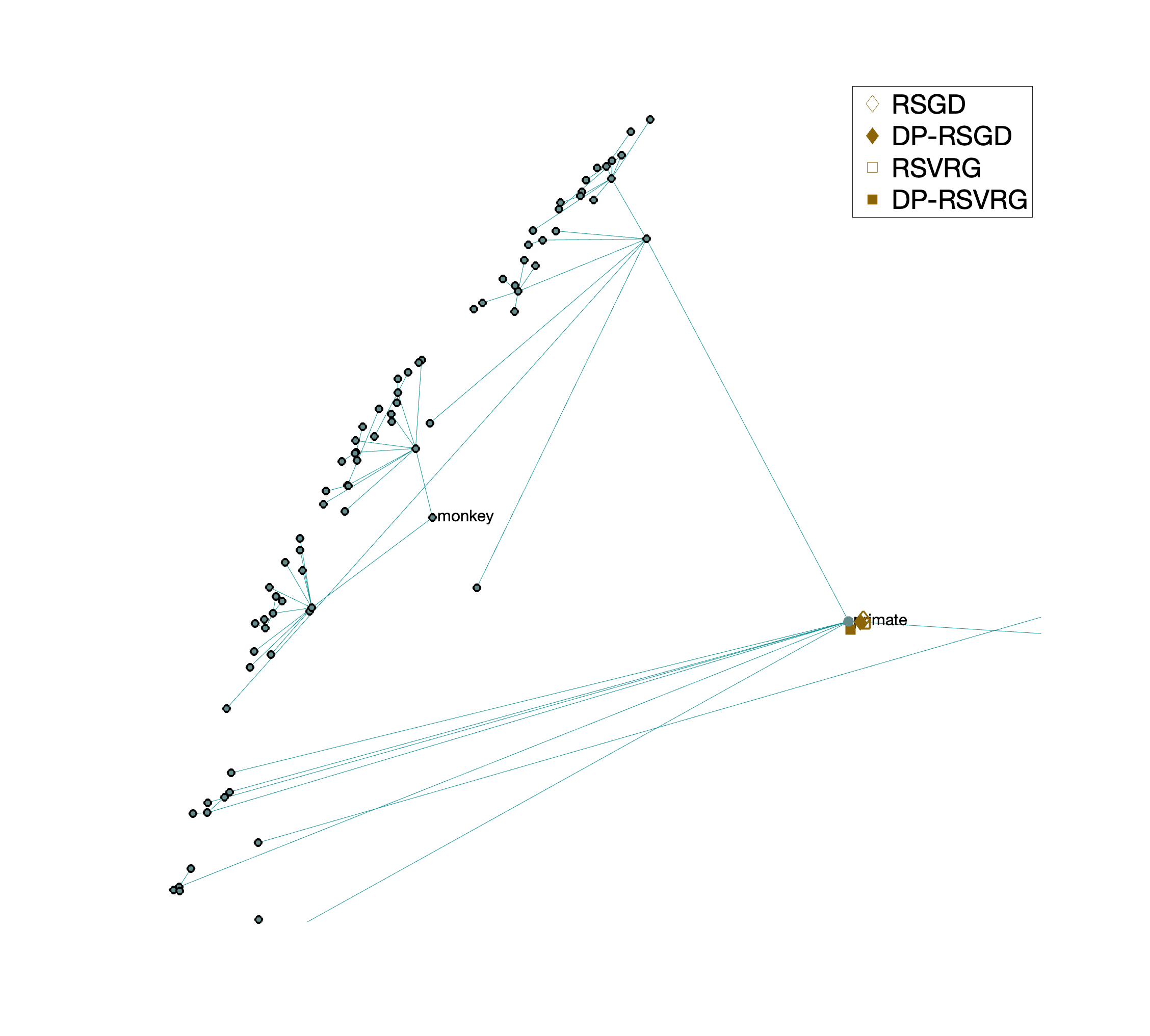}
\label{fig:NumExp:8-2}
}
\caption{The best hyperbolic structured predictions (Problem~\eqref{NumExp:3}) on ``primate'' during the $10$ tests with WordNet dataset. Here ``RSGD'', ``DP-RSGD'', ``RSVRG'' and ``DP-RSVRG'' refer respectively to PriRFed-RSGD, PriRFed-DP-RSGD, PriRFed-RSVRG and PriRFed-DP-RSVRG, respectively. (a) Full embeddings. (b) Partial embedddings for ``primate''.}

\label{fig:NumExp:8}
\end{figure}

\section{Conclusion} \label{sec:concl} 
In this paper, we generalized the framework of federated learning with differential privacy to Riemannian manifolds, proposing a Riemannian federated learning framework obeying differential privacy, termed PriRFed. We proved that with the assumption of local training procedures being differentially private, the proposed Riemannian federated learning framework satisfies differential privacy. This eliminates the necessity of the trusted server while ensuring that the final learned model is privacy-preserved. 
Additionally, we analyzed the convergence properties when the local training procedure is specified as DP-RSGD and DP-RSVRG. Extensive numerical experiments also reflect the validity and effectiveness of the proposed method. 
Although this marks an initial endeavor in privacy-preserving Riemannian federated learning, there is potential for enhancement, such as considering more efficient local training procedures and aggregation.

\appendix 

\section{Proofs for Section~\ref{sec:conv:DP-RSGD}} \label{appendix1}

\subsection{Proof of Theorem~\ref{th:conv:for DP-RSGD1}}

\begin{proof}
	With $K=1$, $b_i=N_i$, and $x_{i,0}^{(t)}=x^{(t)}$ it follows that
	$
		\eta_0^{(t)} = \mathrm{grad}f_i(x^{(t)}) + \varepsilon_{i,0}^{(t)},
	$
	and
	$
		\tilde{x}^{(t+1)}_i = x_{i,1}^{(t)}=\mathrm{Exp}_{x^{(t)}}(-\alpha\eta_0^{(t)})=\mathrm{Exp}_{x^{(t)}}(-\alpha\mathrm{grad}f_i(x^{(t)})-\alpha\varepsilon_{i,0}^{(t)}),
	$
	which is equivalent to
	$
	\mathrm{Exp}_{x^{(t)}}^{-1}(\tilde{x}_i^{(t+1)})=-\alpha\mathrm{grad}f_i(x^{(t)})-\alpha\varepsilon_{i,0}^{(t)}.
	$
	By aggregating formulation and $s_t=N$, we have
	$
		x^{(t+1)} = \mathrm{Exp}_{x^{(t)}}\left(\sum_{i=1}^Np_i\mathrm{Exp}_{x^{(t)}}^{-1}\left(\tilde{x}_i^{(t+1)}\right)\right),
	$
	that is,
	$
		\mathrm{Exp}_{x^{(t)}}^{-1}(x^{(t+1)}) = \sum_{i=1}^Np_i(-\alpha\mathrm{grad}f_i(x^{(t)})-\alpha\varepsilon_{i,0}^{(t)}) = -\alpha\mathrm{grad}f(x^{(t)}) - \alpha\sum_{i=1}^Np_i\varepsilon_{i,0}^{(t)}.
	$
	Since $K=1$, we omit the subscript $k$ for $\varepsilon_{i,k}^{(t)}$. Using $L_g$-smoothness of $f$ (Definition~\ref{def:geodsmooth}), we have
	\begin{equation} \label{th:conv:for DP-RSGD1:1}
	\begin{aligned}
		& f(x^{(t+1)}) - f(x^{(t)}) \le \left< \mathrm{Exp}^{-1}_{x^{(t)}}(x^{(t+1)}),\mathrm{grad}f(x^{(t)}) \right> + \frac{{L}_g}{2}d^2(x^{(t+1)},x^{(t)}) \\
		&= -\alpha\left<\mathrm{grad}f(x^{(t)}), \mathrm{grad}f(x^{(t)})\right> - \alpha\left<\sum_{i=1}^Np_i\varepsilon_{i}^{(t)}, \mathrm{grad}f(x^{(t)})\right> + \frac{{L}_g}{2}\alpha^2\|\mathrm{grad}f(x^{(t)})+\sum_{i=1}^Np_i\varepsilon_{i}^{(t)}\|^2 \\
		&= (\frac{{L}_g}{2}\alpha^2 -\alpha) \|\mathrm{grad}f(x^{(t)})\|^2 + \frac{{L}_g}{2}\alpha^2\|\sum_{i=1}^Np_i\varepsilon_i^{(t)}\|^2 + ({L}_g\alpha^2-\alpha)\left<\sum_{i=1}^Np_i\varepsilon_i^{(t)}, \mathrm{grad}f(x^{(t)})\right>,  
	\end{aligned}
	\end{equation}
in which taking expectation on both sides over $\varepsilon_i^{(t)}, \forall t$ yields 
\begin{equation} \label{th:conv:for DP-RSGD1:xx}
\begin{aligned}
	&\mathbb{E}[f(x^{(t+1)})-f(x^{(t)})] \\
 \le& (\frac{{L}_g}{2}\alpha^2-\alpha)\mathbb{E}[\|\mathrm{grad}f(x^{(t)})\|^2] + \frac{{L}_g}{2}\alpha^2 \mathbb{E}[\|\sum_{i=1}^Np_i\varepsilon_i^{(t)}\|^2] + ({L}_g\alpha^2-\alpha)\mathbb{E}\left[\left<\sum_{i=1}^Np_i\varepsilon_i^{(t)},\mathrm{grad}f(x^{(t)})\right>\right]  \\ 
 =& (\frac{{L}_g}{2}\alpha^2-\alpha)\mathbb{E}[\|\mathrm{grad}f(x^{(t)})\|^2] + \frac{{L}_g}{2}\alpha^2 \sum_{i=1}^Np_i^2 \mathbb{E}[\|\varepsilon_i^{(t)}\|^2]  = (\frac{{L}_g}{2}\alpha^2-\alpha)\mathbb{E}[\|\mathrm{grad}f(x^{(t)})\|^2] + \frac{{L}_gd\alpha^2}{2}\sum_{i=1}^Np_i^2\sigma_i^2, 
\end{aligned}
\end{equation}
since $\mathbb{E}[\varepsilon_i^{(t)}]=0$, $\mathbb{E}[\left<\varepsilon_i^{(t)},\mathrm{grad}f(x^{(t)})\right>]=0$ (as $\varepsilon_i^{(t)}$ is independent), and $\mathbb{E}[\|\varepsilon_i^{(t)}\|^2] = d\sigma_i^2$. With $\alpha\le{1}/{{L}_g}$, we have
$
	\frac{\alpha}{2} \mathbb{E}[\|\mathrm{grad}f(x^{(t)})\|^2] \le \mathbb{E}[f(x^{(t)}) - f(x^{(t+1)})] + \frac{dL_g\alpha^2}{2}\sum_{i=1}^Np_i^2\sigma_i^2.
$
Therefore,
$
		\sum_{t=0}^{T-1}\mathbb{E}[\|\mathrm{grad}f(x^{(t)})\|^2] \le \frac{2}{\alpha} \mathbb{E}[f(x^{(0)})-f(x^{(T)})] + {TdL_g\alpha} \sum_{i=1}^Np_i^2\sigma_i^2 \le \frac{2}{\alpha} \mathbb{E}[f(x^{(0)})-f(x^{*})] + \frac{dTL_g\alpha}{ (\sum_{i=1}^NN_i)^2}\sum_{i=1}^NN_i^2\sigma_i^2.
$
	due to $p_i=N_i/\sum_{i=1}^NN_i$.
\end{proof}

\subsection{Proof of Theorem~\ref{th:conv:for DP-RSGD1_}}
\begin{proof}
	Using~\eqref{th:conv:for DP-RSGD1:xx} together with conditions~\eqref{th:conv:for DP-RSGD1_:1}, we have 
$
	\mathbb{E}[f(x^{(t+1)}) - f(x^{(t)})] \le -\alpha_t(1-\delta)\mathbb{E}[\|\mathrm{grad}f({x}^{(t)})\|^2] + \frac{L_gd\alpha_t^2}{2(\sum_{i=1}^NN_i)^2}\sum_{i=1}^NN_i^2\sigma_i^2.
$
Summing the above inequality over $t=0,1,\dots,T-1$ and dividing the both sides by $(1-\delta)$ yields 
\begin{align} \label{th:conv:for DP-RSGD1_:xx}
	\mathbb{E}[\sum_{t=0}^{T-1}\alpha_t\|\mathrm{grad}f(x^{(t)})\|^2 ] \le \frac{\mathbb{E}[f(x^{(0)}) - f(x^*)]}{(1-\delta)}+\sum_{t=0}^{T-1}\frac{L_gd\alpha_t^2}{2(1-\delta)(\sum_{i=1}^NN_i)^2}\sum_{i=1}^NN_i^2\sigma_i^2,
\end{align}
where we use $f(x^{(0)}) - f(x^{(t)})\le f(x^{(0)})-f(x^*)$. 
Dividing~\eqref{th:conv:for DP-RSGD1_:xx} by $\sum_{t=0}^{T-1}\alpha_t$ on both sides gives rise to 
$
	\mathbb{E}\left[ \sum_{t=0}^{T-1}\frac{\alpha_t}{\sum_{t=0}^{T-1}\alpha_t} \|\mathrm{grad}f(x^{(t)})\|^2 \right] \le 	 \frac{\mathbb{E}[f(x^{(0)}) - f(x^*)]}{(1-\delta)\sum_{t=0}^{T-1}\alpha_t}+\sum_{t=0}^{T-1}\frac{L_gd\alpha_t^2}{2(1-\delta)(\sum_{t=0}^{T-1}\alpha_t)(\sum_{i=1}^NN_i)^2}\sum_{i=1}^NN_i^2\sigma_i^2,
$
which implies 
\begin{equation} \label{eq:01}
\lim_{T\rightarrow \infty}\mathbb{E}\left[ \sum_{t=0}^{T-1}\frac{\alpha_t}{\sum_{t=0}^{T-1}\alpha_t} \|\mathrm{grad}f(x^{(t)})\|^2 \right] = \lim_{T\rightarrow \infty}  \sum_{t=0}^{T-1}\frac{\alpha_t}{\sum_{t=0}^{T-1}\alpha_t} \mathbb{E}\left[\|\mathrm{grad}f(x^{(t)})\|^2 \right] =0
\end{equation}
since $\lim_{T\rightarrow \infty}\frac{1}{\sum_{t=0}^{T-1}\alpha_t}=0$, and $\lim_{T\rightarrow \infty}\frac{\sum_{t=0}^{T-1}\alpha_t^2}{\sum_{t=0}^{T-1}\alpha_t}=0$. Suppose $\liminf_{T \rightarrow \infty} \mathbb{E}\left[ \|\mathrm{grad}f(x^{(t)})\|^2 \right] \neq 0$. Then there exists a positive constant $\epsilon>0$ such that for all $t$, inequality $\mathbb{E}\left[ \|\mathrm{grad}f(x^{(t)})\|^2 \right] > \epsilon$ holds. It follows that
$
\sum_{t=0}^{T-1}\frac{\alpha_t}{\sum_{t=0}^{T-1}\alpha_t} \mathbb{E}\left[\|\mathrm{grad}f(x^{(t)})\|^2 \right] \geq \sum_{t=0}^{T-1}\frac{\alpha_t}{\sum_{t=0}^{T-1}\alpha_t} \epsilon = \epsilon > 0,
$
which conflicts with~\eqref{eq:01}.
\end{proof}

\subsection{Proof of Theorem~\ref{th:conv:for DP-RSGD3}}

\begin{lemma}[{\cite[Corollary~8]{ZS16}}]
\label{lem:App1}
For any Riemannian manifold $\mathcal{M}$ where the sectional curvature is lower bounded by $\kappa_{\min}$ and any point $x,x^{(t)}\in \mathcal{M}$, the update $x^{(t+1)}=\mathrm{Exp}_{x^{(t)}}(-\alpha g^{(t)})$ with $g^{(t)}\in \mathrm{T}_{x^{(t)}}\mathcal{M}$ satisfies 
\[
\left<-g^{(t)},\mathrm{Exp}_{x^{(t)}}^{-1}(x)\right> \le \frac{1}{2\alpha}(\mathrm{dist}^2(x,x^{(t)})-\mathrm{dist}^2(x,x^{(t+1)})) + \frac{\zeta(\kappa_{\min},\mathrm{dist}(x,x^{(t)}))\alpha}{2}\|g^{(t)}\|^2, 
\]
where $\zeta(\kappa,c)=\frac{\sqrt{|\kappa|c}}{\tanh(\sqrt{|\kappa|}c)}$.
\end{lemma}
Now we are ready to prove Theorem~\ref{th:conv:for DP-RSGD3}.
\begin{proof}
 By Lemma~\ref{lem:App1}, it holds that 
	\begin{equation}
	\begin{aligned}			
		-\alpha \left<\sum_{i=1}^N p_i \mathrm{grad}f_i(x^{(t)})+p_i\varepsilon_i^{(t)},\mathrm{Exp}_{x^{(t)}}^{-1} (x)\right> &\le \frac{1}{2}(\mathrm{dist}^2(x^{(t)},x) - \mathrm{dist}^2(x^{(t+1)},x)) \\ 
		& \quad + \frac{\zeta \alpha^2}{2}\left\|\sum_{i=1}^Np_i \mathrm{grad}f_i(x^{(t)}) + p_i \varepsilon_i^{(t)} \right\|^2. 
	\end{aligned}
	\end{equation}
	Let $\Delta_i^t = f_i(x^{(t)})-f_i(x^*)$, and $\Delta^t = \sum_{i=1}^Np_i\Delta_i^{t}=f(x^{(t)})-f(x^*)$ where $x^*=\argmin_{x\in \mathcal{M}}f(x)$. It follows from the geodesic convexity of $f_i$ that $\Delta_i^{t}\le -\left<\mathrm{grad}f_i(x^{(t)}),\mathrm{Exp}_{x^{(t)}}^{-1}(x^*)\right>$. 
	Summing this inequality over $i$ from $1$ to $N$ yields 
	\begin{equation} \label{th:conv:for DP-RSGD3:1}
		\begin{aligned}
		\Delta^t &\le -\left<\sum_{i=1}^Np_i \mathrm{grad}f_i(x^{(t)}), \mathrm{Exp}_{x^{(t)}}^{-1}(x^*)\right> \\ 
		& \le \frac{1}{\alpha}\left<\sum_{i=1}^Np_i\varepsilon_i^{(t)},\mathrm{Exp}_{x^{(t)}}^{-1}(x^*)\right> + \frac{1}{2\alpha}(\mathrm{dist}^2(x^{(t)},x^*)-\mathrm{dist}^2(x^{(t+1)},x^*)) + \frac{\zeta \alpha}{2}\|\mathrm{grad}f(x^{(t)})\|^2 
		\\ 
			 & \quad  + \frac{\zeta \alpha}{2}\left\|\sum_{i=1}^Np_i \varepsilon_i^{(t)}\right\|^2 + \zeta \alpha \left<\mathrm{grad}f(x^{(t)}),\sum_{i=1}^Np_i\varepsilon_i^{(t)}\right> \\ 
		&= \frac{1}{2\alpha}(\mathrm{dist}^2(x^{(t)},x^*)-\mathrm{dist}^2(x^{(t+1)},x^*)) + \frac{\zeta \alpha}{2}\|\mathrm{grad}f(x^{(t)})\|^2 + \mathrm{Au}_i^t,
		\end{aligned}		
	\end{equation}
	where $\mathrm{Au}_i^t =\frac{1}{\alpha}\left<\sum_{i=1}^Np_i\varepsilon_i^{(t)},\mathrm{Exp}_{x^{(t)}}^{-1}(x^*)\right>+ \frac{\zeta \alpha}{2}\left\|\sum_{i=1}^Np_i \varepsilon_i^{(t)}\right\|^2 + \zeta \alpha \left<\mathrm{grad}f(x^{(t)}),\sum_{i=1}^Np_i\varepsilon_i^{(t)}\right>$. 
	From Inequality~\eqref{th:conv:for DP-RSGD1:1}, we have 
	\begin{equation} \label{th:conv:for DP-RSGD3:2}
		\Delta^{t+1} - \Delta^t \le (\frac{{L}_g}{2}\alpha^2 -\alpha) \left\|\mathrm{grad}f(x^{(t)})\right\|^2 + \frac{{L}_g}{2}\alpha^2 \left\|\sum_{i=1}^Np_i\varepsilon_i^{(t)}\right\|^2 + ({L}_g\alpha^2-\alpha)\left<\sum_{i=1}^Np_i\varepsilon_i^{(t)},\mathrm{grad}f(x^{(t)})\right>.
	\end{equation}
	Multiplying Inequality~\eqref{th:conv:for DP-RSGD3:2} by $\zeta$ and adding it into~\eqref{th:conv:for DP-RSGD3:1}, we get 
	\begin{equation} \label{th:conv:for DP-RSGD3:3}
	 \zeta\Delta^{t+1} - (\zeta-1)\Delta^t \le \frac{1}{2\alpha}(\mathrm{dist}^2(x^{(t)},x^*)-\mathrm{dist}^2(x^{(t+1)},x^*)) + \frac{\zeta}{2} ({L}_g\alpha^2-\alpha)\|\mathrm{grad}f(x^{(t)})\|^2 + \tilde{\mathrm{Au}}_i^t,
	\end{equation} 
	where $\tilde{\mathrm{Au}}_i^t=\frac{1}{\alpha}\left<\sum_{i=1}^Np_i\varepsilon_i^{(t)},\mathrm{Exp}_{x^{(t)}}^{-1}(x^*)\right>+ \frac{\zeta}{2}({L}_g\alpha^2 + \alpha)\left\|\sum_{i=1}^Np_i \varepsilon_i^{(t)}\right\|^2 + \zeta {L}_g\alpha^2 \left<\mathrm{grad}f(x^{(t)}),\sum_{i=1}^Np_i\varepsilon_i^{(t)}\right>$.
	Since $\alpha \le 1/(2{L}_g)$, we have ${L}_g\alpha^2-\alpha\le0$, and therefore 
$
	 \zeta\Delta^{t+1} - (\zeta-1)\Delta^t \le \frac{1}{2\alpha}(\mathrm{dist}^2(x^{(t)},x^*)-\mathrm{dist}^2(x^{(t+1)},x^*)) + \tilde{\mathrm{Au}}_i^t.
$
	Summing this up over $t$ from $0$ to $T-1$, we have 
	\begin{equation} \label{eq:03}
	\begin{aligned} 
		\zeta \Delta^{T} + \sum_{t=1}^{T-1}\Delta^t &\le (\zeta-1)\Delta^0 + \frac{\mathrm{dist}^2(x^{(0)},x^*)}{2\alpha} + \sum_{t=0}^{T-1}\tilde{\mathrm{Au}}_i^t 
		 \le \frac{\zeta\mathrm{dist}^2(x^{(0)},x^*)}{2\alpha} + 
		\sum_{t=0}^{T-1}\tilde{\mathrm{Au}}_i^t,
	\end{aligned}
	\end{equation}
 where the last inequality follows from $\alpha\le 1/(2L_g)$ and $\Delta^0\le L_g \mathrm{dist}(x^{(0)},x^*)$.
	Taking expectation over $\varepsilon_i^{(t)}, \forall t$ for~\eqref{th:conv:for DP-RSGD3:2} and $\tilde{\mathrm{Au}}_i$ yields 
	\begin{align*}
		&\mathbb{E}[\Delta^{t}] \ge \mathbb{E}[\Delta^{t+1}] - \frac{1}{8{L}_g}\sum_{i=1}^Np_i^2\mathbb{E}[\|\varepsilon_i^{(t)}\|^2] \ge \mathbb{E}[\Delta^{t+1}] - \frac{1}{8{L}_g}\sum_{i=1}^Ndp_i^2\sigma_i^2, \\
		&\mathbb{E}[\tilde{\mathrm{Au}}_i^{t}]\le \frac{3\zeta}{8{L}_g}\sum_{i=1}^Np_i^2\mathbb{E}[\|\varepsilon_i^{(t)}\|^2] \le \frac{3\zeta}{8{L}_g}\sum_{i=1}^Ndp_i^2\sigma_i^2,
	\end{align*}
	since $\mathbb{E}[\varepsilon_i^{(t)}]=0$, $\mathbb{E}[\left<\varepsilon_i^{(t)}, \mathrm{Exp}_{x^{(t)}}^{-1}(x^*)\right>]=0$, $\mathbb{E}[\left<\varepsilon_i^{(t)},\mathrm{grad}f(x^{(t)})\right>]=0$, and $\mathbb{E}[\|\varepsilon_i^{(t)}\|^2]= d\sigma_i^2$. Therefore, using $\mathbb{E}[\Delta^t]\ge \mathbb{E}[\Delta^T] - \frac{(T-t)}{8{L}_g}\sum_{i=1}^Ndp_i^2\sigma_i^2$ for~\eqref{eq:03} gives
	$
		(\zeta+T-1)\mathbb{E}[\Delta^T] \le \frac{\zeta\mathrm{dist}^2(x^{(0)},x^*)}{2\alpha} + \frac{3\zeta T}{8{L}_g}\sum_{i=1}^Ndp_i^2\sigma_i^2 +  \frac{T(T-1)}{16{L}_g}\sum_{i=1}^Ndp_i^2\sigma_i^2,
	$
	which implies that
	$
		\mathbb{E}[f(x^{(T)})-f(x^*)] \le \frac{\zeta\mathrm{dist}^2(x^{(0)},x^*)}{2\alpha(\zeta + T - 1)} + \frac{T(6\zeta + T-1)}{16{L}_g(\zeta + T - 1)}\sum_{i=1}^Ndp_i^2\sigma_i^2.
	$
\end{proof}

\subsection{Proof of Theorem~\ref{th:conv:for DP-RSGD4}}

\begin{proof}
We have $\mathrm{Exp}_{x^{(t)}}^{-1}(x^{(t+1)})=\sum_{i\in \mathcal{S}_t}\bar{p}_i \mathrm{Exp}_{x^{(t)}}^{-1}(\tilde{x}_i^{(t+1)})=-\alpha\sum_{i\in \mathcal{S}_t}\bar{p}_i(\mathrm{grad}f_i(x^{(t)})+\varepsilon_i^{(t)})$, where $\bar{p}_i=N_i/(\sum_{i\in \mathcal{S}_t}N_i)$.
It follows from $L_g$-smoothness of $f$ that
	\begin{equation}		
	\begin{aligned} \label{th:conv:for DP-RSGD4:1}
	        & 		f(x^{(t+1)}) - f(x^{(t)}) \\ 
 		\le & \left<\mathrm{Exp}_{x^{(t)}}^{-1}(x^{(t+1)}),\mathrm{grad}f(x^{(t)})\right> + \frac{{L}_g}{2}\|\mathrm{Exp}_{x^{(t)}}^{-1}(x^{(t+1)})\|^2  \\ 
		= & -\alpha\left<\sum_{i\in \mathcal{S}_t}\bar{p}_i \mathrm{grad}f_i(x^{(t)}) + \sum_{i\in \mathcal{S}_t}\bar{p}_i\varepsilon_i^{(t)},\mathrm{grad}f(x^{(t)})\right> + \frac{{L}_g}{2}\alpha^2\|\sum_{i\in \mathcal{S}_t}\bar{p}_i \mathrm{grad}f_i(x^{(t)})+ \sum_{i\in \mathcal{S}_t}\bar{p}_i\varepsilon_i^{(t)}\|^2 \\ 
		 = & -\alpha \left<\sum_{i\in \mathcal{S}_t}\bar{p}_i \mathrm{grad}f_i(x^{(t)}),\mathrm{grad}f(x^{(t)})\right> + \frac{{L}_g}{2}\alpha^2\|\sum_{i\in \mathcal{S}_t}\bar{p}_i \mathrm{grad}f_i(x^{(t)})\|^2  + \frac{{L}_g}{2}\alpha^2 \|\sum_{i\in \mathcal{S}_t}\bar{p}_i\varepsilon_i^{(t)}\|^2 \\ 
		& \quad - \alpha\left<\sum_{i\in \mathcal{S}_t}\bar{p}_i\varepsilon_i^{(t)}, \mathrm{grad}f(x^{(t)})\right> + {L}_g\alpha^2 \left<\sum_{i\in \mathcal{S}_t}\bar{p}_i \mathrm{grad}f_i(x^{(t)}),\sum_{i\in \mathcal{S}_t}\bar{p}_i\varepsilon_i^{(t)}\right>.
	\end{aligned}
	\end{equation} 
Taking the expectation over the subsampling $\mathcal{S}_t$ and $\varepsilon_i^{(t)}$ for the loss function $f(x^{(t+1)})$ yields
\begin{equation}
	\begin{aligned} \label{th:conv:for DP-RSGD4:2}
		\mathbb{E}[f(x^{(t+1)})] \le f(x^{(t)}) - \alpha\|\mathrm{grad}f(x^{(t)})\|^2 + \frac{{L}_g}{2}\alpha^2 \mathbb{E}\left[\left\|\sum_{i\in \mathcal{S}_t}\bar{p}_i \mathrm{grad}f_i(x^{(t)})\right\|^2\right] + \frac{{L}_g}{2} \alpha^2\mathbb{E}\left[\left\|\sum_{i\in \mathcal{S}_t}\bar{p}_i\varepsilon_i^{(t)}\right\|^2\right],
	\end{aligned}
\end{equation}
since $\mathbb{E}[\varepsilon_i^{t}]=0$. By the assumption that $N_i=N_j$ for all $i,j\in[N]$, we have $\bar{p}_i={1}/{S}$. It follows that 
\begin{equation}
	\begin{aligned} \label{th:conv:for DP-RSGD4:3}
	&\quad \mathbb{E}[\|\sum_{i\in \mathcal{S}_t}\bar{p}_i \mathrm{grad}f_i(x^{(t)})\|^2]  \\ 
	&= \frac{1}{NS}\sum_{i=1}^N \|\mathrm{grad}f_i(x^{(t)})\|^2 + \frac{S-1}{NS(N-1)}\sum_{i=1}^N\sum_{j=1,j\not=i}^N\left<\mathrm{grad}f_i(x^{(t)}),\mathrm{grad}f_j(x^{(t)})\right> \\ 
	& = \left( \frac{1}{NS} - \frac{S-1}{NS(N-1)} \right)\sum_{i=1}^N\|\mathrm{grad}f_i(x^{(t)})\|^2 + \frac{S-1}{NS(N-1)}\left<\sum_{i=1}^N\mathrm{grad}f_i(x^{(t)}),\sum_{i=1}^N\mathrm{grad}f_i(x^{(t)})\right> \\ 
	&= \frac{N-S}{NS(N-1)}\sum_{i=1}^N\|\mathrm{grad}f_i(x^{(t)})\|^2 + \frac{N(S-1)}{S(N-1)}\|\mathrm{grad}f(x^{(t)})\|^2 .
\end{aligned}
\end{equation}
From Assumption~\ref{ass:4}, it follows that 
$
	\mathbb{E}[v_i]=\frac{1}{N}\sum_{i=1}^Nv_i\ge \frac{1}{N}\sum_{i=1}^N\|\mathrm{grad}f_i(x^{(t)})- \mathrm{grad}f(x^{(t)})\|^2.$
Plugging this inequality 
into Inequality~\eqref{th:conv:for DP-RSGD4:3} yields
\begin{equation}
	\begin{aligned} \label{th:conv:for DP-RSGD4:5}
	& \quad \mathbb{E}[\|\sum_{i\in \mathcal{S}_t}\bar{p}_i \mathrm{grad}f_i(x^{(t)})\|^2]\\ 
	 &= \frac{N-S}{NS(N-1)}\sum_{i=1}^N\|\mathrm{grad}f_i(x^{(t)}) - \mathrm{grad}f(x^{(t)}) + \mathrm{grad}f(x^{(t)})\| ^2 + \frac{N(S-1)}{S(N-1)}\|\mathrm{grad}f(x^{(t)})\|^2 \\ 
	&\le \frac{N-S}{S(N-1)}v + \|\mathrm{grad}f(x^{(t)})\|^2,
\end{aligned}
\end{equation}
where the last inequality follows from $\inner[]{ \sum_{i = 1}^N (\grad f_i(x^{(t)}) - \grad f(x^{(t)})) }{\grad f(x^{(t)})} =0$.
Plugging Inequality~\eqref{th:conv:for DP-RSGD4:5} into Inequality~\eqref{th:conv:for DP-RSGD4:2} and substracting $f(x^*)$ from the result yields 
\begin{equation} \label{th:conv:for DP-RSGD4:6}
		\mathbb{E}[f(x^{(t+1)})] - f(x^*) \le f(x^{(t)}) - f(x^*) + \alpha(\frac{{L}_g\alpha}{2}-1)\|\mathrm{grad}f(x^{(t)})\|^2 + \frac{{L}_g}{2}\alpha^2 \mathbb{E}[\|\sum_{i\in \mathcal{S}_t}\bar{p}_i\varepsilon_i^{(t)}\|^2] + \frac{{L}_g\alpha^2(N-S)}{2S(N-1)}v.
\end{equation}
Using the RPL condition, taking total expectation and applying Inequality~\eqref{th:conv:for DP-RSGD4:6} recursively, we have 
\begin{align*} 
	\mathbb{E}[f(x^{(t)})] - f(x^*) & \le  P^t(f(x^{(0)})-f(x^*)) +  (\sum_{t=0}^{T} P^t) \left(\frac{{L}_g}{2}\alpha^2 \mathbb{E}[\|\sum_{i\in \mathcal{S}_t}\bar{p}_i\varepsilon_i^{(t)}\|^2] + \frac{{L}_g\alpha^2(N-S)}{2S(N-1)}v\right) \\ 
	& \le  P^t(f(x^{(0)})-f(x^*)) + \frac{(1-P^t)}{\mu(2-{L}_g\alpha)}\left(\frac{{L}_g}{2}\alpha \mathbb{E}[\|\sum_{i\in \mathcal{S}_t}\bar{p}_i\varepsilon_i^{(t)}\|^2] + \frac{{L}_g\alpha(N-S)}{2S(N-1)}v\right).
\end{align*}
The final result follows from $\mathbb{E}[\|\varepsilon_i^{(t)}\|^2] = d\sigma_i^2$.
\end{proof}

\subsection{Proof of Theorem~\ref{th:conv:for DP-RSGD2}} \label{app:lem1}
The proof of Theorem~\ref{th:conv:for DP-RSGD2} depends on Lemma~\ref{lem:App2} and~\ref{lem:App3}. 
We list them here for completeness, and the details are referred to~\cite[Lemma~6]{ZS16},~\cite[Lemma~1,~4]{HMJG22}.
\begin{lemma}[{\cite[Lemma~6]{ZS16}~\cite[Lemma~1]{HMJG22}}] \label{lem:App2}
Let $x_1,x_2,x_3\in \mathcal{W}\subseteq \mathcal{M}$ where $\mathcal{W}$ is a totally normal neighbourhood of $\mathcal{M}$ with section curvature lower bounded by $\kappa_{\min}$, $l_1=\mathrm{dist}(x_1,x_2)$, $l_2=\mathrm{dist}(x_2,x_3)$, and $l_3=\mathrm{dist}(x_1,x_3)$. Let $A$ be the angle on 
$\mathrm{T}_{x_1} \mathcal{M}$ such that $\cos(A)=\frac{1}{l_{1}l_{3}} \left<  \mathrm{Exp}_{x_1}^{-1}(x_2),\mathrm{Exp}_{x_1}^{-1}(x_3)\right>$. Let $D_{\mathcal{W}}$ be the diameter of $\mathcal{W}$, i.e., $D_{\mathcal{W}}:= \max_{x,x'\in \mathcal{W}}\mathrm{dist}(x,x')$. Then we have $l_2^2 \le \zeta l_1^2 + l_3^2 - 2l_1l_3\cos(A)$, where $\zeta = \sqrt{|\kappa_{\min}|}D_{\mathcal{W}}/\tanh(\sqrt{|\kappa_{\min}|}D_{\mathcal{M}})$ if $\kappa_{\min}<0$ and $\zeta=1$ otherwise.  
\end{lemma}

\begin{lemma}[{\cite[Lemma~4]{HMJG22}}] \label{lem:App3}
Suppose $f_i(x)=\sum_{j=1}^{N_i}\frac{1}{N_i}f_{i,j}$ is geodesically $L_f$-Lipschitz continuous for any $j\in[N_i]$ ($\forall\;i\in[N]$). Consider a batch $\mathcal{B}$ of size $b$ and $x\in \mathcal{W}\subseteq \mathcal{M}$. Let $\eta=\frac{1}{b}\sum_{j\in \mathcal{B}}\mathrm{grad}f_{i,j}(x)+\varepsilon$ where $\varepsilon \sim \mathcal{N}_x(0,\sigma^2)$. Then it holds that $\mathbb{E}[\eta]=\mathrm{grad}f_i(x)$ and $\mathbb{E}[\|\eta\|^2]\le L_f^2 + d\sigma^2$, where the expectation is over randomness in both $\mathcal{B}$ and $\varepsilon$. 
\end{lemma}

\begin{lemma}\label{lem:1}
	Suppose that Problem~\eqref{prob} satisfies Assumption~\ref{ass:geodes-Lip},~\ref{ass:smooth},  and~\ref{ass:regularization}. Consider Algorithm~\ref{alg:PriRFed} with Algorithm~\ref{alg:DP-RSGD} and \textbf{Option 2}. Set $s_t=1$ and $K>1$. For $c_k$, $c_{k+1}$, $\beta$, $\alpha>0$, define
	$
		c_k = c_{k+1}(1+\alpha\beta) \hbox{ with $c_K = 1$ and } \delta_k = \alpha - \frac{c_{k+1}\alpha}{\beta},
	$
 where $\beta > 1$ is a constant.
	Then the iterate $x_{i,k}^{(t)}$ satisfies
	\begin{align}	\label{lem:1:0}
		\mathbb{E}[\|\mathrm{grad}f(x_{i,k}^{(t)})\|^2] \le \frac{R_k^{t}-R_{k+1}^{t}}{\delta_k} +  \frac{\alpha\beta(L_g+2c_{k+1}\zeta)}{2(\beta-c_{k+1})} (L_f^2 + d\sigma^2),
	\end{align}
	where the expectation is taken over $\mathcal{S}_t$, $\mathcal{B}_k$ and $\varepsilon_{i,k}^{(t)}$, $ R_k^{t} := \mathbb{E}[f(x_{i,k}^{(t)}) + c_k\|\mathrm{Exp}_{x^{(t)}}^{-1}(x_{i,k}^{(t)})\|^2] $ for $1\le t\le T$ and $0\le k\le K$, and $\sigma=\max_{i=1,2,\dots,N}\sigma_i$. 
\end{lemma}

\begin{proof}
The spirit of the proofs follows from~\cite[Lemma~2]{ZRS17}. 
	With $f$ being $L_g$-smooth we have
	\begin{equation}		
	\begin{aligned} \label{lem:1:1}
		\mathbb{E}[f(x_{i,k+1}^{(t)})] & \le \mathbb{E}\left[ f(x_{i,k}^{(t)})+\left< \mathrm{grad}f(x_{i,k}^{(t)}),\mathrm{Exp}_{x_{i,k}^{(t)}}^{-1}(x_{i,k+1}^{(t)}) \right> + \frac{{L}_g}{2}\| \mathrm{Exp}_{x_{i,k}^{(t)}}^{-1}(x_{i,k+1}^{(t)}) \|^2 \right] \\
		& \le \mathbb{E}\left[ f(x_{i,k}^{(t)}) - \alpha \| \mathrm{grad}f(x_{i,k}^{(t)}) \|^2 + \frac{{L}_g\alpha^2}{2}\|\eta_k^{(t)}\|^2 \right],
	\end{aligned}
	\end{equation}
since $\mathbb{E}[\mathrm{Exp}_{x_{i,k}^{(t)}}^{-1}(x_{i,k+1}^{(t)})]=\alpha \mathbb{E}[\eta_k^{(t)}]= \alpha \mathrm{grad}f_i(x_{i,k}^{(t)})$.
	Consider the following
	\begin{equation}		 \label{lem:1:2}
	\begin{aligned}
		& \mathbb{E}[\|\mathrm{Exp}_{x^{(t)}}^{-1}(x_{i,k+1}^{(t)})\|^2]
		\\
		& \le \mathbb{E}\left[ \| \mathrm{Exp}_{x^{(t)}}^{-1}(x_{i,k}^{(t)}) \|^2 + \zeta \| \mathrm{Exp}^{-1}_{x_{i,k}^{(t)}}(x_{i,k+1}^{(t)}) \|^2  - 2 \left< \mathrm{Exp}_{x^{(t)}_{i,k}}^{-1}(x^{(t)}_{i,k+1}),\mathrm{Exp}_{x^{(t)}_{i,k}}^{-1}(x^{(t)})\right> \right] \\
		& = \mathbb{E}\left[ \|\mathrm{Exp}_{x^{(t)}}^{-1}(x_{i,k}^{(t)})\|^2 + \zeta\alpha^2\|\eta_k^{(t)}\|^2 + 2\alpha\left< \mathrm{grad}f(x_{i,k}^{(t)}), \mathrm{Exp}_{x_{i,k}^{(t)}}^{-1}(x^{(t)}) \right> \right] \\
		& \le \mathbb{E}\left[ \| \mathrm{Exp}_{x^{(t)}}^{-1}(x_{i,k}^{(t)}) \|^2 + \zeta \alpha^2 \|\eta_k^{(t)}\|^2 \right] + 2\alpha \mathbb{E} \left[ \frac{1}{2\beta}\| \mathrm{grad}f(x_{i,k}^{(t)}) \|^2 + \frac{\beta}{2} \|\mathrm{Exp}_{x^{(t)}}^{-1}(x_{i,k}^{(t)})\|^2 \right]
	\end{aligned}
	\end{equation}
	where the first inequality follows from Lemma~\ref{lem:App2} and the second inequality follows from $2\left<u,v\right>\le \frac{1}{\beta}\|u\|^2 + \beta\|v\|^2$.
	Substituting~\eqref{lem:1:1} and~\eqref{lem:1:2} into $R_{k+1}^t$ yields
	\begin{equation}	\label{lem:1:3}
	\begin{aligned}
		R_{k+1}^t &\le  \mathbb{E}\left[ f(x_{i,k}^{(t)}) - \alpha \|\mathrm{grad}f(x_{i,k}^{(t)})\|^2 + \frac{{L}_g\alpha^2}{2}\|\eta_k^{(t)}\|^2 \right] + c_{k+1}\mathbb{E}\left[ \|\mathrm{Exp}_{x^{(t)}}^{-1}(x_{i,k}^{(t)})\|^2 + \zeta\alpha^2\|\eta_k^{(t)}\|^2 \right] \\
		& \quad + 2c_{k+1} \alpha \mathbb{E} \left[ \frac{1}{2\beta}\|\mathrm{grad}f(x_{i,k}^{(t)})\|^2 + \frac{\beta}{2} \| \mathrm{Exp}_{x^{(t)}}^{-1}(x_{i,k}^{(t)}) \|^2 \right] \\
		& = \mathbb{E}\left[ f(x_{i,k}^{(t)}) - (\alpha - \frac{c_{k+1}\alpha}{\beta})\|\mathrm{grad}f(x_{i,k}^{(t)})\|^2 \right] + \left( \frac{{L}_g\alpha^2}{2}+c_{k+1}\zeta\alpha^2 \right)\mathbb{E}[\|\eta_k^{(t)}\|^2] \\
		& \quad + (c_{k+1} + c_{k+1}\alpha\beta)\mathbb{E}[\|\mathrm{Exp}_{x^{(t)}}^{-1}(x_{i,k}^{(t)})\|^2].
	\end{aligned}
	\end{equation}
	By Lemma~\ref{lem:App3}, we have that 
	\begin{align} \label{lem:1:6}
		\mathbb{E}[\|\eta_k^{(t)}\|^2] \le {L}_f^2 + d\sigma_i^2  \le {L}_f^2 + d\sigma^2
	\end{align}
	with $\sigma=\max_{i\in[N]}{\sigma_i}$.
	Plugging~\eqref{lem:1:6} into~\eqref{lem:1:3} we have
	\begin{equation}
		\begin{aligned}
			R_{k+1}^t &\le \mathbb{E}\left[ f(x_{i,k}^{(t)})  - (\alpha - \frac{c_{k+1}\alpha}{\beta})\|\mathrm{grad}f(x_{i,k}^{(t)})\|^2\right] + c_{k+1}(1+\alpha\beta)\mathbb{E}[\|\mathrm{Exp}_{x^{(t)}}^{-1}(x_{i,k}^{(t)})\|^2] \\
			& \quad + \left( \frac{{L}_g\alpha^2}{2} + c_{k+1}\zeta\alpha^2 \right) ({L}_f^2 + d\sigma^2 ) \\
			& = R_k^t - \left( \alpha - \frac{c_{k+1}\alpha}{\beta} \right)\mathbb{E}[\|\mathrm{grad}f(x_{i,k}^{(t)})\|^2] + \left( \frac{{L}_g\alpha^2}{2} + c_{k+1}\zeta\alpha^2 \right) ({L}_f^2 + d\sigma^2),
		\end{aligned}
	\end{equation}
 which yields the desired result.
\end{proof}

Now, we are ready to prove Theorem~\ref{th:conv:for DP-RSGD2}.

\begin{proof}[Proof of Theorem~\ref{th:conv:for DP-RSGD2}]
	By Lemma~\ref{lem:1}, we have $c_k=c_{k+1}(1+\alpha\beta)=(1+\alpha\beta)^{K-k}$ due to $c_K=1$. Let $q={L}_f^2+d\sigma^2$, $q_k=\frac{\alpha\beta({L}_g+2c_{k+1}\zeta)}{2(\beta-c_{k+1})}$, and $\delta_{\min}=\min_{k=0,1,\dots,K-1}\delta_k=\alpha(1-\frac{c_1}{\beta})=\alpha(1-\frac{(1+\alpha\beta)^{K-1}}{\beta})$ since $c_k$ is decreasing. 
	Requiring $\delta_{\min}>0$ implies that $\beta>c_1=(1+\alpha\beta)^{K-1}$, i.e., $\alpha<(\beta^{\frac{1}{K-1}}-1)/\beta$ with $\beta>1$. Summing~\eqref{lem:1:0} over $k=0,1,\dots,K-1$ gives
	\begin{align} \label{th:conv:for DP-RSGD2:1}
	\frac{1}{K}\sum_{k=0}^{K-1}\mathbb{E}[\|\mathrm{grad}f(x_{i,k}^{(t)})\|^2] \le \frac{R_0^t - R_K^t}{K\delta_{\min}} + \frac{1}{K}\sum_{k=0}^{K-1}q_kq\le \mathbb{E}\left[ \frac{f(x^{(t)})-f(x_{i,K}^{(t)})}{K\delta_{\min}} \right] + q_0q
	\end{align}
	since $R_0^t=\mathbb{E}[f(x^{(t)})]$, $R_K^{t}=\mathbb{E}[f(x_{i,K}^{(t)})+c_K\|\mathrm{Exp}_{x^{(t)}}^{-1}(x_{i,K}^{(t)})\|^2]\ge \mathbb{E}[f(x_{i,K}^{t})]$, and $q_0\ge q_k$.
	Summing~\eqref{th:conv:for DP-RSGD2:1} over $t=0,1,\dots,T-1$ yields
	$
		\frac{1}{T}\sum_{t=0}^{T-1}\frac{1}{K}\sum_{k=0}^{K-1}\mathbb{E}[\| \mathrm{grad}f(x_{i,k}^{(t)}) \|^2]\le \frac{f(x^{(0)})-f(x^*)}{KT\delta_{\min}} + q_0q.
	$ The desired result follows from the \textbf{Option 2} of the output of Algorihtm~\ref{alg:PriRFed} and~\ref{alg:DP-RSGD}.
 
\end{proof}

\section{Proofs for Section~\ref{sec:conv:DP-RSVRG}} \label{appendix2}

\subsection{Proof of Theorem~\ref{th:conv:for DP-RSVRG}} \label{app:lem2}

\begin{lemma} \label{lem:2}
	Suppose Problem~\eqref{prob} satisfies Assumption~\ref{ass:geodes-Lip},~\ref{ass:smooth} and~\ref{ass:regularization}. Consider Algorithm~\ref{alg:PriRFed} with Algorithm~\ref{alg:DP-RSVRG} and \textbf{Option 2}. Set $K>1$, $m>1$, and $s_t=1$ for $t=0,1,\cdots,T-1$. For $c_j,c_{j+1},\beta,\alpha>0$, define
$
	c_j = c_{j+1}(1+\beta\alpha+2\zeta {L}_g^2\alpha^2) + {L}_g^3\alpha^2 \hbox{ and }\delta_j = \alpha - \frac{c_{j+1}\alpha}{\beta} - {L}_g\alpha^2 - 2c_{j+1}\zeta\alpha^2>0.
$
	Then the iterate $x_{k+1,j}^{(t)}$ satisfies
	\begin{align} \label{lem:2:0}
		\mathbb{E}\left[\|\mathrm{grad}f(x_{k+1,j}^{(t)})\|^2\right] \le \frac{R_{k+1,j}^{(t)}-R_{k+1,j+1}^{(t)}}{\delta_j} + \frac{\frac{1}{2}d{L}_g\alpha^2+c_{j+1}\zeta d\alpha^2}{\delta_j}\sigma^2,
	\end{align}
	where the expectation is taken over $S_t$, $l_j$ and $\varepsilon_{k+1,j}^{(t)}$, $R_{k+1,j}^{(t)}=\mathbb{E}[f(x_{k+1,j}^{(t)})+c_j\|\mathrm{Exp}_{\tilde{x}^{(t)}_{i,k}}^{-1}(x_{k+1,j}^{(t)})\|^2]$ for $k=0,1,\dots,K-1$ and $j=0,1,\dots,m$, and $\sigma=\max_{i=1,2,\dots,N}\sigma_i$.
\end{lemma}

\begin{proof}[Proof of Lemma~\ref{lem:2}]
	Let $\Delta^{(t)}_{k+1,j}=\mathrm{grad}f_{i,l_j}(x_{k+1,j}^{(t)}) - \Gamma_{\tilde{x}_{i,k}^{(t)}}^{x_{k+1,j}^{(t)}}(\mathrm{grad}f_{i,l_j}(\tilde{x}_{i,k}^{(t)}))$. Then we have
$
\mathbb{E}[\Delta^{(t)}_{k+1,j}]=\mathrm{grad}f(x_{k+1,j}^{(t)})-\Gamma_{\tilde{x}_{i,k}^{(t)}}^{x_{k+1,j}^{(t)}}(\mathrm{grad}f(\tilde{x}_{i,k}^{(t)})),
$
and thus $\eta_{k+1,j}^{(t)}=\Delta_{k+1,j}^{(t)}-\mathbb{E}[\Delta_{k+1,j}^{(t)}]+\mathrm{grad}f(x_{k+1,j}^{(t)})+\varepsilon_{k+1,j}^{(t)}$.
	It follows that
	\begin{equation}
	\begin{aligned} \label{lem:2:1}
\mathbb{E}\left[\|\eta_{k+1,j}^{(t)}\|^2\right]
& = \mathbb{E}\left[ \|\Delta_{k+1,j}^{(t)}-\mathbb{E}[\Delta_{k+1,j}^{(t)}]+\mathrm{grad}f(x_{k+1,j}^{(t)})+\varepsilon_{k+1,j}^{(t)} \|^2\right]  \\
& = \mathbb{E}\left[ \|\Delta_{k+1,j}^{(t)}-\mathbb{E}[\Delta_{k+1,j}^{(t)}]+\mathrm{grad}f(x_{k+1,j}^{(t)}) \|^2\right] + \mathbb{E}[\|\varepsilon_{k+1,j}^{(t)}\|^2] \\ 
& \quad + 2\mathbb{E}\left<\Delta_{k+1,j}^{(t)}-\mathbb{E}[\Delta_{k+1,j}^{(t)}]+\mathrm{grad}f(x_{k+1,j}^{(t)}), \varepsilon_{k+1,t}^{(t)}\right> \\
&\le \mathbb{E}\left[ \|\Delta_{k+1,j}^{(t)}-\mathbb{E}[\Delta_{k+1,j}^{(t)}]+\mathrm{grad}f(x_{k+1,j}^{(t)}) \|^2\right] + d\sigma^2   \\
		&\le 2\mathbb{E}[\|\Delta_{k+1,j}^{(t)}-\mathbb{E}[\Delta_{k+1,j}^{(t)}]\|^2] + 2\mathbb{E}[\|\mathrm{grad}f(x_{k+1,j}^{(t)})\|^2] + d\sigma^2 \\
		& \le 2\mathbb{E}[\|\Delta_{k+1,j}^{(t)}\|^2] + 2\mathbb{E}[\|\mathrm{grad}f(x_{k+1,j}^{(t)})\|^2] + d\sigma^2\\
		& \le 2L_g^2\mathbb{E}\left[ \| \mathrm{Exp}_{\tilde{x}_{i,k}^{(t)}}^{-1}(x_{k+1,j}^{(t)}) \|^2 \right] + 2\mathbb{E}[\|\mathrm{grad}f(x_{k+1,j}^{(t)})\|^2] + d\sigma^2,
	\end{aligned}
	\end{equation}
where the first inequality is due to $\mathbb{E}[\varepsilon_{k+1,j}^{(t)}]=0$ and $\mathbb{E}[\|\varepsilon_{k+1,j}^{(t)}\|^2]\le d\sigma^2$ with $\sigma=\max_{i=1,2,\dots,N}{\sigma_i}$, the second due to $\|u+v\|^2\le 2\|u\|^2+2\|v\|^2$, the third due to $\mathbb{E}[\|\xi-\mathbb{E}[\xi]\|^2]\le \mathbb{E}[\|\xi\|^2]-\|\mathbb{E}[\xi]\|^2\le \mathbb{E}[\|\xi\|^2]$, and the last due to ${L}_g$-smooth for $f_{i,j}$. 
Due to $L_g$-smoothness of $f$, we have 
\begin{equation}	
\begin{aligned} \label{lem:2:2} 
	\mathbb{E}[f(x_{k+1,j+1}^{(t)})] & \le \mathbb{E}[f(x_{k+1,j}^{(t)}) + \left<\mathrm{grad}f(x_{k+1,j}^{(t)}),\mathrm{Exp}_{x_{k+1,j}^{(t)}}^{-1}(x_{k+1,j+1}^{(t)})\right>+\frac{L_g}{2}\|\mathrm{Exp}_{x_{k+1,j}^{-1}}(x_{k+1,j+1}^{(t)})\|^2] \\ 
	& \le \mathbb{E}[f(x_{k+1,j}^{(t)})-\alpha \|\mathrm{grad}f(x_{k+1,j}^{(t)})\|^2 + \frac{L_g\alpha^2}{2}\|\eta_{k+1,j}^{(t)}\|^2]. 
\end{aligned}
\end{equation}
We also have
\begin{equation} \label{lem:2:3}
	\begin{aligned}
		&\mathbb{E}[\|\mathrm{Exp}_{\tilde{x}_{i,k}}^{-1}(x_{k+1,j+1}^{(t)})\|^2]  \le \mathbb{E}[\|\mathrm{Exp}_{\tilde{x}_{i,k}}^{-1}(x_{k+1,j}^{-1})\|^2 + \zeta \|\mathrm{Exp}_{x_{k+1,j}^{(t)}}^{-1}(x_{k+1,j+1}^{(t)})\|^2 \\
		& \quad  - 2 \inner{\mathrm{Exp}_{x_{k+1,j}^{(t)}}^{-1}(x_{k+1,j+1}^{(t)})}{\mathrm{Exp}_{x_{k+1,j}^{(t)}}^{-1}(\tilde{x}_{i,k})}] \\ 
		=& \mathbb{E}[\|\mathrm{Exp}_{\tilde{x}_{i,k}}^{-1}(x_{k+1,j}^{-1})\|^2 + \zeta\alpha^2\|\eta_{k+1,j}^{(t)}\|^2 + 2\alpha \inner{\mathrm{grad}f(x_{k+1,j}^{(t)})}{ \mathrm{Exp}_{x_{k+1,j}^{(t)}}^{-1}(\tilde{x}_{i,k})}] \\ 
		\le& \mathbb{E}[\|\mathrm{Exp}_{\tilde{x}_{i,k}}^{-1}(x_{k+1,j}^{-1})\|^2 + \zeta\alpha^2\|\eta_{k+1,j}^{(t)}\|^2] + 2\alpha \mathbb{E}[\frac{1}{2\beta}\|\mathrm{grad}f(x_{k+1,j}^{(t)})\|^2 + \frac{\beta}{2}\|\mathrm{Exp}_{\tilde{x}_{i,k}}^{-1}(x_{k+1,j}^{(t)})\|^2],  
	\end{aligned}
\end{equation}
where the first inequality is due to Lemma~\ref{lem:App2}, the second due to $2 \left<u,v\right>\le \frac{1}{\beta}\|u\|^2+{\beta}\|v\|^2$. Substituting Inequality~\eqref{lem:2:2} and~\eqref{lem:2:3} into $R_{k+1,j+1}^{(t)}$, we have 
\begin{equation} \label{lem:2:4}
	\begin{aligned}
		R_{k+1,j+1}^{(t)} & \le  \mathbb{E}[f(x_{k+1,j}^{(t)})-\alpha\|\mathrm{grad}f(x_{k+1,j}^{(t)})\|^2+\frac{L_g\alpha^2}{2}\|\eta_{k+1,j}^{(t)}\|^2 ]\\
		 & \quad + c_{j+1}\mathbb{E}[\|\mathrm{Exp}_{\tilde{x}_{i,k}}^{-1}(x_{k+1,j}^{-1})\|^2 + \zeta\alpha^2\|\eta_{k+1,j}^{(t)}\|^2] + 2\alpha \mathbb{E}[\frac{1}{2\beta}\|\mathrm{grad}f(x_{k+1,j}^{(t)})\|^2 + \frac{\beta}{2}\|\mathrm{Exp}_{\tilde{x}_{i,k}}^{-1}(x_{k+1,j}^{(t)})\|^2 \\ 
		 & = \mathbb{E}[ f(x_{k+1,j}^{(t)}) - \alpha(1 - \frac{c_{j+1}}{\beta})\|\mathrm{grad}f(x_{k+1,j}^{(t)})\|^2] + \alpha^2(\frac{L_g}{2}+c_{j+1}\zeta)\mathbb{E}[\|\eta_{k+1,j}^{(t)}\|^2] \\ 
		 & \quad + c_{j+1}(1+\alpha\beta)\mathbb{E}[\|\mathrm{Exp}_{\tilde{x}_{i,k}}^{-1}(x_{k+1,j}^{(t)})\|^2].
	\end{aligned}
\end{equation} 
Combining Inequality~\eqref{lem:2:1} and~\eqref{lem:2:4} completes the proof.
\end{proof}
Now, we are ready to prove Theorem~\ref{th:conv:for DP-RSVRG} using Lemma~\ref{lem:2}.
\begin{proof}[Proof of Theorem~\ref{th:conv:for DP-RSVRG}]
	Let $\delta_{\min}=\min_{j=0,1,\dots,m-1}\delta_j$. Summing~\eqref{lem:2:0} over $j=0,1,\dots,m-1$ yields
	\begin{equation} \label{th:conv:for DP-SVRG:1}
	\begin{aligned}
		\frac{1}{m}\sum_{j=0}^{m-1}\mathbb{E}[\|\mathrm{grad}f(x_{k+1,j}^{(t)})\|^2] & \le \frac{R_{k+1,0}^{(t)}-R_{k+1,m}^{(t)}}{m\delta_{\min}} + \sum_{j=0}^{m-1}\frac{(\frac{1}{2}{L}_g+c_{j+1}\zeta)d\alpha^2\sigma^2}{m\delta_{\min}}  \\
		& \le \frac{\mathbb{E}[f(x_{k+1,0}^{(t)})-f(x_{k+1,m}^{(t)})]}{m\delta_{\min}} + \sum_{j=0}^{m-1}\frac{(\frac{1}{2}{L}_g+c_{j+1}\zeta)d\alpha^2\sigma^2}{m\delta_{\min}} \\
		& \le \frac{\mathbb{E}[f(\tilde{x}_{i,k}^{(t)})-f(\tilde{x}_{i,k+1}^{(t)})]}{m\delta_{\min}} + \sum_{j=0}^{m-1}\frac{(\frac{1}{2}{L}_g+c_{j+1}\zeta)d\alpha^2\sigma^2}{m\delta_{\min}}
	\end{aligned}
	\end{equation}
	since $R_{k+1,0}^{(t)} = \mathbb{E}[f(x_{k+1,0}^{(t)})]=\mathbb{E}[f(\tilde{x}_{i,k}^{(t)})]$ and $R_{k+1,m}\le \mathbb{E}[f(x_{k+1,m}^{(t)})]=\mathbb{E}[f(\tilde{x}_{i,k+1}^{(t)})]$.
	Continuing to sum~\eqref{th:conv:for DP-SVRG:1} over $k=0,1,\dots,K-1$, we have
	\begin{equation} \label{th:conv:for DP-SVRG:2}
		\begin{aligned}
			\frac{1}{mK}\sum_{k=0}^{K-1}\mathbb{E}[\|\mathrm{grad}f(x_{k+1,j}^{(t)})\|^2] \le \frac{\mathbb{E}[f(x^{(t)})-f(x^{(t+1)})]}{mK\delta_{\min}} + \frac{(\frac{1}{2}{L}_g+c_{0}\zeta)d\alpha^2\sigma^2}{\delta_{\min}}
		\end{aligned}
	\end{equation}
	where $c_{j}\le c_0$, $c_0>0$, $c_m=0$.
	Finally summing~\eqref{th:conv:for DP-SVRG:2} over $t=0,1,\dots,T-1$ gives
	\begin{equation} \label{th:conv:for DP-SVRG:3}
		\begin{aligned}
			\frac{1}{mTK}\sum_{k=0}^{K-1}\mathbb{E}[\|\mathrm{grad}f(x_{k+1,j}^{(t)})\|^2] \le \frac{f(x^{(0)})-f^*}{mTK\delta_{\min}} + \frac{(\frac{1}{2}{L}_g+c_{0}\zeta)d\alpha^2\sigma^2}{\delta_{\min}}.
		\end{aligned}
	\end{equation}
	Hence, using \textbf{Option 2} of the output of Algorithm~\ref{alg:PriRFed} and~\ref{alg:DP-RSVRG} yields
	$
		\mathbb{E}[\| \mathrm{grad}f(\tilde{x}) \|^2] \le \frac{f(x^{(0)})-f^*}{mTK\delta_{\min}} + \frac{(\frac{1}{2}{L}_g+c_{0}\zeta)d\alpha^2\sigma^2}{\delta_{\min}}.
	$
	Choosing $\beta={L}_g\zeta^{1-a_2}/{N}^{a_1/2}$ and solving recurrence relation $c_j$ using $\alpha$, $m$ given by~\cite[Theorem~2]{ZRS17}, i.e., $\alpha=\mu/({L}_g{N}^{a_1}\zeta^{a_2})$, $m=\lfloor {N}^{3a_1/2}/(3\mu\zeta^{1-2a_2}) \rfloor$, one can get that $c_0$ and $\delta_{\min}$ are bounded by, respectively,
	$
		c_0 \le \frac{\mu {L}_g}{{N}^{a_1/2}\zeta}(e-1), \hbox{ and }\delta_{\min}\ge \frac{\nu}{{L}_g{N}^{a_1}\zeta^{a_2}}
	$
	with $a_1\in(0,1]$, $a_2\in[0,2]$, $\mu\in(0,1)$ and $\nu>0$ being four freely positive constants. It follows that
	\begin{align*}
		\mathbb{E}[\| \mathrm{grad}f(\tilde{x}) \|^2] & \le \frac{{L}_g{N}^{a_1}\zeta^{a_2}}{\nu mKT}[f(x^{(0)})-f^*] + \frac{{L}_g{N}^{a_1}\zeta^{a_2}}{\nu}\left(\frac{1}{2}{L}_g+\frac{\mu {L}_g}{{N}^{a_1/2}}(e-1)\right)\frac{\mu^2}{{L}_g^2{N}^{2a_1}\zeta^{2a_2}}d\sigma^2  \\
		& = \frac{{L}_g{N}^{a_1}\zeta^{a_2}}{\nu mKT}[f(x^{(0)})-f^*] + \frac{d\mu^2\sigma^2}{\nu N^{a_1}\zeta^{a_2}}\left( \frac{1}{2}  + \frac{\mu}{{N}^{a_1/2}}(e-1) \right).
	\end{align*}
	Finally, taking $a_1=2/3$, $a_2=1/2$, $\mu=1/10$ yields the desired result with $c \geq \frac{10 N}{3 m}$.
\end{proof}

\bibliographystyle{alpha}
\bibliography{paper}

\end{document}